%% file: DefectModel.tex
\documentclass[11pt,a4paper,twoside]{article}
\usepackage[utf8]{inputenc}
\usepackage{lmodern}
\usepackage{microtype}
\usepackage[pdfborder={0 0 0}]{hyperref}

\usepackage{amsmath,amsthm}     
\usepackage[T1]{fontenc}        

\usepackage{mathtools}
\usepackage[numbers]{natbib}    
\usepackage{graphicx}
\usepackage{amssymb}
\usepackage{color}
\usepackage{hyperref}
\usepackage{svg}
\usepackage{float}
\usepackage{caption}
\usepackage{pstricks}
\usepackage{pdfpages}
\usepackage[all]{xy}
\usepackage{fancyhdr}
\usepackage{enumerate}
\usepackage{tikz,pgfplots}
\usetikzlibrary{cd}
\usepackage{xr,xr-hyper}
\usepackage{paralist}
\usepackage{marginnote}

\usepackage[text={16.5cm, 24.2cm}, centering]{geometry}
\newtheorem{lemma}{Lemma}[section]
\newtheorem{corollary}[lemma]{Corollary}
\newtheorem{theorem}[lemma]{Theorem}
\newtheorem{proposition}[lemma]{Proposition}
\newtheorem{example}[lemma]{Example}
\newtheorem{remark}[lemma]{Remark}

\theoremstyle{definition}
\newtheorem{definition}[lemma]{Definition}

\numberwithin{equation}{section}

\renewcommand{\epsilon}{\varepsilon}
\renewcommand{\phi}{\varphi}
\renewcommand{\theta}{\vartheta}

\DeclareMathOperator*{\tensor}{\otimes}
\DeclareMathOperator{\Hol}{Hol}
\newcommand{\F}{\mathbb{F}}

\newcommand{\C}{\mathbb{C}}
\newcommand{\oo}{\otimes}

\newcommand{\id}{\mathrm{id}}
\newcommand{\End}{\mathrm{End}}

\newcommand{\HSpace}{\mathcal{N}}

\newcommand{\st}[0]{{\bf s}}
\newcommand{\ta}[0]{{\bf t}}

\newcommand{\inv}[0]{{-1}}

\allowdisplaybreaks
\begin{document}

\def\mytitle{Defects and excitations in the Kitaev model}
\author{Thomas Vo\ss}

%

\begin{center}
  {\huge\mytitle}

  \vspace{2em}

  {\large
   Thomas Vo\ss\footnote{{\tt thomas.voss@uni-hamburg.de}} 
 }

 Fachbereich Mathematik \\
  Universit\"at Hamburg \\
  Bundesstra\ss e 55, Hamburg, Germany\\[+2ex]

{May 1, 2022}

  \begin{abstract}
	  We construct a Kitaev model with defects using twists or 2-cocycles of semi-simple, finite-dimensional Hopf algebras as defect data. This data is derived by applying Tannaka duality to Turaev-Viro topological quantum field theories with defects. From this we also derive additional conditions for moving, fusing and braiding excitations in the Kitaev model with defects. We give a description of excitations in the Kitaev model and show that they satisfy conditions we derive from Turaev-Viro topological quantum field theories with defects. Assigning trivial defect data one obtains transparent defects and we show that they can be removed, yielding the Kitaev model without defects.
  \end{abstract}
\end{center}
\tableofcontents
\section{Introduction}

Kitaev's \emph{lattice model} or \emph{quantum double model} was introduced by Kitaev \cite{Ki} as a realistic physics model for a topological quantum computer. The models provide codes that are protected against errors by topological effects. For this reason, they are investigated extensively in topological quantum computing and in condensed matter physics. 
In \cite{BK12} a close relation between the Kitaev model and topological quantum field theories (TQFTs) of Turaev-Viro type \cite{TV, BW} was shown - both assign the same vector spaces to oriented surfaces.
These TQFTs are of mathematical interest as they produce invariants of two- and three-dimensional manifolds. It was discovered in \cite{KKR} that Turaev-Viro  invariants \cite{TV, BW} can be used to define error correcting quantum codes, providing an additional link to Kitaev models.

The goal of this article is to introduce \emph{topological defects} in Kitaev models. These defects are of interest in topological quantum computing and condensed matter physics, because they are expected to occur in practical realizations of these and related models, and this requires a theoretical understanding of their effects \cite{BD, KK}.
From the mathematical perspective, Kitaev models with defects are of interest as they would describe the two-dimensional parts of Turaev-Viro-TQFTs with defects. A TQFT is by definition a symmetric monoidal functor $\mathcal Z: \text{Cob}_n\to \mathcal C$ from a cobordism category $\text{Cob}_{n}$ into a symmetric monoidal category $\mathcal C$, e.g. $\mathcal{C}=\mathrm{Vect}$.
In a defect TQFT the cobordism category $\text{Cob}_{n}$ is replaced by a modified cobordism category in which the $(n-1)$-manifolds are equipped with distinguished submanifolds that are assigned higher categorical data. These submanifolds can intersect with the incoming and outgoing boundaries of a cobordism and a Kitaev model with defects thus can contribute to the understanding of defect TQFTs.
%
\\\\
\textbf{Kitaev lattice models and TQFTs}\\
	\\
	 Kitaev's original lattice model  from~\cite{Ki} was based on the algebraic data of a group algebra of a finite group. 
	This model was then then generalized to finite-dimensional semisimple Hopf-$*$ algebras $H$ in~\cite{BMCA}, and the relation between the model for $H$ and its dual $H^{*}$ was clarified in~\cite{BCKA}.  	Further generalizations of Kitaev models were developed in~\cite{Ch} to a model based on unitary quantum groupoids and more recently in~\cite{KMM}, to a model based on crossed modules of semisimple finite-dimensional Hopf algebras.
	\\
	The ingredients of the Kitaev model from \cite{BMCA} are a finite-dimensional semisimple Hopf-$*$ algebra $H$ and a \emph{ribbon graph} $\Gamma$ or, equivalently, a \emph{embedded graph} on a  surface $\Sigma$. 
	From this data the model constructs an \emph{extended Hilbert space} $\HSpace$ by assigning a copy of the Hopf algebra $H$ to every edge $e$ of $\Gamma$.\\
	\\
	Every pair of a face and an adjacent vertex of $\Gamma$, usually called \emph{site}, defines an action of $H$ on $\HSpace$ in terms of \emph{vertex operators} and an action of $H^{*}$ on $\HSpace$ in terms of \emph{face operators}. 
	Together, the face and vertex operators define an action of the Drinfel'd double or quantum double $D(H)$ of $H$ on $\HSpace$. This action is local, i.e. it only affects the copies of $H$ assigned to edges in that face or incident at that vertex, and the actions for sites with distinct pairs of faces and vertices commute. 
	For this reason, the model is sometimes referred to as \emph{quantum double model}. \\
	\\
	The invariants of the $D(H)$-actions associated with all sites in $\Gamma$ form the \emph{protected space} $\mathcal{L}$. It was shown in \cite{Ki, BMCA} that it is a topological invariant of the surface: it depends only on the surface, but not on the choice of the ribbon graph in the definition of the model. 
	\emph{Ribbon operators} are operators on the Hilbert space $\HSpace$ associated to ribbons in the graph. When applied to the protected space, they generate a pair of non-trivial $D(H)$-modules at their endpoints, the \emph{excitations}.
	Quantum computation can be performed by moving and braiding these excitations~\cite{Ki}.\\
	\\
	The Turaev-Viro TQFT~\cite{TV,BW} based on the fusion category $H\text{-Mod}$ also assigns a topological invariant $Z_{TV}(\Sigma)$ to the surface $\Sigma$. It was shown in~\cite{BK12} that this topological invariant coincides with the protected space $\mathcal{L}$ of the Kitaev model.
	
	Excitations in the Kitaev model were also investigated in \cite{BK12} and interpreted as a special type of boundary structure in the model.  These insights were also used in~\cite{BA,Kr}   to relate Kitaev models to another class of models from topological quantum computing, the  Levin-Wen models \cite{LW}. 
\\
	\\
	\textbf{Kitaev models and TQFTs with defects}\\
	\\
	As explained above, defects and also boundaries are of interest both from the viewpoint of Kitaev models and from the viewpoint of TQFTs.
	In a Kitaev model based on group algebras of finite groups defects and boundaries have first been described in~\cite{BD}. 
	In~\cite{KK} Kitaev and Kong determined the categorical data for defects in a Levin-Wen model. 
	Bulk regions, i.e. the regions separated by defects and bordered by boundaries, are labeled with unitary tensor categories $\mathcal{C},\mathcal{D}$ and defects are labeled with a $\mathcal{C}\mathrm{-}\mathcal{D}$-bimodule category $\mathcal{N}$.
	\\
	\\
	A Kitaev model with general defects was constructed very recently by Koppen~\cite{K} using the Hopf algebraic counterpart of this data.
	Here bulk regions are labeled with finite-dimensional semisimple Hopf algebras $H_{1},H_{2}$ and defects are labeled with finite-dimensional, semisimple $(H_{1},H_{2})$-bicomodule algebras.
	\\
	\\
	However, in this article we are interested in a more specific type of defects in a Kitaev model  that behave well with respect to excitations and can be viewed as the two-dimensional part of a Turaev-Viro TQFT with \emph{topological defects and boundaries}. Three-dimensional TQFTs of Turaev-Viro and Reshetikhin-Turaev type with \emph{topological defects and boundaries}
were investigated by Fuchs, Schweigert and Valentino in \cite{FSV}. 	They use physics considerations on the braiding, fusion and transport of excitations to determine the categorical data for these defects. More specifically, the categorical data for Turaev-Viro TQFTs with topological boundaries and defects from~\cite{FSV} is
	\begin{compactenum}[$\quad\bullet$]
		\item the center $\mathcal{Z}(\mathcal{A})$ of a fusion category $\mathcal{A}$ for every bulk region,
		\item a fusion category $\mathcal{W}_{a}$ (resp. $\mathcal{W}_{d}$) for every topological boundary $a$ (resp. topological surface defect $d$),
		\item a braided equivalence $\widetilde{F}_{\to a}:\mathcal{Z}(\mathcal{A}) \to \mathcal{Z}(\mathcal{W}_{a})$ for every pair of a topological boundary labeled with $\mathcal{W}_{a}$ and an adjacent bulk labeled with $\mathcal{Z(A)}$,
		\item a braided equivalence $\widetilde{F}_{\to d \leftarrow}: \mathcal{Z}(\mathcal{A}_1)\boxtimes \mathcal{Z}(\mathcal{A}_2)^{rev} \to \mathcal{Z}(\mathcal{W}_{d})$ for every topological surface defect labeled with $\mathcal{W}_{d}$ separating bulk regions labeled with $\mathcal{Z}(\mathcal{A}_{1})$ and $\mathcal{Z}(\mathcal{A}_2)$, composed of  braided monoidal functors $\widetilde{F}_{\to d}:\mathcal{Z}\left( \mathcal{A}_{1} \right) \to \mathcal{Z}(\mathcal{W}_{d})$ and $\widetilde{F}_{d \leftarrow}:\mathcal{Z}\left( \mathcal{A}_{2} \right)^{rev} \to \mathcal{Z}(\mathcal{W}_{d})$.
\end{compactenum}
There are three types of excitation in a Turaev-Viro TQFT with topological boundaries and defects:
\begin{compactenum}[$\quad\bullet$]
	\item bulk excitations are objects of $\mathcal{Z(A)}$,
	\item boundary excitations are objects of $\mathcal{W}_{a}$,
	\item defect excitations are objects of $\mathcal{W}_{d}$.
\end{compactenum}
These excitations can be moved, fused and braided and these procedures obey the following conditions:
\begin{compactenum}[$\quad\bullet$]
	\item moving an excitation $M$ from a bulk region into an adjacent boundary or defect turns $M$ into $\widetilde{F}_{\to a}(M)$, resp. $\widetilde{F}_{\to d}(M)$ or $\widetilde{F}_{d\leftarrow}(M)$.
	\item fusion of excitations in a bulk region, boundary or defect uses the tensor product of $\mathcal{Z(A)},\mathcal{W}_{a}$ or $\mathcal{W}_{d}$, respectively.
	\item braiding two bulk excitations uses the braiding of $\mathcal{Z(A)}$,
	\item braiding a bulk excitation with a boundary excitation or defect excitation uses the half-braiding defined by the functors $\widetilde{F}_{\to a}, \widetilde{F}_{\to d}$ or $\widetilde{F}_{d\leftarrow}$.
\end{compactenum}
	The direct and close relation between Turaev-Viro TQFTs and Kitaev models from \cite{BK12} shows that there must be an associated Kitaev model with topological defects and boundaries that satisfies analogous conditions. However, such a model has not been constructed so far, and  we construct it in this article.
	\\
	\\
	\textbf{Summary of the results}
	\\
	\\
	In this article we consider Kitaev lattice models based on finite-dimensional, semisimple Hopf algebras. We extend the relation between Kitaev models and Turaev-Viro TQFTs by generalizing the Kitaev model to a model with \emph{topological defects and boundary conditions}.  
These defects are less general than the defects considered in~\cite{KK,K}, but allow for more structure. 
In~\cite{K} bicomodule algebras are used for defects, whereas our defects are labeled with twisted Hopf algebras. 
This allows us to move and fuse excitations inside a topological defect or boundary and we can braid bulk excitations with defect or boundary excitations.

To obtain suitable conditions for our model, we use the link between Turaev-Viro TQFTs and the Kitaev model without defects \cite{BK12}.
We first illustrate how the categorical data for a Turaev-Viro TQFT relates to the Hopf algebraic data in the Kitaev model.
We then use the categorical data and the conditions for a Turaev-Viro TQFT with topological boundaries and defects from~\cite{FSV} to derive conditions for a Kitaev model with topological defects and boundaries.

The algebraic input for a Turaev-Viro TQFT without defects is the center $\mathcal{Z}(\mathcal{A})$ of a fusion category $\mathcal{A}$, whereas the algebraic input for the Kitaev model without defects is the Drinfel'd double $D(H)$ of a finite-dimensional, semisimple Hopf algebra $H$. 

These two are related by \emph{Tannaka-Krein duality}~\cite{U,S}: 
\begin{compactenum}[$\quad\bullet$]
	\item the category $H\mathrm{-Mod}$ is a fusion category and its center is equivalent to $D(H)\mathrm{-Mod}$ as braided monoidal category, and
	\item a strict fiber functor $\mathcal{A}\to \mathrm{Vect}_{\C}$ defines a finite-dimensional, semisimple Hopf algebra $H$ such that $\mathcal{A}\cong H\mathrm{-Mod}$.
\end{compactenum}

In the Kitaev model, the Hopf algebra $D(H)$ not only enters as input data.
The prominent symmetries of the extended space $\HSpace$ are local $D(H)$-module structures and to every object $M$ of $D(H)\mathrm{-Mod}$ and site $s$ in $\Gamma$ we can assign a subspace $\HSpace(s,M)$ - the excitation of type $M$ at $s$.
The fusion of excitations in the Kitaev model uses the tensor product of $D(H)\mathrm{-Mod}$, i.e. the coalgebra structure of $D(H)$  and the braiding of excitations uses the $R$-matrix of $D(H)$. 

This comprehensive view of the relation between the Kitaev model based on $H$ and the braided monoidal category $D(H)\mathrm{-Mod}$ serves as a starting point for the translation of the conditions for a Turaev-Viro TQFT with topological defects and boundaries into conditions for a Kitaev model with topological defects and boundaries. 

In Section~\ref{section:TranslationKitaev} we apply Tannaka-Krein duality to the categorical data for Turaev-Viro TQFTS with topological defects and boundaries from~\cite{FSV} and obtain the following result:\\
\\
\textbf{Theorem 1:} \emph{ The algebraic data for Kitaev models with topological boundaries and defects is
\begin{compactenum}[$\quad\bullet$]
		\item a Drinfel'd Double $D(H_{b})$ of a complex finite-dimensional, semisimple Hopf algebra $H_{b}$ for every bulk region $b$,
		\item a twist $F_{c}$ of $D(H_{b})$ for every boundary line $c$ adjacent to the bulk region $b$,
		\item a twist $F_{d}$ of $D(H_{b_{1}}) \oo D(H_{b_{2}})$ for every defect line $d$ separating two bulk regions $b_{1}$ and $b_{2}$.
	\end{compactenum}
	\emph{This data is Tannaka dual to the categorical data for Turaev-Viro TQFTS with topological defects and boundaries.} 
}
We then use this data to define a Kitaev model with topological defects and boundaries. 
As a second ingredient our model uses a ribbon graph $\Gamma$ with additional structure for defects and boundaries. 

From this data we construct an extended space $\HSpace$ together with the following local operators 
\begin{compactenum}[$\quad\bullet$]
	\item For every bulk region $b$ and every boundary adjacent to $b$: Local operators which define a $D(H_{b})$-module structure on the extended space $\HSpace$.
	\item For every defect $d$ separating two boundary regions $b_{1}$ and $b_{2}$: Local operators which define a $D(H_{b_{1}}) \oo D(H_{b_{2}})$-module structure on the extended space $\HSpace$.
\end{compactenum}
Excitations in this model then are either $D(H_{b})$-modules or $D(H_{b_{1}}) \oo D(H_{b_{2}}) $-modules.

We implement the movement, fusion (i.e. the tensor product) and the braiding of excitations in our model by a \emph{transport operator} $T_{\rho}:\HSpace\to \HSpace$. 
It depends on a path $\rho:s_{1}\to s_{2}$ in the thickened graph $D(\Gamma)$ and moves excitations from the site $s_{1}$ to $s_{2}$ and fulfills the following conditions we derived from the conditions for a Turaev-Viro TQFT with topological boundaries and defects in~\cite{FSV}:\\
\\
\textbf{Theorem 2:}\emph{ The map $T_{\rho}$ fuses excitations $M_{1}$ at $s_{1}$ and $M_{2}$ at $s_{2}$ into $M_{2} \oo M_{1}$ at $s_{2}$, where $ \oo $ is the tensor product
		\begin{compactenum}[$\quad\bullet$]
			\item of $D(H_{b})\mathrm{-Mod}$, if $s_{2}$ lies in a bulk region,
			\item of $D(H_{b})_{F_{c}}\mathrm{-Mod}$, if $s_{2}$ lies in a boundary line,
			\item of $\left( D(H_{b_{1}}) \oo D(H_{b_{2}}) \right)_{F_{d}}\mathrm{-Mod}$, if $s_{2}$ lies in a defect line.
		\end{compactenum}
		There is a path $\rho':s_{1}\to s_{2}$ canonically related to $\rho$ such that $T_{\rho'}$ fuses $M_{1}$ and $M_{2}$ into $M_{1} \oo M_{2}$. $T_{\rho}$ and $T_{\rho'}$ only differ up to a braiding
		\begin{compactenum}[$\quad\bullet$]
			\item of $D(H_{b})\mathrm{-Mod}$, if $s_{2}$ lies in a bulk region,
			\item of $D(H_{b})_{F_{c}}\mathrm{-Mod}$, if $s_{2}$ lies in a boundary line,
			\item of $\left( D(H_{b_{1}}) \oo D(H_{b_{2}}) \right)_{F_{d}}\mathrm{-Mod}$, if $s_{2}$ lies in a defect line.
		\end{compactenum}
	}
These two theorems show that our model can indeed be considered a Kitaev model counterpart to a Turaev-Viro TQFT with topological boundaries and defects.

We conclude this article by an examination of \emph{transparent defects}. 
These are trivial defects between two bulk regions labeled with the same Hopf algebra $H$.
We show that these defects can be removed and doing so in a model with only transparent defects, one obtains the Kitaev model without defects and boundaries from \cite{BMCA}.
\\
\\
\textbf{Structure of the article}
\\
\\
In Section~\ref{section:HopfAlgebras} we introduce the required background on Hopf algebras and tensor categories. 
Section~\ref{sec:ribbon} contains the relevant information on ribbon graphs.
In section~\ref{section:KitaevModel} we present the Kitaev model based on a finite-dimensional semisimple Hopf algebra from~\cite{BMCA}. 

In Section~\ref{section:FSVSummary} we summarize the categorical data for a Turaev-Viro-TQFT with topological defects and boundaries from~\cite{FSV} and use Tannaka-Krein reconstruction to obtain corresponding Hopf algebraic data (Theorem~1). 
\\
In Section~\ref{section:TranslationKitaev} we then determine how this Hopf algebraic data should appear in a Kitaev model with topological defects and boundaries. 
We translate the conditions for moving, fusion and braiding of excitations from~\cite{FSV} into conditions on a Kitaev model with topological defects and boundaries.
\\
Section~\ref{sec:defect model} constructs a Kitaev model with topological boundaries and defects based on the data we determined in section~\ref{section:FSVSummary}. 
We define an extended space, face and vertex operators, holonomies and a transport operator for moving excitations. 
We then show that the conditions we derived in~\ref{section:TranslationKitaev} are satisfied by our model (Theorem~2).
We conclude this section by relating the Kitaev model without defects from~\cite{BMCA} to a Kitaev model with transparent defects.
\\
\\
\section{Hopf algebras, Heisenberg doubles, modules}
\label{section:HopfAlgebras}
In this section we summarize background on Hopf algebras. Most results we present can be found in standard textbooks such as~\cite{EGNO,Ka,Ma,Mo,R} and otherwise specific citations are given. 

\subsection{Hopf algebras}

In the following we focus on finite-dimensional semisimple Hopf algebras over $\mathbb C$. 
Throughout the article we use Sweedler notation without summation signs  and write $\Delta(h)=h_{(1)}\otimes h_{(2)}$ for for the comultiplication. 
For a Hopf algebra $H$ we denote by  $H^{op}$, $H^{cop}$  and $H^{op,cop}$  the Hopf algebras with the opposite multiplication, comultiplication and the opposite multiplication and comultiplication. 

We denote by $\langle\;,\;\rangle: H^*\oo H\to \C$, $\langle \alpha, h\rangle=\alpha(h)$ the pairing between $H$ and $H^*$ and use the same symbol for the induced pairing $\langle\;,\;\rangle: H^{*\otimes n}\oo H^{\otimes n}\to\C$. Throughout the article, we use Roman letters for elements of $H$ and Greek letters for elements of $H^*$. 

For any finite-dimensional $H$-left module $M$ with the $h\in H$ acting on $m\in M$ denoted by $h \vartriangleright m$, the dual vector space $M^{*}$ is an $H$-right module and an $H$-left module.  The action of $h\in H$ on $\alpha\in M^{*}$ are given by
\begin{align*}
	\alpha  \vartriangleleft h = \alpha\left( h  \vartriangleright (\cdot)  \right), \qquad h  \vartriangleright \alpha = \alpha\left( S(h)  \vartriangleright (\cdot) \right)
\end{align*}

By the Artin-Wedderburn theorem, see for instance \cite[Chapter II]{Kn}, any finite-dimensional semisimple complex Hopf algebra  can be decomposed as a direct sum of a set of representatives of irreducible $H$-modules tensored with their duals. 

\begin{theorem}[Artin-Wedderburn]
	Let $H$ a semisimple finite-dimensional Hopf algebra and let $\mathrm{Irr}(H)$ be a system of representatives of irreducible $H$-modules. Then there is an isomorphism of bimodules
	\begin{align}
		H \cong \bigoplus_{s \in \mathrm{Irr}(H)} s \oo s^{*},
		\label{eq:ArtinWedderburn}
	\end{align}
	where $H$ is equipped with the bimodule structure by left and right multiplication and $s \oo s^{*}$  with the left action of $H$ on $s$ and the right action on $s^{*}$ dual to the left action on $s$.
	\label{theorem:ArtinWedderburn}
\end{theorem}

Recall that by Larson-Radford's theorem a  finite-dimensional Hopf algebra $H$ over $\C$  is semisimple if and only if
$H^*$ is semisimple and if and only if  its antipode is involutive.
Recall also that by Maschke's theorem this is equivalent to the existence of a normalized Haar integral, and that this normalized Haar integral is unique.

\begin{theorem} For every finite-dimensional  semisimple Hopf algebra over $\C$, there is a unique element $\lambda\in H$, the \emph{Haar integral} of $H$, with
	\begin{align}
		\label{eq:HaarIntegral}
	\varepsilon(\lambda)=1,\qquad \lambda\cdot h=h\cdot \lambda=\varepsilon(h)\lambda\quad \forall h\in H.
	\end{align}
\end{theorem}

For a given Hopf algebra $H$, we also consider the Drinfel'd double $D(H)$, its dual  and  the Hopf algebras obtained by twisting $H$ and $ D(H)$, see for instance the textbooks \cite{Ma, EGNO, Ka, Mo}.

\begin{definition}
	\label{def:Drinfel'dDouble}
	The Drinfel'd double of a finite-dimensional Hopf algebra $H$ is the quasitriangular Hopf algebra $D(H)=H^{*} \oo H$  with multiplication
\begin{align}
	\left( \alpha \oo h \right)\cdot \left( \beta \oo k \right) := \langle \beta_{(1)} \oo \beta_{(3)} , S^{-1}(h_{(3)}) \oo h_{(1)} \rangle \alpha \beta_{(2)} \oo h_{(2)} k
	\label{eq:Drinfel'dMultiplication}
\end{align}
and comultiplication
\begin{align}
	\Delta\left( \alpha \oo h \right) = \alpha_{(2)} \oo h_{(1)} \oo \alpha_{(1)} \oo h_{(2)}.
\end{align}
Its unit $1_{D(H)}$, counit $\varepsilon:D(H)\to \F$ and antipode $S:D(H) \to D(H)$ are 
\begin{align*}
	1_{D(H)} &= \varepsilon \oo 1_{H},\qquad
	\varepsilon\left( \alpha \oo h \right) = \alpha(1)\varepsilon(h),
\\
	S(\alpha \oo h) &= \langle \alpha_{(1)} \oo  \alpha_{(3)} , h_{(3)} \oo S^{-1}(h_{(1)}) \rangle S^{-1}(\alpha_{(2)})  \oo  S(h_{(2)}).
\end{align*}
Its $R$-Matrix is the element $R\in D(H) \oo D(H)$ given by
\begin{align}\label{eq:Rmatrix}
	R = \sum_{i=1}^{n} \varepsilon \oo a_{i} \oo \alpha_{i}  \oo 1_{H}=: \varepsilon\oo x\oo X\oo 1_H,
\end{align}
where $\left( a_{1},\dots,a_{n} \right)$ is a basis of $H$ with dual basis $\left( \alpha_{1},\dots,\alpha_{n} \right)$ and the  right-hand side is symbolic notation. 
\end{definition}
By the formula of the antipode of $D(H)$ it is easy to see that the antipode of $D(H)$ is involutive, if and only if the antipode of $H$ is involutive. For $H$ finite-dimensional over $\C$ Larson-Radford's theorem then implies that $D(H)$ is semisimple, if and only if $H$ is semisimple.

When using multiple copies of $x \oo X$ in the same equation we distinguish them by using different upper indices $x \oo X=  x^{1} \oo X^{1} = x^{2} \oo X^{2}=\cdots $.

By using Hopf algebra duality and some direct computations, one finds that the dual of the Drinfel'd double is given as follows. 

\begin{remark}
	\label{lemma:DualDrinfel'd}
	The dual of the Drinfel'd double of a finite-dimensional Hopf algebra $H$ is the vector space $D(H)^{*}= H \oo H^{*}$ with multiplication and comultiplication
	\begin{align}
		&\left( h \oo \alpha \right) \left( k \oo \beta \right) = kh \oo \alpha\beta
		\label{eq:Drinfel'dDualMultiplication}\\
	&\Delta_{D(H)^{*}} (h \oo \alpha) = \sum_{i,j=1}^n h_{(1)} \oo \alpha_{i} \alpha_{(1)} \alpha_{j} \oo S(a_{j})h_{(2)}a_{i} \oo \alpha_{(2)}
	\label{eq:Drinfel'dDualComultiplication}\\
	&\qquad\qquad \qquad =: h_{(1)}\oo X^1\alpha_{(1)} X^2\oo S(x^2)h_{(2)} x^1\oo \alpha_{(2)}.\nonumber
\end{align}
Its unit $1_{D(H)^{*}}$, counit $\varepsilon_{D(H)^{*}}$ and antipode $S_{D(H)^{*}}$ are
\begin{align}
	1_{D(H)^{*}}= 1_{H} \oo \varepsilon_{h},\quad \varepsilon_{D(H)^{*}}(h \oo \alpha) = \varepsilon(h)\alpha(1)
	\\
	S_{D(H)^{*}}(h \oo \alpha) = x^{1}S^{-1}(h)x^{(2)} \oo S^{-1}(X^{2})S(\alpha)X^{1}
\end{align}
\end{remark}

The dual of the Drinfel'd double is a special case of a twisted Hopf algebra, namely a twisting of  the Hopf algebra $H\oo H^*$ with the opposite of the universal $R$-matrix in \eqref{eq:Rmatrix}. In the following, we consider more general twistings with unitary cocycles.

\begin{definition}
	\label{def:Twist}
	Let $B$ be a bialgebra.
	A twist (or unitary 2-cocycle) for $B$ is an invertible element $F \in B \otimes B$ such that:
	\begin{align}
		(\varepsilon \otimes \id_{B}) (F) = 1_{B} = (\id_{B} \otimes \varepsilon ) (F)
		\label{eq:TwistEpsilon}\\
		F_{12} \cdot (\Delta \otimes \id_{B}) (F) = F_{23} \cdot (\id_{B} \otimes \Delta) (F)
		\label{eq:TwistDelta}
	\end{align}
	We write $F = F^{(1)} \otimes F^{(2)}$ in Sweedler notation, and $F_{12} := F^{(1)} \otimes F^{(2)} \otimes 1_{B}$, $F_{23} := 1_{B} \otimes F^{(1)} \otimes F^{(2)}$.
\end{definition}
When using multiple copies of the same element $F\in B \oo B$ in the same equation we distinguish them by using different upper indices $F=F^{(1)} \oo F^{(2)}=F^{(3)} \oo F^{(4)}=\cdots$. For instance, Equation~\eqref{eq:TwistDelta} reads in Sweedler notation:
\begin{align*}
	F^{(1)}F^{(3)}_{(1)} \oo F^{(2)}F^{(3)}_{(2)} \oo F^{(4)} = F^{(3)} \oo F^{(1)}F^{(4)}_{(1)} \oo F^{(2)}F^{(4)}_{(2)}
\end{align*}
If $F$ is invertible we similarly write $F^{-1}=F^{(-1)} \oo F^{(-2)}=F^{(-3)} \oo F^{(-4)}=\cdots$.

\begin{example}
	\label{example:ProjectedTwist}
	Let $H$ and $K$ be bialgebras.
	\begin{itemize}
		\item Any universal R-matrix $R\in H \oo H$ also is a twist for $H$.
		\item If $H\subseteq K$ is a sub-bialgebra and $F \in H \oo H$ is a twist for $H$, then $F$ also is a twist for $K$.
		\item Let $F\in H \oo K \oo H \oo K$ a twist for the bialgebra $H \oo K$. Then the projections of $F$ onto $H \oo H$ and $K \oo K$
	\begin{align}
		F_{H} &= \left(\id_{H} \oo \varepsilon_{K} \oo \id_{H} \oo \varepsilon_{K}\right)(F) \in H \oo H
		\label{eq:ProjectedTwistLeft}
		\\
		F_{K} &= \left( \varepsilon_{H} \oo \id_{K} \oo \varepsilon_{H} \oo id_{K} \right)(F) \in K \oo K
		\label{eq:ProjectedTwistRight}
	\end{align}
	are twists for $H$ and $K$, respectively.
	\end{itemize}
\end{example}

Twists allow one to change  a bialgebra or Hopf algebra by modifying its comultiplication and antipode. The following is a standard result, see for instance  \cite[Chapter 5.14]{EGNO}.

\begin{lemma}\label{lemma:TwistedHopfAlgebra}
	Let $H=\left( H,\mu,\eta, \Delta,\varepsilon,S \right)$ be a finite-dimensional Hopf algebra and $F\in H \oo H$ a twist for $H$. Then there is a twisted Hopf algebra $H_{F} = \left( H,\mu_{F}=\mu,\eta_{F}=\eta,\Delta_{F},\varepsilon_{F}=\varepsilon, S_{F} \right)$ with comultiplication and antipode given by
\begin{align*}
	\Delta_{F}(h)&= F\cdot \Delta(h)\cdot F^{-1},\qquad
	S_{F}(h)=Q\cdot S(h) \cdot Q^{-1},
\end{align*}
where
\begin{align}
	\label{eq:TwistDrinfel'dElement}
	Q:=F^{(1)}S(F^{(2)}),\quad
	Q^{-1}=S(F^{(-1)})F^{(-2)}.
\end{align}
If $H$ is quasitriangular with universal $R$-matrix $R$, then $H_F$ is quasitriangular with universal $R$-matrix
\begin{align}	
		R^F=F_{21} R F^{-1}.
	\label{eq:TwistedRMatrix}
\end{align}
We write $\Delta_{F}(h)=h_{(F1)} \oo h_{(F2)}$ in Sweedler notation.
\end{lemma}
Another standard result relates the tensor categories defined by the two Hopf algebras $H$ and $H_{F}$, see for instance \cite[Proposition 5.14.4]{EGNO}:
\begin{proposition}
	\label{proposition:FunctorTranslation}
	The following data defines a tensor equivalence $G:H_{F}\mathrm{-Mod}\to H\mathrm{-Mod} $ :
	\begin{itemize}
		\item For any $H$-module $M$, $GM$ is the same underlying vector space equipped with the same $H$-module structure.
		\item For any morphism $f:M\to N$ in $H\mathrm{-Mod}$, $Gf:GM \to GN$ is the same linear map.
		\item For the tensor unit $\mathbb{F}$ of $D(H)\mathrm{-Mod}$, the natural isomorphism $G\mathbb{F}\to \mathbb{F}$ to the tensor unit $\mathbb{F}$ of $H_{F}\mathrm{-Mod}$ is the identity.
		\item For $H$-modules $(M, \vartriangleright_{M}),(N, \vartriangleright_{N})$, the natural isomorphism $G\left( M \oo_{F} N \right) \to GM \oo GN$ is the linear map
			\begin{align}
				G\left(M\oo_{F} N\right) &\to GM \oo GN,\nonumber\\
				m \oo n &\mapsto \left(F^{(-1)} \vartriangleright_{M}m \right) \oo \left( F^{(-2)} \vartriangleright_{N} n \right)
				\label{eq:TwistCoherenceData}
		\end{align}
	\end{itemize}
	If $H$ is quasitriangular, then $G$ is a braided tensor equivalence.
\end{proposition}

Twisting defines an equivalence relation on the class of complex Hopf algebras. In the following, we consider Hopf algebras that are isomorphic up to a twist and call such Hopf algebras \emph{twist equivalent.}

\begin{definition}
	\label{definition:TwistEquivalence}
	Let $H$ and $K$ be Hopf algebras. A \emph{twist equivalence} is a pair $(F,\varphi)$ of a twist $F$ of $H$ and an isomorphism $\varphi:H_{F} \to K$ of Hopf algebras. If both $H$ and $K$ are quasitriangular and $\varphi$ is an isomorphism of quasitriangular Hopf algebras, then we call $(F,\varphi)$ a \emph{braided twist equivalence}.
\end{definition}

\begin{example}\cite[13.6.1]{R}
	If $(K,R)$ is a factorizable Hopf algebra, then $K \oo K$ is twist equivalent to $D(K)$, where $F=1_{K} \oo R^{(2)} \oo R^{(1)} \oo 1_{K}\in K^{\oo 4}$ is a twist of $K \oo K$ and the isomorphism $\varphi$ is given by 
	\begin{align}
		\varphi:  &D(K) \to (K \oo K)_{F}\nonumber\\
		&\alpha \oo h \mapsto \langle \alpha_{(2)} , R^{(2)} \rangle \langle \alpha_{(1)} , R^{(3)} \rangle  S(R^{(1)})\cdot h_{(1)} \oo R^{(4)}h_{(2)}
		\label{eq:Drinfel'dFactorisableIsomorphism}
	\end{align}
	in the notation introduced after Definition \ref{def:Twist}.  The twist equivalence becomes braided, when  $K \oo K$ is equipped with the $R$-matrix
	\begin{align*}
		R^{(-2)} \oo R^{(1)} \oo R^{(-1)} \oo R^{(2)}\in K^{ \oo 4}.
	\end{align*}
	\label{example:Drinfel'dQuadrupleTwist}
\end{example}


Twisting a semi-simple Hopf algebra $H$ preserves its semi-simplicity, as twisting does not change $H$'s algebra structure.
Additionally,  the Haar integral of $H$ and $H^*$ remain unchanged under twisting and the element $Q$ from \eqref{eq:TwistDrinfel'dElement} is preserved by the antipode:

\begin{remark}
	\label{remark:TwistedSemisimpleHopfalgebra}\cite[Remark 3.8]{AE}
	If $H$ is a semisimple finite-dimensional Hopf algebra with Haar integrals $\lambda\in H$, $\smallint \in H^{*}$, then $S(Q)=Q$, $\lambda$ is an Haar integral of $H_{F}$ and $\smallint$ is an Haar integral of $(H_{F})^{*}$.
\end{remark}

\subsection{Module and comodule algebras}

In this section we introduce certain module and comodule algebras  over a Hopf algebra $H$ and its Drinfel'd double $D(H)$ that are required in the following. 

\begin{definition} Let $H$ be a Hopf algebra. 
\begin{enumerate}
\item A \emph{$H$(-left) module algebra} is an associative unital algebra $(A,m,\eta)$ with an $H$-module structure $\rhd: H\oo A\to A$ such that
$m$ and $\eta$ are morphisms of $H$-modules. 

\item An \emph{$H$(-left) comodule algebra} is an associative unital algebra $(A,m,\eta)$ with an $H$-comodule structure $\delta:  A\to H\oo A$ such that
$m$ and $\eta$ are morphisms of $H$-comodules. 

\end{enumerate}
\end{definition}

$H$-right module algebras are defined as $H$-left module algebras over $H^{op}$ and $(H,H)$-bimodule algebras as left module algebras over $H\oo H^{op}$.
Analogously, $H$-right comodule algebras are left comodule algebras over $H^{cop}$ and $(H,H)$-bicomodule algebras are left comodule algebras over $H\oo H^{cop}$.
By composing the action of $H$ with the antipode of $H$, one can transform $H$-right module (co)algebras into $H^{cop}$-left module algebras and $H$-right comodule (co)algebras into $H^{op}$-left comodule (co)algebras.
By duality, any $H$-comodule algebra structure on a  vector space $V$ corresponds to a $H^*$-right module algebra structure on $V$ and vice versa.

\begin{example}\label{definition:CoregularAction} Let $H$ be a finite-dimensional Hopf algebra. Then $H$ is an $H$-bicomodule algebra with the \emph{left and right regular coaction}
$\delta=\Delta: H\to H\oo H$. Its dual
 $H^*$ is an $H$-module algebra  with the \emph{left coregular action}
	\begin{align}
		\label{eq:CoregularLeftAction}
		\vartriangleright: H \oo H^{*} &\to H^{*},\quad h \vartriangleright \alpha= \langle \alpha_{(2)} , h \rangle \alpha_{(1)}
	\end{align}
and an $H$-right module algebra  with the \emph{right coregular action}
	\begin{align}
		\label{eq:CoregularRightAction}
		\vartriangleleft:  H^{*} \oo H &\to H^{*},\quad \alpha  \vartriangleleft h= \langle \alpha_{(1)} , h \rangle \alpha_{(2)}.
	\end{align}
The latter defines an $H^{cop}$-left module algebra structure	given by
	\begin{align}
		\rhd': H \oo H^{*} \to H^{*},\quad h \oo \alpha\mapsto \alpha  \vartriangleleft S(h).
		\label{eq:RightCoregularLeftAction}
	\end{align}

\end{example}

\begin{remark}
	For a Drinfel'd double $D(H)$, the left and right  coregular actions take the  form:
\begin{align}
	\vartriangleright: D(H)  \oo D(H)^{*} \;\; &\to D(H)^{*} \nonumber\\
	\left( \beta \oo k \right)  \vartriangleright \left( h \oo \alpha \right)  	
	&= \langle \beta \oo k, S(x^{2})h_{(2)}x^{1} \oo \alpha_{(2)}  \rangle h_{(1)} \oo X^{1}\alpha_{(1)}X^{2}\nonumber
	\\
	&= \langle \beta_{(2)} \oo k , h_{(2)} \oo \alpha_{(2)} \rangle h_{(1)} \oo \beta_{(3)}\alpha_{(1)}S(\beta_{(1)})
	\label{eq:LeftCoregularActionDrinfel'd}
	\\
	\vartriangleleft:  D(H)^{*} \oo D(H) \;\;&\to D(H)^{*}\nonumber\\
	\left( h \oo \alpha \right)  \vartriangleleft \left( \beta \oo k \right) 
	&= \langle \beta \oo k , h_{(1)} \oo X^{1}\alpha_{(1)} X^{2} \rangle S(x^{2})h_{(2)}x^{1}  \oo \alpha_{(2)} \nonumber
	\\
	&= \langle \beta  \oo k_{(2)} , h_{(1)}  \oo \alpha_{(1)} \rangle  S(k_{(3)})h_{(2)}k_{(1)}  \oo \alpha_{(2)},
	\label{eq:RightCoregularActionDrinfel'd}
\end{align}
where we use the symbolic notation from Definition \ref{def:Drinfel'dDouble} and Remark \ref{lemma:DualDrinfel'd}.
\end{remark}

Given a module algebra $A$ over a Hopf algebra $H$, one can form the smash product  algebra $A\# H$. For the module algebras from Example \ref{definition:CoregularAction} this yields the  \emph{Heisenberg double} algebras. As $H$ is finite-dimensional, modules over the Heisenberg double are in bijection with Hopf modules over $H$, see for instance \cite[Cor 7.10.7, Defs 7.10.8, 7.10.9]{EGNO}. Just as Hopf modules, Heisenberg doubles exist in several versions, 
for the different combinations of $H$-left or right-actions on $H^*$. Further versions are obtained by exchanging the Hopf algebras $H$ and $H^*$.

\begin{definition}
	\label{def:HeisenbergDouble}

The Heisenberg double $\mathcal{H}_{R}(H)$ is 
the vector space $H \oo H^{*}$ with the algebra structure given by 
\begin{align}
	\left( h \oo \alpha \right) \cdot \left( k \oo \beta \right) :=
	\langle \alpha_{(1)} , k_{(2)} \rangle hk_{(1)}  \oo \alpha_{(2)}\beta.
	\label{eq:HeisenbergRightMult}
\end{align}

The Heisenberg double $\overline{\mathcal{H}}_{R}(H)$ is the vector space
$H \oo H^{*}$ with the algebra structure given by
\begin{align}
	\left( h \oo \alpha \right) \cdot \left( k \oo \beta \right) :=
	\langle \alpha_{(1)} , S(k_{(2)}) \rangle k_{(1)}h_{(1)}  \oo \beta\alpha_{(2)}.
	\label{eq:HeisenbergRightAlternativeMult}
\end{align}
\end{definition}

\begin{remark}\label{rem:heisenbergdoubleendo}
	The right Heisenberg double $\mathcal{H}_{R}(H)$ is isomorphic to $\mathrm{End}_{\C}(H)$, where the isomorphism is given in terms of the action
	\begin{align}
		\vartriangleright: \mathcal{H}_{R}(H) \oo H &\to H\nonumber\\
		\left( h \oo \alpha \right) \vartriangleright m = \langle \alpha , m_{(2)} \rangle hm_{(1)}
		\label{eq:HeisenbergAction}
	\end{align}
	of $\mathcal{H}_{R}(H)$ on $H$, see for instance \cite[9.4.3]{Mo}.
\end{remark}

\begin{remark}$\quad$
	\label{remark:HeisenbergDoubleComoduleAlgebra}
	\begin{compactenum}
	\item $\mathcal{H_{R}}(H)$ is a $D(H)^{*,op}$-left comodule algebra with the coaction given by the comultiplication
		\begin{align*}
			\Delta_{D(H)^{*}}: \mathcal{H}_{R}(H) \to D(H)^{*,op} \oo \mathcal{H}_{R}(H)
		\end{align*}
		and a $D(H)^{*}$-right comodule algebra with the coaction given by the comultiplication
		\begin{align*}
			\Delta_{D(H)^{*}}: \mathcal{H}_{R}(H) \to  \mathcal{H}_{R}(H) \oo D(H)^{*}
		\end{align*}
		\label{remarkPoint:RightHeisenbergComoduleAlgebra}
	\item $\overline{\mathcal{H}}_{R}(H)$ is a $D(H)^{*}$-left comodule algebra with the coaction given by the comultiplication
		\begin{align*}
			\Delta_{D(H)^{*}}: \overline{\mathcal{H}}_{R}(H) \to D(H)^{*} \oo \mathcal{H}_{R}(H)
		\end{align*}
		and a $D(H)^{*,op}$-right comodule algebra with the coaction given by the comultiplication
		\begin{align*}
			\Delta_{D(H)^{*}}: \mathcal{H}_{R}(H) \to  \mathcal{H}_{R}(H) \oo D(H)^{*,op}
		\end{align*}
		\label{remarkPoint:RightHeisenbergAlternativeComoduleAlgebra}
	\item The comultiplication of $D(H)^{*}$ is an algebra homomorphism
		\begin{align}
			\Delta_{D(H)^{*}} : D(H)^{*} \to \overline{\mathcal{H}}_{R}(H)  \oo \mathcal{H}_{R}(H) 
			\label{eq:ComultHeisenbergDoubleAlgebraHom}
		\end{align}
		and an algebra homomorphism
		\begin{align}
			\Delta_{D(H)^{*}} : D(H)^{*,op} \to \mathcal{H}_{R}(H)  \oo \overline{\mathcal{H}}_{R}(H)
			\label{eq:ComultHeisenbergDoubleOppositeAlgebraHom}
		\end{align}
	\end{compactenum}
\end{remark}

\begin{remark}
	The $D(H)^{*}$-comodule (resp. $D(H)^{*,op}$-comodule) algebras $\mathcal{H}_{R}(H)$ and $\overline{\mathcal{H}}_{R}(H)$ are related to $D(H)^{*}$ and $D(H)^{*,op}$ by a one-sided cotwist with the $R$-matrix from \eqref{eq:Rmatrix}. Concretely, we have for $h \oo \alpha,k \oo \beta \in H \oo H^{*}$:
	\begin{align}
		\left( h \oo \alpha \right)\cdot_{\mathcal{H}_{R}(H)} \left( k \oo \beta \right)
		&= \left( R^{(2)}  \vartriangleright \left( k \oo \beta \right) \right) \cdot_{D(H)^{*}} \left( R^{(1)}  \vartriangleright \left( h \oo \alpha \right) \right)
		\label{eq:RightHeisenbergDoubleRTwist}
		\\
		\left( h \oo \alpha \right)\cdot_{\overline{\mathcal{H}}_{R}(H)} \left( k \oo \beta \right)
		&= \left( R^{(-1)}  \vartriangleright \left( h \oo \alpha \right) \right) \cdot_{D(H)^{*}} \left( R^{(-2)}  \vartriangleright \left( k \oo \beta \right) \right)
		\label{eq:RightHeisenbergDoubleAlternativeRTwist}
	\end{align}

\end{remark}

\subsection{Tensor categories and Tannaka duality}

In this section we consider Tannaka reconstruction for fusion categories over $\C$ (see e.g.~\cite{EGNO,S,U,Ma,JS}).
We follow the conventions of \cite[Chapters 6 and 9]{EGNO} and define

\begin{definition}$\quad$
	\label{definition:FusionCat}
	
	\begin{compactenum}
	\item A \emph{fusion category} is a $\C$-linear, finitely semisimple, rigid monoidal category. 

\item 	A \emph{fiber functor} is a faithful, $\C$-linear, strict monoidal functor from a fusion category to $\mathrm{Vect}_{\C}^{fin}$.
\end{compactenum}
\end{definition}
For $\mathcal{C}$ a fusion category and $\omega_{\mathcal{C}}: \mathcal{C}\to \C$ a fiber functor, we denote $\mathrm{End}(\omega_{\mathcal{C}})$ the natural endomorphisms of $\omega_{\mathcal{C}}$. The composition of natural endomorphisms equips $\mathrm{End}(\omega_{\mathcal{C}})$ with a $\C$-algebra structure. The algebra $\mathrm{End}(\omega_{\mathcal{C}})$ also inherits a $\C$-coalgebra structure from the tensor product of $\mathcal{C}$, for details see~\cite[Chapter~5.2]{EGNO}.

We denote by $\boxtimes$ the Deligne product of locally finite $\C$-linear categories, see \cite[Chapter 1.11]{EGNO} and by  $\varphi_{\mathcal C}: \mathcal Z(\mathcal C)\to \mathcal C$ the canonical monoidal functor that forgets the half-braidings. We then have
\begin{lemma}
	\label{lemma:FibreFunctorClassic}
	Let  $\mathcal{C},\mathcal{D}$ be fusion categories and  
	$\omega_{\mathcal{C}}, \omega_{\mathcal{D}}$
	be fiber functors for $\mathcal{C}$ and $\mathcal{D}$.
	\begin{compactenum}
		\item Then $H=\mathrm{End}(\omega_{\mathcal{C}})$ is a semisimple, finite-dimensional Hopf algebra over $\C$.
		\item There is a strict tensor equivalence $\mathcal{C}\cong H\mathrm{-Mod}$. 
		\item If $\mathcal{C}$ is braided, then $H$ admits a quasitriangular structure such that the equivalence is braided.
		\item There is an isomorphism $D(H) \cong \mathrm{End}(\omega_\mathcal{C}\circ \varphi_\mathcal{C})$ of quasitriangular Hopf algebras. 
		\item There is an isomorphism $H \oo K \cong \mathrm{End}( \oo_{\C} \circ (\omega_{\mathcal{C}} \boxtimes \omega_{\mathcal{D}}))$ of Hopf algebras. 
	\end{compactenum}
\end{lemma}

The first two statements in Lemma \ref{lemma:FibreFunctorClassic} follow directly from \cite[Theorem 5.3.12]{EGNO}, and the third is a direct consequence of the first two. For the fourth statement, see \cite[Chapter 7.14]{EGNO}, in particular the discussion before Definition 7.14.1. The last statement follows from the first two and the Definition of the Deligne product.
\begin{lemma}
	\label{lemma:FibreFunctorTensorEquivalence}
	Let $\mathcal{C}, \mathcal{D}$ be  fusion categories, $T:\mathcal{C} \to \mathcal{D}$ a $\C$-linear tensor equivalence and $\omega_{\mathcal{C}},\omega_{\mathcal{D}}$ be fiber functors for $\mathcal{C}$ and $\mathcal{D}$ such that there is a natural isomorphism
	\begin{align*}
		\omega_{\mathcal{C}} \cong \omega_{\mathcal{D}}\circ T
	\end{align*}
	of tensor functors.
	Then the Hopf algebra $H=\mathrm{End}(\omega_{\mathcal{C}})$ admits a twist $F\in H \oo H$ such that $K=\mathrm{End}(\omega_{\mathcal{D}})\cong H_{F}$ as Hopf algebras.
	If $\mathcal{C}$ and $\mathcal{D}$ are braided and $T$ is a braided equivalence, then $K\cong H_{F}$ as quasitriangular Hopf algebras.
\end{lemma}
\begin{proof}
	The statement for tensor categories $\mathcal{C},\mathcal{D}$ follows immediately from \cite[Prop 5.14.4]{EGNO}. The braided version follows by combining with Lemma~\ref{lemma:FibreFunctorClassic}.3.
\end{proof}

The following table shows the relation between categorical notions and algebraic structures given by a fiber functor.
\\

\begin{tabular}[center]{|c|c|}
	\hline
	categorical notion & algebraic structure
	\\
	\hline\hline
	(braided) fusion category $\mathcal{C},\mathcal{D}$	& (quasitriangular) semisimple Hopf algebras $H$, $K$
	\\
	\hline
	center $\mathcal{Z(D)}$		& Drinfel'd double $D(K)$
	\\\hline
	Deligne product $\mathcal{C} \boxtimes \mathcal{D}$ & tensor product Hopf algebra $H \oo K$
	\\\hline
	tensor equivalence $\mathcal{C}\to \mathcal{D}$ & twist $F$ with $H_{F} \cong K$ 
	\\\hline
\end{tabular}
\section{Ribbon graphs}
\label{sec:ribbon}

In this section, we summarize the required background on {\em embedded graphs} or {\em ribbon graphs}. For more details,  see for instance \cite{BL,LZ}. All graphs considered in this article are directed and finite, but we allow  loops, multiple edges and univalent vertices.
\subsection{Paths}
\label{subsec:paths}
Paths in a graph are most easily described by orienting the edges of $\Gamma$ and considering the free groupoid generated by the resulting directed graph. Note that different choices of orientation yield isomorphic groupoids. 

 \begin{definition} \label{def:pathgroupoid}The \emph{path groupoid} $\mathcal G_\Gamma$ of a graph $\Gamma$ is the free groupoid generated by $\Gamma$. A \emph{path} in $\Gamma$ from a vertex $v$ to a vertex $w$ is a morphism $\gamma: v\to w$ in $\mathcal G_\Gamma$.
 \end{definition}

 The objects of $\mathcal G_\Gamma$ are the vertices of $\Gamma$. A morphism from $v$ to $w$ in $\mathcal G_\Gamma$ is a finite sequence $\rho=e_1^{\epsilon_1}\circ ...\circ e_n^{\epsilon_n}$, $\epsilon_i\in\{\pm 1\}$ of oriented edges $e_i$ and their inverses such that the starting vertex of the first edge 
$e_{n}^{\epsilon_{n}}$
 is $v$, the target vertex of the last edge 
$e_{1}^{\epsilon_{1}}$
 is $w$, and the starting vertex of each edge in the sequence is the target vertex of the preceding edge.
 These sequences are taken with the relations  $e\circ e^\inv=1_{t(e)}$ and $e^\inv\circ e=1_{s(e)}$, where $e^\inv$ denotes the edge $e$ with the reversed orientation, $s(e)$ the starting and $t(e)$ the target vertex of $e$ and we set $s(e^{\pm 1})=t(e^{\mp 1})$. 

 A sequence $e_{1}^{\varepsilon_{1}} \circ \cdots \circ e_{n}^{\varepsilon_{n}}$  is said to be a \emph{reduced word} or \emph{in reduced form}, if it does not contain subsequences of the form $e^{-1}\circ e$ and $e \circ e^{-1}$ for $e\in E$.
 
 An edge $e\in E$ with  $s(e)=t(e)$ is called a {\em loop}, and a path  $\rho\in \mathcal G_\Gamma$  is called {\em closed} if it is an automorphism of a vertex.  
 We call a path $\rho\in\mathcal G_\Gamma$ a {\em subpath} of a path $\gamma \in \mathcal G_\Gamma$ if  the expression for $\gamma$ as a reduced word in $E$  is of the form $\gamma=\gamma_1\circ\rho\circ\gamma_2$ with (possibly empty) reduced words $\gamma_1,\gamma_2$. We call it  a {\em proper subpath} of $\gamma$ if $\gamma_1,\gamma_2$ are not both empty. We say that two paths $\rho,\gamma\in \mathcal G_\Gamma$ {\em overlap} if there is an edge in $\Gamma$ that is traversed by both $\rho$ and $\gamma$, and by both  in the same direction.

\subsection{Ribbon graphs}
The graphs we consider have additional structure. They are called
 {\em ribbon graphs}, \emph{fat graphs} or \emph{embedded graphs} and give a combinatorial description of oriented surfaces with or without boundary. 

 \begin{definition} A \emph{ribbon graph} is a directed graph together with a cyclic ordering of the edge ends at each vertex. 
 A \emph{ciliated vertex} in a ribbon graph is a vertex together with  a linear ordering of the incident  edge ends  that is compatible with their cyclic ordering.
 \end{definition}
 
A ciliated vertex in a ribbon graph is obtained  by selecting one of its incident edge ends  as the starting end of the linear ordering. We indicate 
this in figures by assuming the counterclockwise cyclic ordering in the plane and inserting a line, the {\em cilium},  that separates the edges of minimal and maximal order, as shown in Figure \ref{fig:ciliated}. We say that an edge end $e_{2}$ at a ciliated vertex $v$ is {\em between} two edge ends $e_{1}$ and $e_{3}$ incident at $v$ if $e_{1}<e_{2}<e_{3}$ or $e_{3}<e_{2}<e_{1}$. We denote by $\st(e_{1})$ the starting end and by $\ta(e_{1})$ the target end of a directed edge $e_{1}$.

The cyclic ordering of the edge ends at each vertex  allows  one to thicken the edges of a ribbon graph to  strips or ribbons and its vertices to polygons.
 It also equips the ribbon graph with the notion of a face. One says that a path in a ribbon graph $\Gamma$ turns {\em maximally left} at a vertex $v$ if it
 enters $v$ through an edge end $\alpha$ and leaves it through an edge end $\beta$ that comes directly before $\alpha$ with respect to 
  the cyclic ordering at $v$. 

 \begin{definition}\label{def:face} Let $\Gamma$ be a ribbon graph. 
 \begin{compactenum}
 \item A \emph{partial face} in $\Gamma$ is a path  that turns maximally left at each  vertex in the path and traverses each edge at most once in each direction. 
 \item A \emph{ciliated face} in $\Gamma$  is a closed partial face whose cyclic permutations are also partial faces.
 \item A \emph{face} of $\Gamma$ is an equivalence class of ciliated faces under  cyclic permutations. 
 \end{compactenum}
 \end{definition}

 \begin{figure}
	\centering
		\begin{tikzpicture}[scale=.6]
			\draw [color=black, fill=black] (0,0) circle (.2); 
			\draw[color=black, style=dotted, line width=1pt] (0,-.2)--(0,-1);
			\draw [black,->,>=stealth,domain=-90:240] plot ({cos(\x)}, {sin(\x)});
			\draw[color=red, line width=1.5pt, ->,>=stealth] (.2,-.2)--(2,-2);
			\draw[color=violet, line width=1.5pt, <-,>=stealth] (.2,0)..controls (6,0) and (0,6).. (0,.2);
			\draw[color=cyan, line width=1.5pt, ->,>=stealth] (.2,.2)..controls (4,4) and (-4,4).. (-.2,.2);
			\draw[color=magenta, line width=1.5pt, <-,>=stealth] (-.2,0)--(-3,0);
			\node at (.7,-.7)[color=red, anchor=north]{${1}$};
			\node at (1,-.3)[color=violet, anchor=west]{${2}$};
			\node at (.7,.7)[color=cyan, anchor=west]{$3$};
			\node at (0,1.2)[color=violet, anchor=west]{${4}$};
			\node at (-.6,.8)[color=cyan, anchor=south]{$5$};
			\node at (-.9,.4)[color=magenta, anchor=east]{$6$};
		\end{tikzpicture}
		\qquad\qquad
				\begin{tikzpicture}[scale=.6]
		\begin{scope}[shift={(-6,0)}]
		\draw [color=black, fill=black] (-2,0) circle (.2); 
		\draw [color=black, fill=black] (2,0) circle (.2); 
		\draw [color=black, fill=black] (1,2) circle (.2); 
		\draw [color=black, fill=black] (4,4) circle (.2); 
		\draw [color=black, fill=black] (-4,4) circle (.2); 
		\draw[color=black, line width=1pt, style=dotted](1.9,.1)--(1.3,.7);
		\draw[color=blue, line width=1.5pt, ->,>=stealth] (-1.8,0)--(1.8,0);
		\node at (.7,1.2)[color=blue, anchor=north]{$1$};
		\draw[color=red, line width=1.5pt, ->,>=stealth] (-1.8,.2).. controls (2,0) and (-2,4).. (-2,.2);
		\node at (-1.5,2.5)[color=red, anchor=west]{$2$};
		\draw[color=magenta, line width=1.5pt, ->,>=stealth] (-2.2,.2)--(-3.8,3.8);
		\node at (-2,2.3)[color=magenta, anchor=east]{$3$};
		\draw[color=cyan, line width=1.5pt, ->,>=stealth] (2,.2)--(4,3.8);
		\node at (2.2,1.3)[color=cyan, anchor=east]{$7$};
		\draw[color=violet, line width=1.5pt, ->,>=stealth] (3.8,4)--(-3.8,4);
		\node at (0,3.7)[color=violet, anchor=north]{${4}$};
		\draw[color=brown, line width=1.5pt, ->,>=stealth] (3.8,3.8)--(1.2,2.2);
		\node at (2.1,2.4)[color=brown, anchor=north]{$6$};
		\node at (1.9,2.8)[color=brown, anchor=south]{$5$};
		\draw[color=black, line width=1.5pt] (-2.2,0)--(-2.8,0);
		\draw[color=black, line width=1.5pt ] (-2,-.2)--(-2,-.8);
		\draw[color=black, line width=1.5pt ] (-2.1,-.1)--(-2.6,-.6);
		\draw[color=black, line width=1.5pt] (2.2,0)--(2.8,0);
		\draw[color=black, line width=1.5pt ] (2,-.2)--(2,-.8);
		\draw[color=black, line width=1.5pt ] (2.1,-.1)--(2.6,-.6);
		\draw[color=black, line width=1.5pt] (-4,4.2)--(-4,4.6);
		\draw[color=black, line width=1.5pt ] (-4.2,4)--(-4.8,4);
		\draw[color=black, line width=1.5pt ] (-4.1,4.1)--(-4.6,4.6);
		\draw[color=black, line width=1.5pt] (4,4.2)--(4,4.6);
		\draw[color=black, line width=1.5pt ] (4.2,4)--(4.8,4);
		\draw[color=black, line width=1.5pt ] (4.1,4.1)--(4.6,4.6);
		\draw[line width=.5pt, color=black,->,>=stealth] plot [smooth, tension=0.6] coordinates 
      {(1.65,.45)(3.4,3.2)(1.5,2)(1,1.5)(.5,2)(2.8,3.6)(0,3.6)(-3.2,3.6)(-2.6,2) (-2.2,.9)(-1.8,2)(-.5,2)(.3,1)(0,.4)(1.4,.4)};
		\end{scope}
		\end{tikzpicture}
\caption{A ciliated vertex and a ciliated face in a directed ribbon graph}
\label{fig:ciliated}
\end{figure}
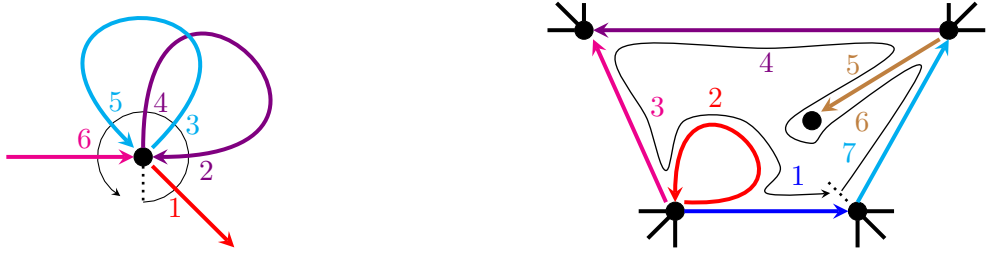

Examples of  faces and partial faces are shown in Figure \ref{fig:ciliated} .
Each  face defines a cyclic  ordering of the edges and their inverses in  the face. A ciliated face  is a face together with the choice of a starting vertex and induces a linear ordering of these edges. These orderings are taken counterclockwise, as shown in Figure \ref{fig:ciliated}. 
 
 A graph $\Gamma$ embedded into an oriented surface $\Sigma$ inherits a cyclic ordering of the edge ends at each vertex from the orientation of $\Sigma$ and hence a ribbon graph structure. Conversely, every ribbon graph $\Gamma$ defines a compact oriented surface $\Sigma_\Gamma$ that is obtained by attaching a disc at each face.  For a graph $\Gamma$ embedded into an oriented surface  $\Sigma$, the surfaces $\Sigma_\Gamma$ and $\Sigma$ are homeomorphic if and only if $\Sigma\setminus\Gamma$ is a disjoint union of discs.  
 An oriented surface with a boundary is obtained from a ribbon graph $\Gamma$ by attaching annuli instead of discs to some of its faces. For details on this topic see~\cite[Chapter 1.3]{LZ}.

  \subsection{Thickening of a ribbon graph}

The cyclic ordering of the edge ends  at each vertex allows one to thicken the edges of the graph to strips and its vertices to polygons, as shown in Figure \ref{figure:Thickening}. 
The resulting graph inherits a ribbon graph structure from $\Gamma$. 

\begin{definition} \label{def:sites}
	The \emph{thickening} of a ribbon graph $\Gamma$ with edge set $E$ is the ribbon graph $D(\Gamma)$ in which each edge $e\in E$ of $\Gamma$  is replaced by  four edges $e^{s},e^{t},e^{L}, e^{R}$, as shown in Figure \ref{figure:Thickening}. 
	Vertices of $D(\Gamma)$ are called \emph{sites}.
\end{definition}
Every site is associated to a vertex and a face of $\Gamma$, see Figure~\ref{figure:Thickening}. We call two such sites $s,t$ \emph{disjoint}, if neither the vertices nor the faces of $\Gamma$ associated to $s$ and $t$ coincide.
\begin{figure}[H]
	\centering
	\begin{tikzpicture}[scale=1, baseline=(current bounding box.center)]
		\begin{scope}[shift={(0,0)}]
			\draw[color=black, line width=1,fill=black] (0,0) circle (.15);
			\draw[color=black, line width=1,fill=black] (2,0) circle (.15);
			\draw[color=red, line width=2.5,->,>=stealth] (-1,1) -- (-.1,.1);
			\draw[color=blue, line width=2.5,->,>=stealth] (-1,-1) -- (-.1,-.1);
			\draw[color=violet, line width=2.5,->,>=stealth] (0.15,0) -- node[above]{$e$} (1.85,0);
			\draw[color=green, line width=2.5,->,>=stealth] (2.1,.1) -- (3,1);
			\draw[color=cyan, line width=2.5,->,>=stealth] (2.1,-.1) -- (3,-1);
		\end{scope}
		\draw[color=black, line width=1.5pt,->,>=stealth] (4,0) -- (6,0);
		\begin{scope}[shift={(8,0)}]
			\draw[color=red, line width=2.5,->,>=stealth] (-.8,0.05) -- (-.05,.45);
			\draw[color=red, line width=2.5,->,>=stealth] (-1,1.5) -- (-.05,.45);
			\draw[color=red, line width=2.5,->,>=stealth] (-1.85,1) -- (-.9,.05);
			\draw[color=blue, line width=2.5,->,>=stealth] (-.05,-.45) -- (-.8,-.05);
			\draw[color=blue, line width=2.5,->,>=stealth] (-1,-1.5) -- (-.05,-.55);
			\draw[color=blue, line width=2.5,->,>=stealth] (-1.85,-1) -- (-.9,-.05);
			\draw[color=violet, line width=2.5,->,>=stealth] (0,-.4) --node[right]{$e^{s}$} (0,.4);
			\draw[color=violet, line width=2.5,->,>=stealth] (3,-.4) --node[left]{$e^{t}$} (3,.4);
			\draw[color=violet, line width=2.5,->,>=stealth] (.1,.5) --node[above]{$e^{L}$} (2.9,.5);
			\draw[color=violet, line width=2.5,->,>=stealth] (.1,-.5) --node[above]{$e^{R}$} (2.9,-.5);
			\draw[color=green, line width=2.5,->,>=stealth] (3.8,.05) -- (3.05,.45);
			\draw[color=green, line width=2.5,->,>=stealth] (3.05,.55) -- (4,1.5);
			\draw[color=green, line width=2.5,->,>=stealth] (3.9,.05) -- (4.85,1);
			\draw[color=cyan, line width=2.5,->,>=stealth] (3.05,-.45) -- (3.8,-.05);
			\draw[color=cyan, line width=2.5,->,>=stealth] (3.05,-.55) -- (4,-1.5);
			\draw[color=cyan, line width=2.5,->,>=stealth] (3.9,-.05) -- (4.85,-1);
			\draw[color=black, line width=1,fill=black] (0,.5) circle (.1);
			\draw[color=black, line width=1,fill=black] (0,-.5) circle (.1);
			\draw[color=black, line width=1,fill=black] (-.85,0) circle (.1);
			\draw[color=black, line width=1,fill=black] (3,.5) circle (.1);
			\draw[color=black, line width=1,fill=black] (3,-.5) circle (.1);
			\draw[color=black, line width=1,fill=black] (3.85,0) circle (.1);
		\end{scope}
	\end{tikzpicture}
	\caption{Thickening of a ribbon graph.}
	\label{figure:Thickening}
\end{figure}
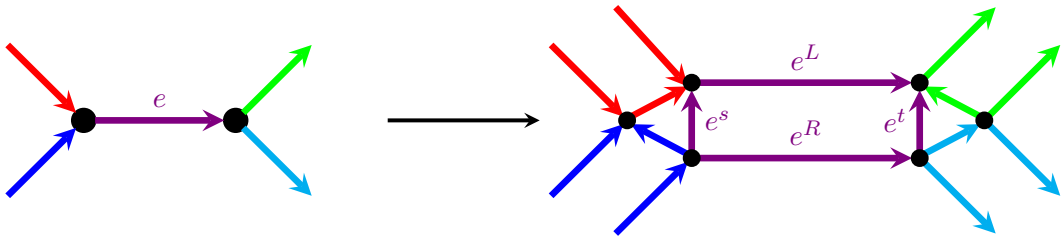


In the following, we often consider various paths in the thickening $D(\Gamma)$, that is, morphisms in the associated path groupoid.
We write $\rho: u\to v$ for a path $\rho$ in $D(\Gamma)$ from a site $u$ to a site $v$.  
We often draw paths in $D(\Gamma)$ in the graph $\Gamma$ for better legibility. 
In this case, $e^{\pm L}$ and $e^{\pm R}$ are drawn slightly to the left and right of an edge $e\in E$, viewed in the direction of its orientation, and $e^{\pm s}, e^{\pm t}$ cross $e$ at the starting and target end, as shown in Figure \ref{fig:paths}.

We also consider the starting point and end point of paths in $D(\Gamma)$ relative to its sites. We say that a path ends to the left or right of a site, when it ends to the left or right of the site viewed in the direction from the adjacent vertex of $\Gamma$ to the adjacent face. 
This translates into the following condition involving the four edges in each rectangle.

\begin{definition}\label{def:leftorright} Let $u,v$ be sites and  $\rho=\rho_1^{\eta_1}\circ \ldots\circ \rho_n^{\eta_n}: u\to v$ the reduced expression for a path $\rho$ in $D(\Gamma)$ with $\rho_i\in E$ and $\eta_i\in \{\pm s,\pm t,\pm L,\pm R\}$.  We say that $\rho$
\begin{itemize}
\item  \emph{ends to the left} of $v$, denoted $\rho: u\to v^L$, if $\eta_1\in\{-s,t,-R,L\}$,
\item  \emph{ends to the right} of $v$, denoted $\rho: u\to v^R$, if $\eta_1\in\{s,-t,R,-L\}$,
\item  \emph{starts to the left} of $u$,  denoted $\rho: u^L\to v$, if $\eta_n\in\{s,-t,R,-L\}$,
\item  \emph{starts to the right} of $u$,  denoted $\rho: u^R\to v$, if $\eta_n\in\{-s,t,-R,L\}$.
\end{itemize}
We call a path $\rho:u^{L}\to v^{R}$ a \emph{left-right path} and similarly use the terms \emph{left-left, right-left} and \emph{right-right path} (cf. Figure~\ref{fig:paths}).
\end{definition}

In the following, we often consider certain distinguished paths in thickened ribbon graphs known as \emph{ribbon paths}. 
These paths were considered in many publications concerned with Kitaev models such as \cite{Ki,BD, Me} to define vertex and face operators and ribbon operators.

\begin{definition}\label{def:ribbon path} Let $\Gamma$ be a ribbon graph with thickening $D(\Gamma)$.
\begin{compactenum}
\item A \emph{ribbon path} is a path $\rho$ in $D(\Gamma)$ such that
	\begin{compactenum}
		\item every edge of $D(\Gamma)$ occurs at most once in the reduced form of $\gamma$.
		\item if the reduced form of $\gamma$ contains one of the edges $e^{\pm L}$ or $e^{\pm R}$ for an edge  $e\in E$, then it does not contain the edges $e^{\pm s}$ and $e^{\pm t}$.
	\end{compactenum}
	\item The \emph{vertex path of the site $u$} is the closed  ribbon path $v:u\to u$ in $D(\Gamma)$ whose reduced form only contains edges of the form $e^{- s}, e^{ t}$. 
	\item The \emph{face path of the site $u$} is the closed ribbon path $f:u\to u$ in $D(\Gamma)$ whose reduced form only contains edges of the form $e^{- L}, e^{ R}$. 
\end{compactenum}
\end{definition}
Examples of ribbon paths are shown in Figure \ref{fig:paths}. 
As a consequence of the first property there are two types of ribbon paths with opposite orientation. 
A \emph{right-left} ribbon path traverses only edges of the form $e^{L}, e^{-R}, e^{-s}, e^{t}$ while a \emph{left-right} ribbon path only traverses edges of the form $e^{-L}, e^{R}, e^{s}, e^{-t}$, as shown in Figure \ref{fig:paths}. Note that vertex paths are right-left paths and face paths are left-right paths by definition.

\begin{figure}
	\centering
\def\svgscale{.4}
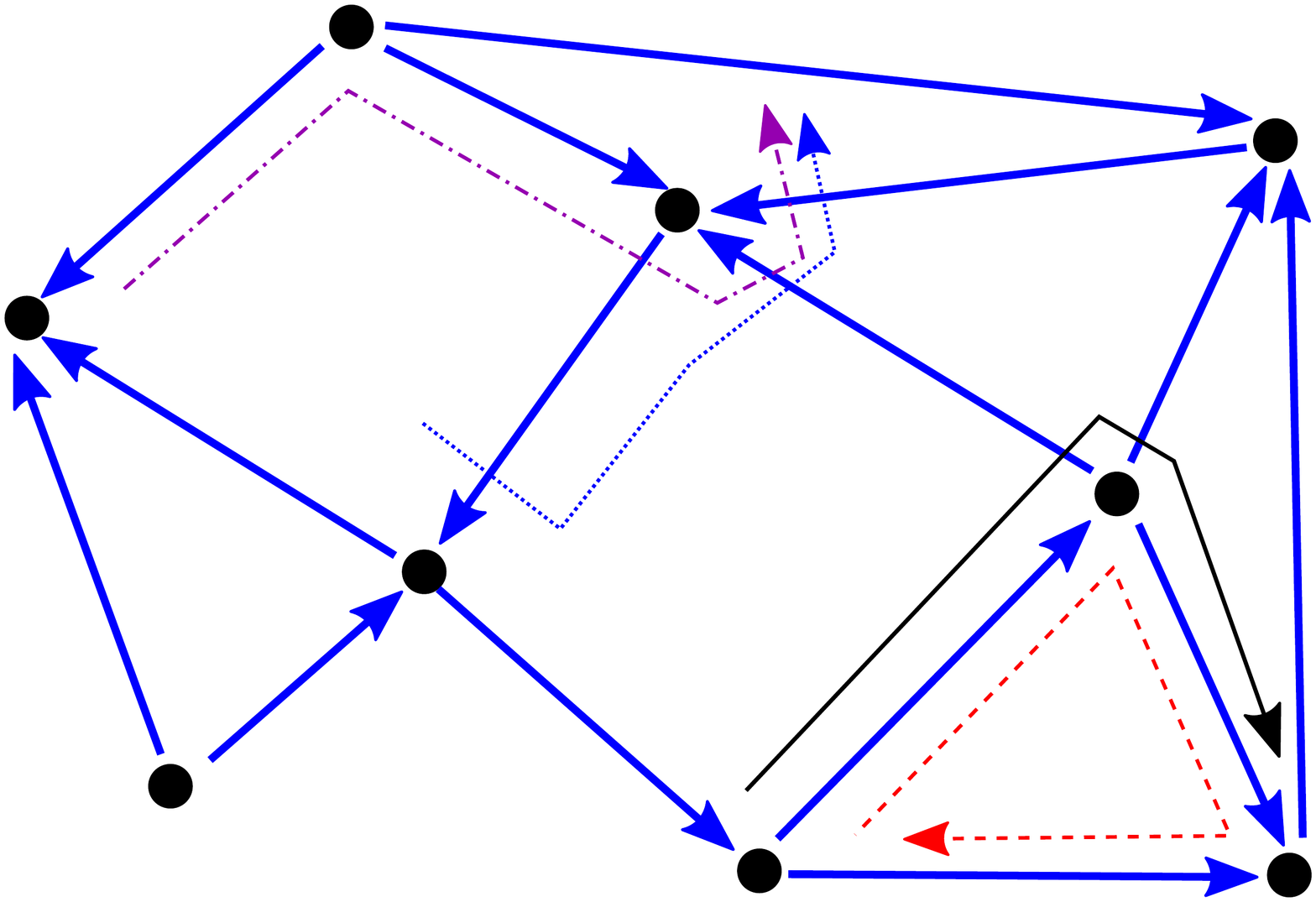
\caption{Examples of simple paths, drawn in a ribbon graph:\\
$\bullet$  $\nu$ is a right-left vertex path,  \\
$\bullet$ $\rho$ is a right-left ribbon path,\\
$\bullet$ $\sigma$ is a  left-right face path,  \\
$\bullet$ $\omega$ is a left-right ribbon path, \\
$\bullet$ $\tau$ is a  right-right simple path, but not a ribbon path.\\
Only the paths $\tau$ and $\omega$ cross, and they form a left joint $(\tau,\omega)_{\prec}$. }
\label{fig:paths}
\end{figure}

To describe the relative position of paths in $D(\Gamma)$, it is convenient to introduce a shorthand notation for the paths around the thickened edge $e$ in Figure \ref{figure:Thickening}. 
We set
\begin{align}\label{eq:eddef}
	&e^{d}= e^{t}\circ e^{R},
	&
	&e^{\overline{d}} = e^{-t}\circ e^{L}
	&
	&e^{d'}= e^{L}\circ e^{s} &
	&e^{\overline{d'}} = e^{R}\circ e^{-s}.
\end{align}

\begin{definition}
	\label{def:cross}Let $\rho_1,\rho_2$ be paths in $D(\Gamma)$. We say that $\rho_1$ and $\rho_2$
\begin{itemize}
	\item \emph{do not cross}, denoted $(\rho_{1}, \rho_{2})_{\emptyset}$, if for any edge $e \in E$ traversed by both $\rho_{1}$ and $\rho_{2}$ the following condition is fulfilled: If $\rho_{1}$ traverses $e^{\pm s}$, then $\rho_{2}$ only traverses $e^{\pm t}$ and if $\rho_{1}$ traverses $e^{\pm L}$ then $\rho_{2}$ only traverses $e^{\pm R}$ and vice versa. 
	\item 	form a \emph{left joint}, denoted $(\rho_{1}, \rho_{2})_{\prec}$, if there are (possibly empty) paths $\rho, \rho_{1}', \rho_{2}'$ and an edge $e$ such that $\rho$ is a ribbon path, $\rho, \rho_{1}',\rho_{2}'$ do not cross and the reduced expressions for $\rho_{1},\rho_{2}$ are
		\begin{align*}
			\rho_{1} = \rho \circ e^{a} \circ \rho_{1}' \text{ and } \rho_{2} = \rho\circ e^{b} \circ \rho_{2'}
		\end{align*}
		with either
		 $a\in \left\{ \pm L, \pm R, \pm d, \pm \overline{d}, \pm d', \pm \overline{d'} \right\}$, $b\in \left\{ \pm s, \pm t \right\}$
		 \\
		  $\qquad$ or $a\in \left\{ \pm L, \pm R \right\}$, $b\in \left\{ \pm s, \pm t, \pm d, \pm \overline{d},\pm d', \pm \overline{d'} \right\}$.
	\item 	form a \emph{right joint}, denoted $(\rho_{1}, \rho_{2})_{\succ}$, if $\left( \rho_{1}^{-1}, \rho_{2}^{-1} \right)_{\prec}$.
	\item form a \emph{middle joint}, denoted $\left( \rho_{1},\rho_{2} \right)_{\succ\prec}$, if there are decompositions $\rho_{1}= \rho_{1}'\circ\rho_{1}''$ and $\rho_{2}=\rho_{2}'\circ\rho_{2}''$ in reduced form such that either $(\rho_{1}',\rho_{2}')_{\succ}$ and $(\rho_{1}'',\rho_{2}'')_{\prec}$ or $(\rho_{2}',\rho_{1}')_{\succ}$ and $(\rho_{2}'',\rho_{1}'')_{\prec}$.
\end{itemize}
\end{definition}

Examples of paths that do not cross and paths that form a left joint are shown in Figure \ref{fig:paths}.

We also consider  slight generalizations of ribbon paths. These are paths that are allowed to traverse two adjacent sides of the rectangle in $D(\Gamma)$ associated with an edge $e\in E$ in Figure \ref{figure:Thickening}. An example of such a path is shown in Figure \ref{fig:paths}.

\begin{definition} \label{def:simplepath} A reduced path $\rho$ in $D(\Gamma)$ is called \emph{simple}, if for each decomposition  $\rho=\rho_1\circ \rho_2$ with reduced paths $\rho_1,\rho_2$ one has either $(\rho_1,\rho_2)_\emptyset$ or $\rho_1=\rho'_1\circ x$ and $\rho_s=y\circ \rho'_2$ with $x\circ y\in \{e^{\pm d}, e^{\pm \bar d}, e^{\pm d'}, e^{\pm \bar d'}\}$ for an edge $e\in E$, $(\rho'_1,\rho'_2)_\emptyset$, $(x\circ y, \rho'_1)_\emptyset$ and $(x\circ y, \rho'_2)_\emptyset$.

\end{definition}

\section{Kitaev models}
\label{section:KitaevModel}

Kitaev lattice models were introduced in \cite{Ki}  as a realistic model of a topological quantum computer. The models in \cite{Ki} were based on group algebras of finite groups and were later generalized in \cite{BMCA,BCKA} to finite-dimensional semisimple Hopf-${*}$ algebras.  They were related to Levin-Wen models in \cite{BA, Kr} and interpreted as a Hopf algebra gauge theory in \cite{Me}.  In
 \cite{BK12} it was shown that these models are closely related to topological quantum field theories of Turaev-Viro type \cite{TV}. More specifically, the \emph{protected space} of the Kitaev model for a finite-dimensional semisimple Hopf algebra $H$ on an oriented surface $\Sigma$ is the vector space a Turaev-Viro TQFT for the spherical fusion category $H\text{-Mod}$ assigns to $\Sigma$.  
 
 We summarize the Kitaev models without defects and boundaries  from \cite{BMCA}, but with some minor differences in conventions and in a formulation that prepares their generalization to Kitaev models with topological defects and boundaries in the following sections.

The Kitaev model requires a ribbon graph $\Gamma=(V,E)$ and a finite-dimensional, semisimple Hopf algebra $H$. In its most basic form the model then features:
	\begin{itemize}
		\item The \emph{extended space}  $\HSpace = H^{ \oo E}= \tensor_{e\in E}H$, i.e. the $|E|$-fold tensor product of $H$ with itself. Here we associate one copy of $H$ to every edge $e\in E$. 
		\item  \emph{Vertex operators} $A^{h}_{s}$ and \emph{face operators} $B^{\alpha}_{s}:H^{ \oo E}\to H^{ \oo E}$ for every site $s$ of $\Gamma$ and $h\in H, \alpha\in H^{*}$. As explained in Theorem \ref{ht:siterep} below, they define representations of the Drinfel'd double $D(H)$ on $H^{ \oo E}$.
		 \item The \emph{protected space}		
			 $$\mathcal{L}= \left\{ m \in H^{ \oo E} \;|\; A^{h}_{s}m = \varepsilon(h) m, B^{\alpha}_{s}m = \alpha(1) m \text{ for all sites }s,h\in H,\alpha\in H^{*}  \right\}$$
			 of invariants of the face and vertex operators.%
		 \end{itemize}
		 To consider models with excitations and to investigate the transport and braiding of excitations one also requires a set of operators on the extended space $\HSpace$ called \emph{ribbon operators} in \cite{Ki,BD} and \emph{holonomies} in \cite{Me} and assigned to certain paths in $D(\Gamma)$:		 
	\begin{itemize}	 
		 \item Holonomies $\Hol_{\rho}^{h \oo \alpha}\in \End_\C( \HSpace)$ for every path $\rho$ in $D(\Gamma)$ and $h \oo \alpha\in H \oo H^{*}$.
	\end{itemize}

All linear endomorphisms on $\HSpace$ are composites of four operators, called \emph{triangle operators} in \cite{BD}, that are assigned to each edge of $D(\Gamma)$ and viewed as the holonomies of the associated paths $e^t,e^s, e^L,e^R$. Hence, every edge $e\in E$ is associated with  four triangle operators.
 The operators are linear maps $\HSpace\to\HSpace$ which only act on the tensor factor associated to $e$. 
\begin{definition}
	\label{definition:HolonomyBasic}
	Let $e$ be an edge, $\eta\in \left\{ \pm t,\pm s,\pm R,\pm L \right\}$ and $h,m \in H, \alpha\in H^{*}$. The \emph{triangle operators} for $e$ are the following maps 
	$\Hol_{e^{\eta}}^{h \oo \alpha} :\HSpace\to \HSpace$ which only act on the tensor factor associated to $e$:
\begin{align}
	\Hol_{e^{s}}^{h \oo \alpha} (m) &= \alpha(1) \cdot mh, &
	\Hol_{e^{t}}^{h \oo \alpha} (m) &=\alpha(1) \cdot hm
	\label{eq:HolonomyEndsStandard}
	\\
	\Hol_{e^{R}}^{h \oo \alpha} (m) &= \varepsilon(h) \langle \alpha , m_{(2)} \rangle m_{(1)}, &
	\Hol_{e^{L}}^{h \oo \alpha} (m) &= \varepsilon(h)\langle \alpha , m_{(1)} \rangle m_{(2)}
	\label{eq:HolonomySidesStandard}
	\\
	\Hol_{e^{\eta}}^{h \oo \alpha} (m) &= \Hol_{e^{-\eta}}^{ S_{D(H)^{*}}(h \oo \alpha) } (m) & &\text{for $\eta\in \left\{ -s,-t,-R,-L \right\}$.}
	\label{eq:HolonomyOpposites}
\end{align}
Here $m\in H \subset \HSpace$ is the element in the tensor factor associated to $e$ and we extend linearly to $\HSpace$.
\end{definition}

A direct computation shows that for the element $m\in H$ associated to the edge $e$ we have 
\begin{align*}
	\Hol_{e^t}^{h\oo\varepsilon}\circ\Hol_{e^R}^{1\oo\alpha} (m)= \langle \alpha , m_{(2)} \rangle hm_{(1)},
\end{align*}
i.e. the edge operators $\Hol_{e^t}^{h\oo\alpha}$ and $\Hol_{e^R}^{h\oo\alpha}$ generate an algebra isomorphic to the Heisenberg double $\mathcal H_R(H)$ and they act on the copy of $H$ associated to the edge $e$ by the action from~\eqref{eq:HeisenbergAction}.
 By Remark \ref{rem:heisenbergdoubleendo}, this implies  that the edge operators for all edges $e\in E$ generate an algebra isomorphic to $\End_\C(\HSpace)$.

Holonomies along paths in $D(\Gamma)$ are constructed from these basic building blocks by applying the comultiplication of $D(H)^*$ to their argument $h\oo\alpha$ and composing these maps as follows:

\begin{definition}
	\label{definition:HolonomyComposite}
Let $\rho=\rho_{1}\circ\rho_{2}$ a simple path in reduced form with $\rho_{2}:s \to t$ and $\rho_{1}:t \to u$.
Write $\beta:=h \oo \alpha\in D(H)^{*}$ and $\Delta_{D(H)^{*}}(\beta)=\beta_{(1)} \oo \beta_{(2)}$. Then we define recursively
\begin{align}
	&\Hol_{\rho}^{\beta}  = \Hol_{\rho_{1}}^{\beta_{(1)}}\circ \Hol_{\rho_{2}}^{\beta_{(2)}} \quad &\text{if $(\rho_{1},\rho_{2})_{\emptyset}$ or $(\rho_{2},\rho_{1}^{-1})_{\prec}$}
	\label{eq:HolonomyRecursionStandard}
	\\
	&\Hol_{\rho}^{\beta}  = \Hol_{\rho_{2}}^{\beta_{(2)}}\circ \Hol_{\rho_{1}}^{\beta_{(1)}} &\text{if $(\rho_{1}^{-1},\rho_{2})_{\prec}$}
	\label{eq:HolonomyRecursionLeftJoin}
\end{align}
\end{definition}

\begin{lemma}
	\label{lemma:HolonomyWellDefined}
	For a simple path $\rho$, the definition of $\Hol_{\rho}^{\beta} $ does not depend on the decomposition chosen in~\ref{definition:HolonomyComposite}.
\end{lemma}
\begin{proof}
	Let $\rho=\rho_{1}\circ\rho_{2}\circ\rho_{3}$ in reduced form. First let $(\rho_{1}^{-1},\rho_{2})_{\prec}$ and $(\rho_{2}^{-1},\rho_{3})_{\prec}$. Then we also have $(\rho_{1}^{-1},\rho_{2}\circ\rho_{3})_{\prec}$ and $((\rho_{1}\circ\rho_{2})^{-1},\rho_{3})_{\prec}$. Then
	\begin{align*}
		\Hol_{\rho_{1}\circ (\rho_{2}\circ\rho_{3})}^{ \beta}
		&= 
		\Hol_{\rho_{2}\circ\rho_{3}}^{ \beta_{(2)}} \circ \Hol_{\rho_{1}}^{ \beta_{(1)}} 
		= \Hol_{\rho_{3}}^{ \beta_{(3)}} \circ \Hol_{\rho_{2}}^{\beta} \circ\Hol_{\rho_{1}}^{ \beta_{(1)}} 
		= \Hol_{\rho_{3}}^{\beta_{(2)}} \circ \Hol_{\rho_{1}\circ\rho_{2}}^{ \beta_{(1)}} 
		=\Hol_{(\rho_{1}\circ \rho_{2})\circ\rho_{3}}^{ \beta}
	\end{align*}
	Here we see that applying the definition for the two decompositions $\rho_{1}\circ\left(\rho_{2}\circ\rho_{3}  \right)$ and $\left( \rho_{1}\circ\rho_{2} \right)\circ\rho_{3}$ leads to the same result. This can be proven analogously for the other cases.
\end{proof}

We also consider symmetries on the extended space $\HSpace$, given by the vertex and face operators of a Kitaev model. These are obtained as holonomies along special paths, namely the vertex and face paths  from Definition \ref{def:ribbon path} and Figure \ref{fig:paths}.
\begin{definition}
	\label{definition:VertexFaceOperators}
	Let $s$ be a site with vertex path $v$ 
	and face path $f$ and let $h\in H,\alpha\in H^{*}$. The vertex operator at $s$ is the map
	\begin{align}
		A_{s}^{h}:\, &\HSpace \to \HSpace, A_{s}^{h} = 	\Hol_{v}^{h \oo \varepsilon} 
		\label{eq:VertexOperatorStandard}
	\end{align}
	The face operator at $s$ is the map
	\begin{align}
		B_{s}^{\alpha}:\, &\HSpace \to \HSpace, B_{s}^{\alpha} = \Hol_{f}^{1 \oo \alpha} 
		\label{eq:FaceOperatorStandard}
	\end{align}
\end{definition}

It was shown in \cite{BMCA} that the commutation relation between the vertex and face operators associated with a given site form a representation of the Drinfel'd double  $D(H)$ and that the $D(H)$-actions at disjoint sites commute:

\begin{theorem}\label{ht:siterep}\cite{Ki,BMCA}
		Let $s$ be a site. Then
		\begin{align}
			D(H) \oo \HSpace &\to \HSpace\nonumber\\
			(\alpha \oo h) \oo m &\mapsto B^{\alpha}_{s}A^{h}_{s}(m)
			\label{eq:Drinfel'dActionHSpace}
		\end{align}
		defines an action of $D(H)$ on $\HSpace$. The actions of $D(H)$ for disjoint sites commute.
		\label{theorem:Drinfel'dActionHSpace}
	\end{theorem}

	The invariants of these actions are of special interest, as the \emph{protected space}
	\begin{align}
		\mathcal{L} := \left\{ m \in H^{ \oo E} \;|\; B^{\alpha}_{s}A^{h}_{s}m = \alpha(1) \varepsilon(h) \text{ for all sites }s,\alpha \oo h\in D(H)  \right\}.
		\label{eq:ProtectedSpace}
	\end{align}
	is a topological invariant, as shown in~\cite{Ki,BMCA}. 
	For a given site $s$, the map
	\begin{align*}
		P_{s}:=B^{\smallint}_{s}A^{\lambda}_{s}:\HSpace\to\HSpace
	\end{align*}
	is a projector onto these invariants \cite{Ki,BMCA}.
	Here $\smallint \in H^{*}, \lambda\in H$ are the Haar integrals of $H^{*}$ and $H$. 
	Multiplying the $P_{s}$ for all sites $s$ we obtain a projector onto the protected space:
	\begin{theorem}\cite{BMCA,Me}
		For any two sites $s,t$, the projectors $P_{s}$ and $P_{t}$ commute. The map
		\begin{align}
			P:= \prod_{s\text{ site}} P_{s}:\HSpace \to \mathcal{L}
			\label{eq:ProjectorProtected}
		\end{align}
		is a projector onto the protected space.
		\label{theorem:ProjectorProtected}
	\end{theorem}
\subsection{Excitations}
Recall that $H$ semi-simple implies that $D(H)$ also is semi-simple and that vertex and face operators at each site $s$ define a $D(H)$-module structure on $\HSpace$. This allows us to decompose the extended space into its isotypic components
	\begin{align}
		\HSpace= \bigoplus_{i \in \mathrm{Irr}(D(H))} \HSpace(s,i)
		\label{eq:HilbertSpaceDecomposition}
	\end{align}
	where $\mathrm{Irr}(D(H))$ is a set of representatives of irreducible $D(H)$-modules.
	For non-simple $D(H)$-modules $M$ we define the \emph{excitation of type $M$ at $s$} as
	\begin{align}
		\HSpace(s,M):=\bigoplus_{i \in \mathrm{Irr}(D(H)), M_{i}\neq 0} \HSpace(s,i)\subset \HSpace,
		\label{eq:SingleExcitationSpace}
	\end{align}
	where $M_i$ denotes the isotypic component of $M$ of type $i$.
	The \emph{trivial excitation} is the excitation corresponding to the trivial $D(H)$-representation.
	Finally, let $s_{1},\dots,s_{n}$ be $n$ disjoint sites and let $M_{1},\dots,M_{n}$ be $D(H)$-modules. We define
	\begin{align}
		\label{eq:MultipleExcitationSpace}
		\HSpace(s_{1},M_{1},\dots,s_{n},M_{n}):= \bigcap_{k=1}^{n} \HSpace(s_{k},M_{k})
	\end{align}
	This space is closely related to the \emph{space of $n$-particle states} in~\cite{Ki} for disjoint sites $s_1,...,s_n$ 
\begin{align}
	\mathcal{L}(s_{1},\dots,s_{n})= \big\{ 
	m\in \HSpace \;|\; &A_{s}^{h}m = \varepsilon(h) m, B^{\alpha}_{s}m = \alpha(1) m \text{ for  } h\in H,\alpha\in H^{*},\nonumber\\
	&s \text{ site disjoint to $s_{1},\dots,s_{n}$}\big\}.
	\label{eq:SpaceParticleStates}
\end{align}

Both $\HSpace(s_{1},M_{1},\dots,s_{n},M_{n})$ and $\mathcal{L}(s_{1},\dots,s_{n})$ have a $D(H)^{\otimes n}$-module structure from the face and vertex operators at the sites $s_{1},\dots,s_{n}$. There are two differences between these subspaces:
\begin{compactenum}[(i)]
	\item Elements $l\in\mathcal{L}(s_{1},\dots,s_{n}) $ are invariants of the $D(H)$-action defined by sites disjoint to $s_{1},\dots,s_{n}$, while elements $m \in \HSpace(s_{1},M_{1},\dots,s_{n},M_{n}) $ need not be invariant.
	\item The isotypic decomposition of $\HSpace(s_{1},M_{1},\dots,s_{n},M_{n})$ for the $D(H)$-action defined by a site $s_{i}$ only contains the same irreducible representations as the $D(H)$-module $M_{i}$, possibly with different multiplicities. The isotypic decomposition of $\mathcal{L}(s_{1},\dots,s_{n})$ may contain arbitrary irreducible representations. 
\end{compactenum}

In the following we will be interested in excitations of a type $M$ at a site $s$ and thus will only consider spaces of the form $\HSpace(s_{1},M_{1},\dots,s_{n},M_{n})$.

\subsection{Holonomies and excitations}
\label{subsection:HolonomiesExcitationClassic}
We now explain how to generate excitations of a specific type $i\in \mathrm{Irr}(D(H))$ at a site $s$. Our goal is to obtain linear endomorphisms of the extended space $\HSpace$, which restrict to maps $\mathcal{L} \to \HSpace(s,i)$. Examples for such linear endomorphisms are given by \emph{holonomies}.

The commutation relation for the holonomies of two paths $\rho_{1}, \rho_{2}$ depends on the relative constellation of the paths. The following technical lemma describes the commutation relation for several constellations. We will use this lemma in the remainder of the section to describe how holonomies generate excitations. 
The first four identities generalize the relations~(B17),(B18),(B23) from \cite{BD}.
\begin{lemma}
	\label{lemma:HolonomyCommutators}
	Let $\alpha,\beta\in D(H)^{*}$, $\rho,\rho_{1},\rho_{2},\gamma,\gamma_{1},\gamma_{2}$ be simple paths such that $\rho_{1},\rho_{2}$ end to the right of a site and $\gamma_{1},\gamma_{2}$ start to the left of a site. Denote $R= \varepsilon \oo x \oo X \oo 1\in D(H)\oo D(H)$ the $R$-matrix of $D(H)$.
	Then we have
	\begin{align}
		\label{eq:NonOverlap}
		(\rho,\gamma)_{\emptyset} \Rightarrow &\Hol_{\rho_{1}}^{\alpha} \Hol_{\rho_{2}}^{\beta}  = \Hol_{\rho_{2}}^{\beta} \Hol_{\rho_{1}}^{\alpha} 
		\\
		\label{eq:HolonomyReversal}
		&\Hol_{\rho^{-1}}^{ \alpha} = \Hol_{\rho}^{ S(\alpha)} 
		\\
		\label{eq:LeftRightHolonomy}
		\text{$\rho$ left-right path} \Rightarrow 	
		&\Hol_{\rho}^{\alpha}\Hol_{\rho}^{\beta} =\Hol_{\rho}^{\alpha\beta} 
		\\
		\label{eq:RightLeftHolonomy}
		\text{$\rho$ right-left path} \Rightarrow 	
		&\Hol_{\rho}^{\alpha}\Hol_{\rho}^{\beta} =\Hol_{\rho}^{\beta\alpha} 
		\\
		\label{eq:LeftLeftHolonomy}
		\text{$\rho$ left-left path} \Rightarrow 	
		&\Hol_{\rho}^{\alpha}\Hol_{\rho}^{\beta}
		= \Hol_{\rho}^{\left( R^{(2)} \vartriangleright \beta \right) \cdot \left( R^{(1)} \vartriangleright \alpha \right)}
		= \Hol_{\rho}^{ \alpha \cdot_{\mathcal{H}_{R}(H)} \beta} 
		\\
		\label{eq:RightRightHolonomy}
		\text{$\rho$ right-right path} \Rightarrow 	
		&\Hol_{\rho}^{\alpha}\Hol_{\rho}^{\beta} =\Hol_{\rho}^{\left( R^{(-1)} \vartriangleright \alpha \right) \cdot \left( R^{(-2)} \vartriangleright \beta \right)}
		= \Hol_{\rho}^{ \alpha \cdot_{\overline{\mathcal{H}}_{R}(H)} \beta} 
		\\
		\label{eq:LeftJoint}
		(\rho_{1},\rho_{2})_{\prec} \Rightarrow 
		&\Hol_{\rho_{1}}^{\alpha} \Hol_{\rho_{2}}^{\beta} 
		=\Hol_{\rho_{2}}^{\beta  \vartriangleleft  R^{(-2)}}   \Hol_{\rho_{1}}^{\alpha  \vartriangleleft R^{(-1)}} 
		\\
		\label{eq:RightJoint}
		(\gamma_{1},\gamma_{2})_{\succ} \Rightarrow 
		&\Hol_{\gamma_{1}}^{\alpha} \Hol_{\gamma_{2}}^{ \beta} 
		= 
		\Hol_{\gamma_{2}}^{R^{(2)}  \vartriangleright \beta} 
		\Hol_{\gamma_{1}}^{R^{(1)} \vartriangleright \alpha}
		\\
		\label{eq:MiddleJoint}
		(\rho,\gamma)_{\succ\prec} \Rightarrow 
		&\Hol_{\rho_{1}}^{\alpha} \Hol_{\rho_{2}}^{\beta} 
		= 
		 \Hol_{\rho_{2}}^{\beta} \Hol_{\rho_{1}}^{\alpha}
	\end{align}
\end{lemma}
\begin{proof}
	Proof of~\eqref{eq:NonOverlap}. For an edge $e$, the triangle operators for $e^{\pm s}$ commute with those for $e^{\pm t}$, see Definition~\ref{definition:HolonomyBasic}. The triangle operators for $e^{\pm L}$ commute with the ones for $e^{\pm R}$. Triangle operators of different edges also commute with one another, since they act on different copies of $H$. 
	For $(\rho_{1}, \rho_{2})_{\emptyset}$, the holonomy along $\rho_{1}$ is therefore composed of triangle operators that commute with the triangle operators in the holonomy along $\rho_{2}$.
	\\
	\\
	Proof of~\eqref{eq:HolonomyReversal} by induction over $\rho$. For $\rho=e^{\nu}$ with $e$ an edge and $\nu\in \left\{ \pm s,\pm t,\pm L,\pm R \right\} $ the claim follows from~\eqref{eq:HolonomyOpposites} in Definition~\ref{definition:HolonomyBasic}. 
	For $\rho =\rho_{1}\circ\rho_{2}$ as in Definition~\ref{definition:HolonomyComposite} we either have (i) $(\rho_{1},\rho_{2})_{\emptyset}$, (ii) $(\rho_{2},\rho_{1}^{-1})_{\prec}$, or (iii) $(\rho_{1}^{-1},\rho_{2})_{\prec}$. 
	In case (i) we have
	\begin{align*}
		\Hol_{\rho^{-1}}^{ \alpha} 
		&\overset{\eqref{eq:HolonomyRecursionStandard}}{=} \Hol_{\rho_{2}^{-1}}^{\alpha_{(1)}}  \circ \Hol_{\rho_{1}^{-1}}^{ \alpha_{(2)}} 
		\overset{\mathrm{(I)}}{=} \Hol_{\rho_{2}}^{S(\alpha_{(1)})}  \circ \Hol_{\rho_{1}}^{S( \alpha_{(2)})} 
		=\Hol_{\rho_{2}}^{S(\alpha)_{(2)}}  \circ \Hol_{\rho_{1}}^{S( \alpha)_{(1)}} 
		\\
		&\overset{\eqref{eq:NonOverlap}}{=}  \Hol_{\rho_{1}}^{S( \alpha)_{(1)}}\circ \Hol_{\rho_{2}}^{S(\alpha)_{(2)}} 
		\overset{\eqref{eq:HolonomyRecursionStandard}}{=} \Hol_{\rho}^{ S(\alpha)} 
	\end{align*}
	Here we have used induction in step $\mathrm{(I)}$. For case (ii) note that we need to apply the first case~\eqref{eq:HolonomyRecursionStandard} of the recursive Definition~\ref{definition:HolonomyComposite} for $\Hol_{\rho}$, but the second case~\eqref{eq:HolonomyRecursionLeftJoin} for $\Hol_{\rho^{-1}}$. We then obtain
	\begin{align*}
		\Hol_{\rho^{-1}}^{ \alpha} 
		&\overset{\eqref{eq:HolonomyRecursionLeftJoin}}{=} \Hol_{\rho_{1}^{-1}}^{\alpha_{(2)}}  \circ \Hol_{\rho_{2}^{-1}}^{ \alpha_{(1)}} 
		\overset{\mathrm{(I)}}{=} \Hol_{\rho_{1}}^{S(\alpha_{(2)})}  \circ \Hol_{\rho_{1}}^{S( \alpha_{(1)})} 
		=\Hol_{\rho_{1}}^{S(\alpha)_{(2)}}  \circ \Hol_{\rho_{2}}^{S( \alpha)_{(2)}} 
		\overset{\eqref{eq:HolonomyRecursionStandard}}{=} \Hol_{\rho}^{ S(\alpha)} 
	\end{align*}
	Here we again used induction in step $\mathrm{(I)}$. The third case follows analogously to the second case.
	\\
	\\
	Proof of \eqref{eq:LeftRightHolonomy} to \eqref{eq:RightRightHolonomy} by induction. We have four basic cases, one for each equation:
	\begin{compactenum}
		\item $\rho$ is a left-right ribbon path.
		\item $\rho$ is a right-left ribbon path.
		\item $\rho=e^{d}$ or $e^{-d}$.
		\item $\rho=e^{\overline{d}}$ or $\rho=e^{-\overline{d}}$.
	\end{compactenum}
	For case~1 we write $\rho=\rho_{1}^{\nu_{1}}\circ\cdots \rho_{n}^{\nu_{n}}$ in reduced form with $\rho_{i} \in E$, $\nu_{i} \in \left\{ s,-t,R,-L \right\}$. The triangle operators $\Hol_{\rho_{i}^{\nu_{i}}}$ from~\eqref{eq:HolonomyEndsStandard} to~\eqref{eq:HolonomyOpposites} define actions of $D(H)^{*}$ on $\HSpace$. Since $\rho$ is a ribbon path, the triangle operator for $\rho_{i}^{\nu_{i}}$ commutes with the one for $\rho_{k}^{\nu_{k}}$ for $i\neq k$.
	Combining this with~\eqref{eq:HolonomyRecursionStandard} we conclude that $\Hol_{\rho}$ is a tensor product of the $D(H)^{*}$-actions $\Hol_{\rho_{i}^{\nu_{i}}} $, i.e. a $D(H)^{*}$-action. The proof for case~2 is analogous and uses that the operators $\Hol_{\rho_{i}^{\nu_{i}}}$ for $\nu_{i}\in \left\{ -s,t,-R,L \right\}$ from~\eqref{eq:HolonomyEndsStandard} to~\eqref{eq:HolonomyOpposites} define actions of $D(H)^{*}$ on $\HSpace$.
	For case~3 let $m\in H\subseteq \HSpace$ be the element associated to the edge $e$. Then we have for $\alpha=h \oo \delta\in H \oo H^{*}$:
	\begin{align*}
		\Hol_{e^{d}}^{\alpha} (m)
		&\overset{\eqref{eq:HolonomyRecursionStandard}}{=} 
		\Hol_{e^{t}}^{ h_{(1)} \oo X^{1}\delta_{(1)}X^{2}} \circ \Hol_{e^{R}}^{S(x^{2})h_{(2)}x^{1} \oo \delta_{(2)}} (m)
		\\
		&=
		\langle \varepsilon , S(x^{2})h_{(2)}x^{(1) } \rangle \langle \delta_{(2)} , m_{(2)} \rangle  \Hol_{e^{t}}^{ h_{(1)} \oo X^{1}\delta_{(1)}X^{2}} (m_{(1)})
		\\
		&=
		\langle \delta , m_{(2)} \rangle (hm_{(1)})
	\end{align*}
	where we extend linearly to $\HSpace$.
	Using this identity together with $\beta=k \oo \theta\in H \oo H^{*}$ we obtain
	\begin{align*}
		\Hol_{e^{d}}^{\alpha}\circ \Hol_{e^{d}}^{\beta} (m)
		&=\langle \delta_{(1)} , k_{(2)} \rangle \langle \delta_{(2)} \theta ,m_{(2)}  \rangle hk_{(1)}m_{(1)}
		=
		\langle \delta_{(1)} , k_{(2)} \rangle 
		\Hol_{e^{d}}^{hk_{(1)} \oo \delta_{(2)}\theta} (m)
		\\
		&\overset{\eqref{eq:RightHeisenbergDoubleRTwist}}{=}
		\Hol_{e^{d}}^{\left( R^{(2)} \vartriangleright\beta \right)\cdot \left( R^{(1)} \vartriangleright\alpha \right)}(m)
	\end{align*}
	The proof is analogous for $e^{-d}$, $ e^{\pm d'}$ and~case 4.
	\\
	Now let $\rho$ be a simple 
	left-right path that is not a ribbon path. 
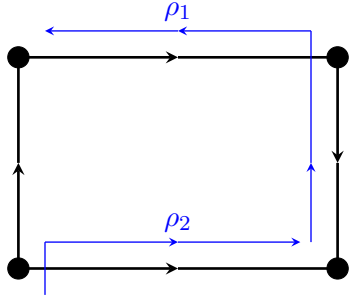
\begin{figure}[H]
\begin{center}
\begin{tikzpicture}[scale=.7]
\draw[line width=1pt, color=black, ->, >=stealth] (-3,0)--(0,0);
\draw[line width=1pt, color=black, ] (3,0)--(0,0);
\draw[line width=1pt, color=black, ->, >=stealth] (-3,4)--(0,4);
\draw[line width=1pt, color=black, ] (3,4)--(0,4);
\draw[line width=1pt, color=black, ->, >=stealth] (-3,0)--(-3,2);
\draw[line width=1pt, color=black, ] (-3,4)--(-3,2);
\draw[line width=1pt, color=black, ] (3,0)--(3,2);
\draw[line width=1pt, color=black, ->, >=stealth ] (3,4)--(3,2);
\draw[color=black, fill=black] (-3,0) circle (.2);
\draw[color=black, fill=black] (3,0) circle (.2);
\draw[color=black, fill=black] (-3,4) circle (.2);
\draw[color=black, fill=black] (3,4) circle (.2);
\draw[line width=.5, color=blue] (-2.5,-.5)--(-2.5,.5);
\draw[line width=.5, color=blue, ->, >=stealth] (-2.5,.5)--(0,.5) node[anchor=south]{$\rho_2$};
\draw[line width=.5, color=blue, ->, >=stealth] (0,.5)--(2.3,.5);
\draw[line width=.5, color=blue, ->, >=stealth] (2.5,.5)--(2.5, 2);
\draw[line width=.5, color=blue,] (2.5,2)--(2.5,4.5);
\draw[line width=.5, color=blue, ->, >=stealth] (2.5,4.5)--(0, 4.5) node[anchor=south]{$\rho_1$};
\draw[line width=.5, color=blue, ->, >=stealth] (0,4.5)--(-2.5, 4.5);
\end{tikzpicture}
\end{center}
\caption{Splitting a left-right path $\rho=\rho_1\circ \rho_2$ into a right-right path $\rho_1$ and a left-left path $\rho_2$.}
\label{fig:leftrightsplit}
\end{figure}
	Then we can write $\rho=\rho_{1}\circ\rho_{2}$ with $(\rho_{1},\rho_{2})_{\emptyset}$, $\rho_{1}$ a nonempty right-right path and $\rho_{2}$ a nonempty left-left path, as shown in Figure \ref{fig:leftrightsplit}. We compute
	\begin{align*}
		\Hol_{\rho}^{ \alpha} \circ \Hol_{\rho}^{ \beta} 
		&\overset{\eqref{eq:HolonomyRecursionStandard}}{=} 
		\Hol_{\rho_{1}}^{ \alpha_{(1)}} \circ \Hol_{\rho_{2}}^{ \alpha_{(2)}}  
		\circ \Hol_{\rho_{1}}^{ \beta_{(1)}} \circ \Hol_{\rho_{2}}^{ \alpha_{(2)}} 
		\overset{\eqref{eq:NonOverlap}}{=}
		\Hol_{\rho_{1}}^{ \alpha_{(1)}}\circ \Hol_{\rho_{1}}^{ \beta_{(1)}}
		\circ \Hol_{\rho_{2}}^{ \alpha_{(2)}}  \circ \Hol_{\rho_{2}}^{ \alpha_{(2)}} 
		\\
		&\overset{\eqref{eq:RightRightHolonomy},\eqref{eq:LeftLeftHolonomy}}{=}
		\Hol_{\rho_{1}}^{ \alpha_{(1)} \cdot_{\overline{\mathcal{H}}_{R}(H)}\beta_{(1)}}  \circ \Hol_{\rho_{2}}^{ \alpha_{(2)} \cdot_{\mathcal{H}_{R}(H)}\beta_{(2)}}
		\overset{\eqref{eq:ComultHeisenbergDoubleAlgebraHom},\eqref{eq:HolonomyRecursionStandard}}{=}
		\Hol_{\rho}^{ \alpha \cdot_{D(H)^{*}} \beta} 
	\end{align*}
	For right-left paths we proceed analogously and split $\rho$ into a left-left path $\rho_{1}$ and a right-right path $\rho_{2}$. Here we use~\eqref{eq:ComultHeisenbergDoubleOppositeAlgebraHom} instead of~\eqref{eq:ComultHeisenbergDoubleAlgebraHom}. 
	If $\rho$ is a left-left path which is not of the form $e^{d}$ or $e^{-d}$, we can write $\rho=\rho_{1}\circ\rho_{2}$ such that $(\rho_{1},\rho_{2})_{\emptyset}$ and either
	\begin{itemize}
		\item $\rho_{1}$ is a left-left path and $\rho_{2}$ is a left-right path, or
		\item $\rho_{1}$ is a left-right path and $\rho_{2}$ is a left-left path.
	\end{itemize}
	The calculation is again analogous, but we use Remark~\ref{remark:HeisenbergDoubleComoduleAlgebra}.\ref{remarkPoint:RightHeisenbergComoduleAlgebra} instead of~\eqref{eq:ComultHeisenbergDoubleAlgebraHom}. The proof for right-right paths follows analogously from Remark~\ref{remark:HeisenbergDoubleComoduleAlgebra}.\ref{remarkPoint:RightHeisenbergAlternativeComoduleAlgebra}.
	\\
	\\
	Proof of~\eqref{eq:LeftJoint}: Let $(\rho_{1},\rho_{2})_{\prec}$ with $\rho_{1}=\rho\circ e^{a}\circ\rho_{1}'$ and $\rho_{2}=\rho\circ e^{b}\circ\rho_{2}'$ such that $\rho$ is ribbon and ends to the right of its target. 
	W.l.o.g. we can assume $e^{a}\in\left\{ e^{R}, e^{R}\circ e^{-s}, e^{-t}\circ e^{L} \right\}$, since we can reverse the orientation of $e$. The possible combinations of $a$ and $b$ then are:
	\begin{align*}
		(e^{a},e^{b})\in \left\{ (e^{R},e^{\overline{d}}), (e^{R}, e^{-t}), (e^{\overline{d}}, e^{-t}) \right\}
	\end{align*}
	In all of these cases, direct calculations show
	\begin{align}
		\Hol_{e^{a}}^{h \oo \alpha} \circ \Hol_{e^{b}}^{k \oo \beta}  
		&= \langle \alpha_{(1)} , S(k_{(2)}) \rangle \Hol_{e^{b}}^{k_{(4)} \oo \beta} \Hol_{e^{a}}^{k_{(1)}hS(k_{(3)}) \oo \alpha_{(2)}} 
		\nonumber
		\\
		&= \Hol_{e^{b}}^{ \left( k \oo \beta \right) \vartriangleleft X \oo 1} \Hol_{e^{a}}^{ \left( h \oo \alpha \right) \vartriangleleft \varepsilon \oo S(x)} 
		\label{eq:LeftJointBasic}
	\end{align}
	We now consider the tensor product $D(H)^{*} \oo D(H)^{*}$ of $D(H)^{*}$ as $D(H)^{*}$-left comodule with the left regular coaction from Example \ref{definition:CoregularAction}.
	As $R^{(1)} \oo R^{(2)}:=\varepsilon \oo x \oo X \oo 1\in D(H) \oo D(H)$ is the $R$-matrix for $D(H)$ from \eqref{eq:Rmatrix}, the braiding
	\begin{align*}
		c: D(H)^{*} \oo D(H)^{*}&\to D(H)^{*} \oo D(H)^{*}\\
		\alpha \oo \beta &\mapsto \left( \beta  \vartriangleleft R^{(2)} \right) \oo \left( \alpha   \vartriangleleft S(R^{(1)}) \right)
	\end{align*}
	is a $D(H)^{*}$-comodule morphism. Concretely, this means that for $\alpha,\beta\in D(H)^{*}$
	\begin{align}
		&\alpha_{(1)}\beta_{(1)} \oo \left( \beta_{(2)}  \vartriangleleft R^{(2)} \right) \oo \left( \alpha_{(2)}   \vartriangleleft S(R^{(1)}) \right)
		\nonumber
		\\&\;=\left(\left( \beta  \vartriangleleft R^{(2)} \right)_{(1)} \cdot \left( \alpha   \vartriangleleft S(R^{(1)}) \right)_{(1)}\right)
		\oo \left( \beta  \vartriangleleft R^{(2)} \right)_{(2)} \oo \left( \alpha   \vartriangleleft S(R^{(1)}) \right)_{(2)}
		\label{eq:BraidingDH*Comodule}
	\end{align}
	For the paths $\rho\circ e^{a}, \rho\circ e^{b}$ and $\alpha,\beta\in D(H)^{*}$ we now compute
	\begin{align}
		\Hol_{\rho \circ e^{a}}^{\alpha} \circ \Hol_{\rho\circ e^{b}}^{\beta}
		&\overset{\eqref{eq:HolonomyRecursionStandard}}{=}
		\Hol_{\rho}^{\alpha_{(1)}} \Hol_{e^{a}}^{\alpha_{(2)}} 
		\Hol_{\rho}^{\beta_{(1)}} \Hol_{e^{b}}^{\beta_{(2)}} 
		\overset{\eqref{eq:NonOverlap}}{=}
		\Hol_{\rho}^{\alpha_{(1)}}\Hol_{\rho}^{\beta_{(1)}}
		\Hol_{e^{a}}^{\alpha_{(2)}} 
		\Hol_{e^{b}}^{\beta_{(2)}} \nonumber
		\\&\overset{\eqref{eq:LeftRightHolonomy}}{=}
		\Hol_{\rho}^{\alpha_{(1)}\beta_{(1)}}
		\Hol_{e^{a}}^{\alpha_{(2)}}
		\Hol_{e^{b}}^{\beta_{(2)}} 
		\overset{\eqref{eq:LeftJointBasic}}{=}
		\Hol_{\rho}^{\alpha_{(1)}\beta_{(1)}}
		\Hol_{e^{b}}^{\beta_{(2)} \vartriangleleft R^{(2)}} 
		\Hol_{e^{a}}^{\alpha_{(2)} \vartriangleleft S(R^{(1)})}\nonumber
		\\&\overset{\eqref{eq:BraidingDH*Comodule}}{=}
		\Hol_{\rho}^{\left( \beta  \vartriangleleft R^{(2)} \right)_{(1)} \cdot \left( \alpha   \vartriangleleft S(R^{(1)}) \right)_{(1)}}
		\Hol_{e^{b}}^{\left( \beta  \vartriangleleft R^{(2)} \right)_{(2)}} 
		\Hol_{e^{a}}^{\left( \alpha   \vartriangleleft S(R^{(1)}) \right)_{(2)}}\nonumber
		\\&\overset{\eqref{eq:LeftRightHolonomy}}{=}
		\Hol_{\rho}^{\left( \beta  \vartriangleleft R^{(2)} \right)_{(1)}}
		\Hol_{\rho}^{\left( \alpha   \vartriangleleft S(R^{(1)}) \right)_{(1)}}
		\Hol_{e^{b}}^{\left( \beta  \vartriangleleft R^{(2)} \right)_{(2)}} 
		\Hol_{e^{a}}^{\left( \alpha   \vartriangleleft S(R^{(1)}) \right)_{(2)}}\nonumber
		\\
		&
		\overset{\eqref{eq:NonOverlap}}{=}
		\Hol_{\rho}^{\left( \beta  \vartriangleleft R^{(2)} \right)_{(1)}}
		\Hol_{e^{b}}^{\left( \beta  \vartriangleleft R^{(2)} \right)_{(2)}} 
		\Hol_{\rho}^{\left( \alpha   \vartriangleleft S(R^{(1)}) \right)_{(1)}}
		\Hol_{e^{a}}^{\left( \alpha   \vartriangleleft S(R^{(1)}) \right)_{(2)}}\nonumber
		\\
		&\overset{\eqref{eq:HolonomyRecursionStandard}}{=}
		\Hol_{\rho\circ e^{b}}^{\beta  \vartriangleleft R^{(2)}}
		\Hol_{\rho\circ e^{a}}^{\alpha   \vartriangleleft S(R^{(1)}) }
		\label{eq:LeftJointSecond}
		\end{align}
		In the last step we also used the following identity for $g\in D(H), \alpha\in D(H)^{*}$		
		\begin{align}
			\left( \alpha_{(1)}  \vartriangleleft g \right)
			\oo \alpha_{(2)}
			= \left( \alpha  \vartriangleleft g \right)_{(1)} \oo \left( \alpha  \vartriangleleft g \right)_{(2)}
			\label{eq:RightCoactionLeftActionCommute}
		\end{align}
		Finally, we obtain for the paths $\rho\circ e^{a}\circ \rho_{1}',\rho\circ e^{b}\circ \rho_{2}'$:
		\begin{align*}
			\Hol_{\rho\circ e^{a}\circ \rho_{1}'}^{\alpha} \Hol_{\rho\circ e^{b}\circ \rho_{2}'}^{ \beta} 
			&=
			\Hol_{\rho\circ e^{a}}^{\alpha_{(1)}} \Hol_{\rho_{1}'}^{\alpha_{(2)}} 
			\Hol_{\rho\circ e^{b}}^{\beta_{(1)}} \Hol_{\rho_{2}'}^{\beta_{(2)}} 
			\overset{\eqref{eq:NonOverlap}}{=}
			\Hol_{\rho\circ e^{a}}^{\alpha_{(1)}} 
			\Hol_{\rho\circ e^{b}}^{\beta_{(1)}}
			\Hol_{\rho_{2}'}^{\beta_{(2)}}
			\Hol_{\rho_{1}'}^{\alpha_{(2)}} 
			\\&\overset{\eqref{eq:LeftJointSecond}}{=}
			\Hol_{\rho\circ e^{b}}^{\beta_{(1)}  \vartriangleleft R^{(2)}}
			\Hol_{\rho\circ e^{a}}^{\alpha_{(1)}  \vartriangleleft S(R^{(1)})} 
			\Hol_{\rho_{2}'}^{\beta_{(2)}}
			\Hol_{\rho_{1}'}^{\alpha_{(2)}}
			\\&\overset{\eqref{eq:NonOverlap}}{=}
			\Hol_{\rho\circ e^{b}}^{\beta_{(1)}  \vartriangleleft R^{(2)}}
			\Hol_{\rho_{2}'}^{\beta_{(2)}}
			\Hol_{\rho\circ e^{a}}^{\alpha_{(1)}  \vartriangleleft S(R^{(1)})} 
			\Hol_{\rho_{1}'}^{\alpha_{(2)}}
			\\&\overset{\eqref{eq:RightCoactionLeftActionCommute}}{=}
			\Hol_{\rho\circ e^{b}}^{\left(\beta  \vartriangleleft R^{(2)}\right)_{(1)}}
			\Hol_{\rho_{2}'}^{\left(\beta  \vartriangleleft R^{(2)}\right)_{(2)}}
			\Hol_{\rho\circ e^{a}}^{\left(\alpha  \vartriangleleft S(R^{(1)})\right)_{(1)}} 
			\Hol_{\rho_{1}'}^{\left(\alpha  \vartriangleleft S(R^{(1)})\right)_{(2)}}
			\\
			&\overset{\eqref{eq:HolonomyRecursionStandard}}{=}
			\Hol_{\rho\circ e^{b}\circ \rho_{2}'}^{ \beta  \vartriangleleft R^{(2)} } 
			\Hol_{\rho\circ e^{a} \circ \rho_{1}'}^{\alpha  \vartriangleleft S(R^{(1)})} 
		\end{align*}
		This shows the identity~\eqref{eq:LeftJoint}.
		Equation~\eqref{eq:RightJoint} follows by an analogous proof. 
	\\
	\\
	Proof of \eqref{eq:MiddleJoint}. If $(\rho,\gamma)_{\succ\prec}$, then there are paths $\rho_{1},\rho_{2},\gamma_{1},\gamma_{2}$ such that $(\rho_{1},\gamma_{1})_{\succ}$, $(\rho_{2},\gamma_{2})_{\prec}$ and $\rho_{1}$ and $\gamma_{1}$ are disjoint to $\rho_{2},\gamma_{2}$. The claim follows by writing
	\begin{align*}
		\Hol_{\rho}^{ \alpha} = \Hol_{\rho_{1}}^{ \alpha_{(1)}} \circ \Hol_{\rho_{2}}^{ \alpha_{(2)}},\quad 
		\Hol_{\gamma}^{ \beta} = \Hol_{\rho_{1}}^{ \alpha_{(1)}} \circ \Hol_{\rho_{2}}^{ \alpha_{(2)}}
	\end{align*}
	and using the equations~\eqref{eq:NonOverlap}, \eqref{eq:LeftJoint} and~\eqref{eq:RightJoint}.
	\\
\end{proof}
We  now explain how the holonomy along a ribbon path $\rho:s_{1}\to s_{2}$ creates excitations at the end points of that ribbon path in more detail.
We will first describe the commutation relation of the vertex and face operators at $s_{1}$ and $s_{2}$ with the holonomy along $\rho$,
generalizing the commutation relations~(B41) and~(B42) for Kitaev models based on groups from \cite{BD}. 
\begin{lemma}
	\label{lemma:RibbonHolonomyCommutators}
	Let $\rho:s_{1}^{\eta_{1}}\to s_{2}^{\eta_{2}}$ a simple path between disjoint sites $s_{1}$ and $s_{2}$ and $\eta_{1},\eta_{2}\in \left\{ L,R \right\}$.
	Then the holonomy along $\rho$ satisfies the following commutation relations with the vertex and face operators at $s_{1}$ and $s_{2}$:
	\begin{align}
		&B_{s_{1}}^{\beta} \circ A_{s_{1}}^{k} \circ \Hol_{\rho}^{h \oo \alpha} =  \Hol_{\rho}^{(\beta_{(1)} \oo k_{(2)})  \vartriangleright\left( h \oo \alpha \right)  } \circ B_{s_{1}}^{\beta_{(2)}} \circ A_{s_{1}}^{k_{(1)}}, \quad \text{if $\eta_{1}=R$}
		\label{eq:HolonomyCommutatorsStartRight}
		\\
		&B_{s_{1}}^{\beta} \circ A_{s_{1}}^{k} \circ \Hol_{\rho}^{h \oo \alpha} =  \Hol_{\rho}^{\left( \beta_{(2)} \oo k_{(1)} \right)  \vartriangleright \left( h \oo \alpha \right)} \circ B_{s_{1}}^{\beta_{(1)}} \circ A_{s_{1}}^{k_{(2)}}, \quad \text{if $\eta_{1}=L$}
		\label{eq:HolonomyCommutatorsStartLeft}
		\\
		&B_{s_{2}}^{\beta} \circ A_{s_{2}}^{k} \circ \Hol_{\rho}^{h \oo \alpha} =  \Hol_{\rho}^{  \left( h \oo \alpha \right) \vartriangleleft S\left( \beta_{(2)} \oo k_{(1)} \right)} \circ B_{s_{2}}^{\beta_{(1)}} \circ A_{s_{2}}^{k_{(2)}},\quad \text{if $\eta_{2}=L$}
		\label{eq:HolonomyCommutatorsEndLeft}
		\\
		&B_{s_{2}}^{\beta} \circ A_{s_{2}}^{k} \circ \Hol_{\rho}^{h \oo \alpha} =  \Hol_{\rho}^{\left( h \oo \alpha \right) \vartriangleleft S(\beta_{(1)} \oo k_{(2)})} \circ B_{s_{2}}^{\beta_{(2)}} \circ A_{s_{2}}^{k_{(1)}},\quad \text{if $\eta_{2}=R$}
		\label{eq:HolonomyCommutatorsEndRight}
	\end{align}
\end{lemma}
\begin{proof}
	We only show~\eqref{eq:HolonomyCommutatorsEndRight}, the rest follows analogously. First note, that we can prove the identity separately for $\beta=\varepsilon$ and $k\in H$ and for $k=1$ and $\beta\in H^{*}$. 
	Let $v$ be the vertex path of $s_{2}$. Then we have $(\rho,v^{-1})_{\prec}$. By writing $A_{s_{2}^{k}}= \Hol_{v}^{k \oo \varepsilon} $ we obtain
	\begin{align}
		\Hol_{v}^{k \oo \varepsilon} \Hol_{\rho}^{h \oo \alpha} 
		&\overset{\eqref{eq:HolonomyReversal}}{=}
		\Hol_{v^{-1}}^{S(k) \oo \varepsilon} \Hol_{\rho}^{h \oo \alpha} 
		\overset{\eqref{eq:LeftJoint}}{=} \Hol_{\rho}^{\left(h \oo \alpha\right) \vartriangleleft R^{(1)}} 
		\Hol_{v^{-1}}^{\left(S(k) \oo \varepsilon\right)  \vartriangleleft R^{(2)}} 
		\\
		&\overset{\eqref{eq:HolonomyReversal}}{=} \Hol_{\rho}^{\left(h \oo \alpha\right) \vartriangleleft R^{(1)}} \Hol_{v}^{S(R^{(2)})  \vartriangleright \left(k \oo \varepsilon\right)} 
		\label{eq:1}
	\end{align}
	We now compute
	\begin{align*}
		\left(\left( h \oo \alpha \right)  \vartriangleleft R^{(1)}\right) 
		\oo \left( S(R^{(2)})  \vartriangleright (k \oo \varepsilon)\right)
		&=
		\left(\left( h \oo \alpha \right)  \vartriangleleft \varepsilon \oo x\right) 
		\oo 
		\langle S(X) \oo 1 , k_{(2)} \oo \varepsilon \rangle 
		\left(k_{(1)} \oo \varepsilon\right)
		\\
		&=
		\left(\left( h \oo \alpha \right)  \vartriangleleft \varepsilon \oo S(k_{(2)})\right) 
		\oo 
		\left(k_{(1)} \oo \varepsilon\right)
	\end{align*}
	By inserting this into~\eqref{eq:1} we obtain~\eqref{eq:HolonomyCommutatorsEndRight} for $\beta=\varepsilon$. The result for $k=1, \beta\in H^{*}$ follows analogously by writing $B_{s_{2}}^{\beta}= \Hol_{f}^{1 \oo \beta} $, where $f$ is the face path of $s_{2}$ and using $(f,\rho)_{\prec}$.
\end{proof}

\begin{corollary}
	Let $\rho:s_{1}\to s_{2}$ a simple path between disjoint vertices. Then $\Hol_{\rho}$ defines module homomorphisms
\begin{align}
	D(H)^{*}_{ \vartriangleright}  &\oo \HSpace_{s_{1}} \to \HSpace_{s_{1}} \quad&&\text{if $\rho$ starts to the left of $s_{1}$}
	\label{eq:HolonomyModuleLeftStart}
	\\
	D(H)^{*}_{ \vartriangleright} &\oo^{cop} \HSpace_{s_{1}}  \to \HSpace_{s_{1}}\quad&&\text{if $\rho$ starts to the right of $s_{1}$}\\
	D(H)^{*}_{ \vartriangleleft}  &\oo \HSpace_{s_{2}} \to \HSpace_{s_{2}}\quad&&\text{if $\rho$ ends to the left of $s_{2}$}\\
	D(H)^{*}_{ \vartriangleleft} &\oo^{cop}  \HSpace_{s_{2}} \to \HSpace_{s_{2}}\quad&&\text{if $\rho$ ends to the right of $s_{2}$}
	\label{eq:HolonomyModuleRightEnd}
\end{align}
Here $ \oo $ is the tensor product of $D(H)$-modules and $ \oo^{cop} $ the tensor product of $D(H)^{cop}$-modules, $\HSpace_{s_{i}}$ is the extended space with the $D(H)$-action defined by the site $s_{i}$, $D(H)^{*}_{ \vartriangleright}$ is the vector space $D(H)^{*}$ with the left regular action of $D(H)$ and $D(H)^{*}_{ \vartriangleleft}$ is the vector space $D(H)^{*}$ with the left action of $D(H)$ defined by~\eqref{eq:RightCoregularLeftAction}.
\label{corollary:HolonomyModuleHom}
\end{corollary}
\begin{proof}
	We show the claim~\eqref{eq:HolonomyModuleLeftStart}. The action of $\beta \oo k\in D(H)$ on $(h \oo \alpha) \oo m\in D(H)^{*}_{ \vartriangleright} \oo \HSpace_{s_{1}}$ is given by
	\begin{align*}
		\left( \beta_{(2)} \oo k_{(1)}  \vartriangleright\left( h \oo \alpha \right) \right)  \oo \left( B_{s_{1}}^{\beta_{(1)}} \circ A_{s_{1}}^{k_{(2)}} (m)  \right).
	\end{align*}
	and the action of $D(H)$ on $m\in \HSpace_{s_{1}}$ is given by
	\begin{align*}
		B_{s_{1}}^{\beta}\circ A_{s_{1}}^{k} (m)
	\end{align*}
	The claim that $\Hol_{\rho}$ is a module homomorphism then simply translates to the identity~\eqref{eq:HolonomyCommutatorsStartLeft}.
	The other three claims follow analogously from Lemma~\ref{lemma:RibbonHolonomyCommutators}.
\end{proof}

We now use this corollary to show how holonomies generate and fuse excitations. 
We first note that Theorem of Artin-Wedderburn (Theorem~\ref{theorem:ArtinWedderburn}) implies, that the dual $D(H)^{*}$ of the Drinfel'd double with left and right regular action decomposes as a $D(H)$-bimodule into
\begin{align}
	\varphi: D(H)^{*} \xrightarrow{\sim} \bigoplus_{d \in \mathrm{Irr}\left( D(H) \right))} d \oo d^{*}.
	\label{eq:Drinfel'dDualArtinWedderburn}
\end{align}
Here $d$ runs through a set of representatives of irreducible $D(H)$-modules.
Combining this with Corollary~\ref{corollary:HolonomyModuleHom}, we obtain the following statement, illustrated by Figure~\ref{figure:HolonomyFusion}.
\begin{corollary}
	\label{corollary:HolonomyFusion}
	Let $\rho: s_{1} \to s_{2}$ a simple left-right path, $M_{1}, M_{2}$ $D(H)-modules$, $d$ a simple $D(H)$-module and $h \oo \alpha \in D(H)^{*}$ in the $d \oo d^{*}$-component of the decomposition in~\eqref{eq:Drinfel'dDualArtinWedderburn}. 
	Then $\Hol_{\rho}^{h \oo \alpha} :\HSpace \to \HSpace$ restricts to a map
	\begin{align}
		\Hol_{\rho}^{h \oo \alpha} : \HSpace\left( s_{1},M_{1},s_{2},M_{2} \right) \to \HSpace \left( s_{1},d \oo M_{1}, s_{2}, M_{2} \oo d^{*} \right)
		\label{eq:HolonomyFusion} 
	\end{align}
\end{corollary}

\begin{proof}
	Consider $D(H)^{*}$ with the left coregular action $ \vartriangleright$ of $D(H)$ from~\eqref{eq:LeftCoregularActionDrinfel'd}, $\HSpace$ with the left $D(H)$-action defined by the site $s_{1}$ and $D(H)^{*} \oo \HSpace$ the tensor product of these $D(H)$-modules in $D(H)\mathrm{-Mod}$.
	Then~\eqref{eq:HolonomyModuleLeftStart} states that 
	\begin{align*}
		\Hol_{\rho}: D(H)^{*}_{ \vartriangleright} \oo \HSpace_{s_{1}} \to \HSpace_{s_{1}}
	\end{align*}
	is a homomorphism of $D(H)$-modules. 
	After restricting $\Hol_{\rho}$ to the submodule $d \oo d^{*}\subseteq D(H)^{*}$ with the $D(H)$-left action on $d$ and the submodule $\HSpace(s_{1},M_{1})\subseteq \HSpace$, we obtain the corestriction 
	\begin{align*}
		\Hol_{\rho} : (d \oo d^{*})  \oo  \HSpace(s_{1} ,M_{1}) \to \HSpace(s_{1}, d \oo M_{1}).
	\end{align*}
	Now consider instead the $D(H)$-action on $D(H)^{*}$ from~\eqref{eq:RightCoregularLeftAction} and the $D(H)$-action on $\HSpace$ defined by the site $s_{2}$. Again $\HSpace  \oo D(H)^{*}$ is the tensor product of these $D(H)$-modules in $D(H)\mathrm{-Mod}$.
	Then~\eqref{eq:HolonomyModuleRightEnd} states that 
	\begin{align*}
		\Hol_{\rho}:  \HSpace \oo D(H)^{*} \to \HSpace
	\end{align*}
	is a homomorphism of $D(H)$-modules.
	After restricting $\Hol_{\rho}$ to the submodule $d \oo d^{*}\subseteq D(H)^{*}$ with the $D(H)$-left action on $d^{*}$ and the submodule $\HSpace(s_{2},M_{2})\subseteq \HSpace$, we obtain the corestriction 
	\begin{align*}
		\Hol_{\rho} :  \HSpace(s_{2} ,M_{2}  \oo (d \oo d^{*})   ) \to \HSpace(s_{2}, M_{2} \oo d^{*}).
	\end{align*}
	Combining these two restrictions and inserting $h \oo \alpha \in d  \oo d^{*} \subseteq D(H)^{*}$, we obtain~\eqref{eq:HolonomyFusion}.
\end{proof}

\begin{figure}[H]
	\centering
	\scalebox{0.5}{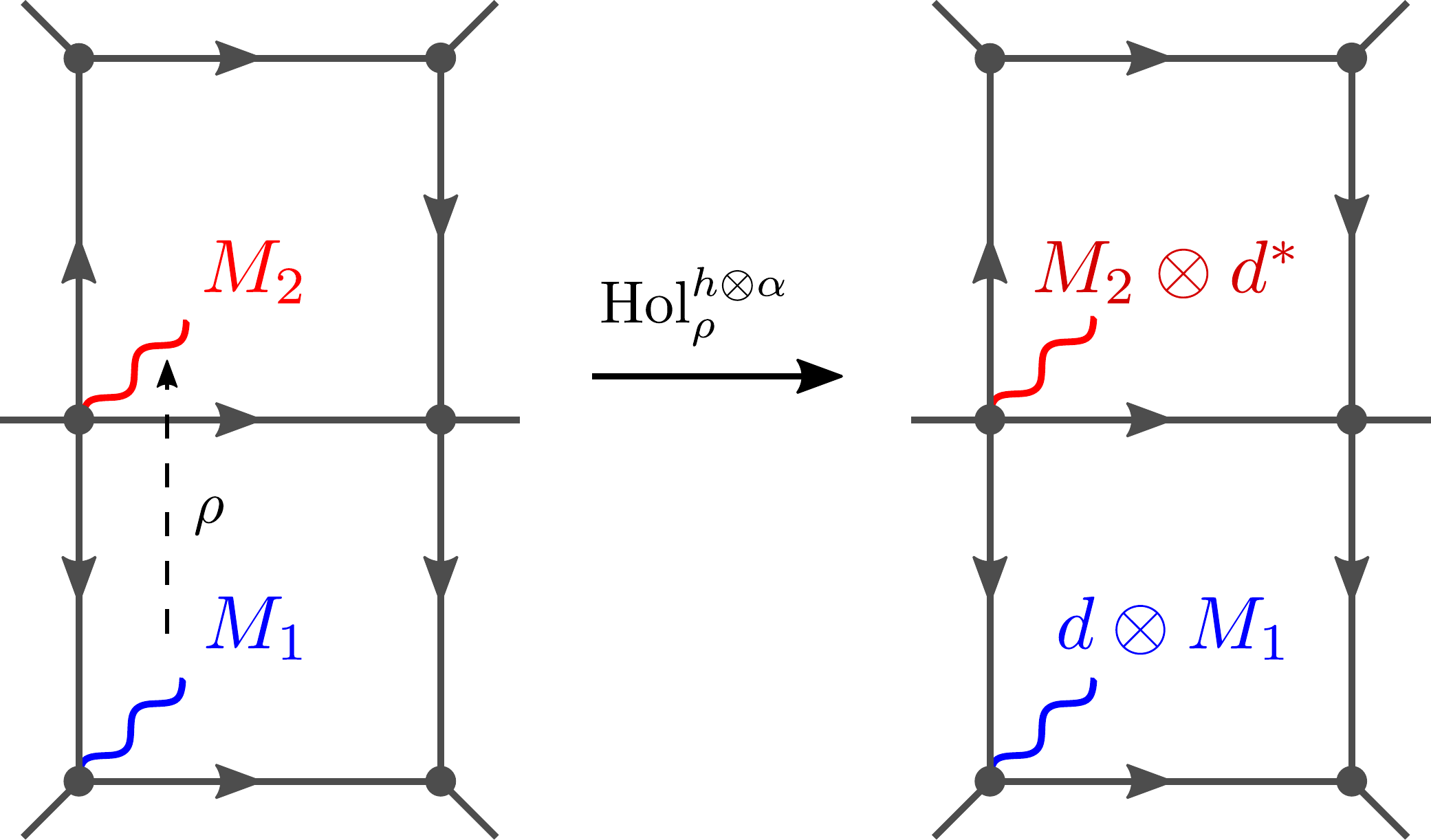}
	\caption{The holonomy $\Hol_{\rho}^{h \oo \alpha} $ along a left-right path $\rho$ for $h \oo \alpha \in d \oo d^{*} \subseteq D(H)^{*}$ creates excitations of type $d$ and $d^{*}$ at the endpoints of $\rho$ and fuses them with the preexisting excitations at these points.}
	\label{figure:HolonomyFusion}
\end{figure}


\section{Turaev-Viro-TQFTs with topological boundaries and topological defects}

\label{section:FSVSummary}

It was shown in~\cite{BK12} that the Kitaev model for a semisimple, finite-dimensional Hopf algebra $H$ reproduces the two-dimensional parts of the Turaev-Viro-TQFT based on the finitely semisimple fusion category $H\mathrm{-Mod}$. 
If one starts with a Turaev-Viro-TQFT based on a finitely semisimple fusion category $\mathcal{A}$ one can use a fiber functor on $\mathcal{A}$ to construct a Hopf algebra $H$ with $\mathcal{A} \cong H\mathrm{-Mod}$. This Tannaka duality relates the input data for Kitaev models and Turaev-Viro-TQFTs.
In this section we summarize the categorical data for Turaev-Viro-TQFTs with topological boundaries and topological surface defects from \cite{FSV}. We then equip this categorical data with fiber functors to obtain corresponding Hopf-algebraic data. In Section \ref{sec:defect model} we will use this Hopf algebraic data to define a Kitaev model with topological defects and topological boundaries. 

\begin{definition}
	\label{definition:TVTQFTDefect}
	A Turaev-Viro-TQFT with topological boundary conditions and topological surface defects as defined in \cite{FSV} consists of the following data:
	\begin{itemize}
		\item The center $\mathcal{Z}(\mathcal{A})$ of a fusion category $\mathcal{A}$ for every three-dimensional region.
		\item A fusion category $\mathcal{W}_{a}$ (resp. $\mathcal{W}_{d}$) for every topological boundary $a$ (resp. topological surface defect $d$).
		\item A braided equivalence $\widetilde{F}_{\to a}:\mathcal{Z}(\mathcal{A}) \to \mathcal{Z}(\mathcal{W}_{a})$ for every pair of a topological boundary labeled with $\mathcal{W}_{a}$ and an adjacent bulk labeled with $\mathcal{Z(A)}$. 
		\item A braided equivalence $\widetilde{F}_{\to d \leftarrow}: \mathcal{Z}(\mathcal{A}_1)\boxtimes \mathcal{Z}(\mathcal{A}_2)^{rev} \to \mathcal{Z}(\mathcal{W}_{d})$ for every topological surface defect labeled with $\mathcal{W}_{d}$ separating three-dimensional regions labeled with $\mathcal{Z}(\mathcal{A}_{1})$ and $\mathcal{Z}(\mathcal{A}_2)$. The functor $\widetilde{F}_{\to d \leftarrow}$ is a composite of two braided monoidal functors $\widetilde{F}_{\to d}:\mathcal{Z}\left( \mathcal{A}_{1} \right) \to \mathcal{Z}(\mathcal{W}_{d})$ and $\widetilde{F}_{d \leftarrow}:\mathcal{Z}\left( \mathcal{A}_{2} \right)^{rev} \to \mathcal{Z}(\mathcal{W}_{d})$.
	\end{itemize}
\end{definition}
By composing $\widetilde{F}_{\to a}$ with the forgetful functor $\mathcal{Z}(\mathcal{W}_{a})\to \mathcal{W}_{a}$, one obtains a functor $F_{\to a}: \mathcal{Z(A)} \to \mathcal{W}_{a}$. Similarly one obtains functors $F_{\to d \leftarrow}, F_{\to d}, F_{d \leftarrow}$ by composing $\widetilde{F}_{\to d \leftarrow}, \widetilde{F}_{\to d}, \widetilde{F}_{d\leftarrow}$ with the forgetful functor $\mathcal{Z}(\mathcal{W}_{d}) \to \mathcal{W}_{d}$.

A concrete example for the data for a defect and adjacent bulk regions in Definition~\ref{definition:TVTQFTDefect} is given by Hopf algebras:
\begin{example}
	\label{example:FSVCategories}
	Let $H_{1}$ and $H_{2}$ be semisimple finite-dimensional Hopf algebras over $\C$, and $I\in H_{1} \oo H_{2} \oo H_{1} \oo H_{2}$ a twist for $H_{1} \oo H_{2}$. Then $\mathcal{A}_1 = H_{1}\mathrm{-Mod}, \mathcal{A}_2=H_{2}\mathrm{-Mod}$  and $ \mathcal{W}_{d} = \left( H_{1} \oo H_{2} \right)_{I}\mathrm{-Mod}$ are fusion categories. There is a braided equivalence 
	\begin{align}
		\mathcal{Z}(\mathcal{A}_1) \boxtimes \mathcal{Z}(\mathcal{A}_2)^{rev}\cong D(H_{1} \oo H_{2})\mathrm{-Mod} \to D\left( \left( H_{1} \oo H_{2} \right)_{I} \right)\mathrm{-Mod} \cong \mathcal{Z}(\mathcal{W}_{d}).
		\label{eq:FSVExampleTransport}
	\end{align}
\end{example}
In a Turaev-Viro TQFT with topological boundaries and defects, the objects of the categories $ \mathcal{W}_{a},\mathcal{Z(A)}$ and $\mathcal{W}_{d}$ appear as insertions in Wilson lines. The behavior of the insertions when moved, fused and braided was analyzed in~\cite{FSV} to derive the data in~Definition~\ref{definition:TVTQFTDefect}.

These insertions are the counterpart of the excitations in Kitaev models, as shown in~\cite[Theorem~6.1]{BK12}. From now on we therefore refer to them as excitations. We now describe the conditions from~\cite{FSV} for these excitations.

\begin{itemize}
	\item[\textbf{Excitations:}] Excitations in a bulk labeled with $\mathcal{Z(A)}$ are objects of $\mathcal{Z(A)}$. Excitations in a topological surface defect (resp. boundary) labeled with $\mathcal{W}_{d}$ (resp. $\mathcal{W}_{a}$) are objects of $\mathcal{W}_{d}$ (resp. $\mathcal{W}_{a}$).
	\item[\textbf{Fusion in the bulk:}] Excitations can be moved around inside a bulk region. If an excitation $M_{1}\in \mathcal{Z(A)}$ is moved to a spot occupied by an excitation $M_{2}\in \mathcal{Z(A)}$, then $M_{1}$ and $M_{2}$ are fused to an excitation of type $M_{1} \oo M_{2}$ or to an excitation of type $M_{2} \oo M_{1}$. The first occurs if $M_{1}$ is moved to the left of $M_{2}$, the second if $M_{1}$ is moved to the right of $M_{2}$. Analogously, excitations can also be moved around inside a topological surface defect and inside a topological boundary and analogous rules apply.
	\item[\textbf{Fusion in the boundary:}] Moving an excitation $M\in \mathcal{Z(A)}$ from the bulk region labeled with $\mathcal{Z(A)}$ to an adjacent boundary labeled with $\mathcal{W_{a}}$ turns $M$ into a boundary excitation $F_{\to a}M$.
		\\
		Fusing two excitations $M_{1},M_{2}\in \mathcal{Z(A)}$ in the bulk region and transporting them to the boundary results in a boundary excitation $F_{\to a}\left( M_{1} \oo M_{2} \right)$. Moving $M_{1}$ and $M_{2}$ to the boundary separately and fusing them there instead results in an excitation of type $F_{\to a}M_{1}  \oo F_{\to a}M_2$. These two objects are related by the natural isomorphism $F_{\to a}\left( M_{1} \oo M_{2} \right) \cong F_{\to a}M_{1}  \oo F_{\to a} M_2$ that is coherence data of the the monoidal functor $F_{\to a}$.
	\item[\textbf{Fusion in the defect:}] Moving an excitation $M\in \mathcal{Z}(\mathcal{A}_{1})$ from the bulk region labeled with $\mathcal{Z}(\mathcal{A}_{1})$ to the topological defect labeled with $\mathcal{W}_{d}$ is subject to analogous rules as moving an excitation from a bulk region into a boundary.
		Moving $M$ into the defect turns $M$ into $F_{\to d}M$.
	Again, one can fuse two excitations $M_{1},M_{2}\in \mathcal{Z}(\mathcal{A}_{1})$ in the bulk region and transport them to the defect results in a defect excitation $F_{\to d}\left( M_{1} \oo M_{2} \right)$.
	Instead moving $M_{1}$ and $M_{2}$ to the defect separately and fusing them there results in an excitation of type $F_{\to d}M_{1}  \oo F_{\to d}M_2$. These two procedures only differ up to the natural isomorphism $F_{\to d}\left( M_{1} \oo M_{2} \right) \cong F_{\to d}M_{1}  \oo F_{\to d} M_2$ given by the monoidal functor $F_{\to d}$. 
	Moving an excitation $N\in \mathcal{Z}(\mathcal{A}_{2})$ from the bulk region labeled with $\mathcal{Z}(\mathcal{A}_{2})$ to the topological defect labeled with $\mathcal{W}_{d}$ is subject to same rules, where one replaces the monoidal functor $F_{\to d}$ with $F_{d \leftarrow}$.
	\item[\textbf{Braiding in the bulk:}] Let $M_{1},M_{2}\in \mathcal{Z(A)}$ be excitations in a bulk region labeled with $\mathcal{Z(A)}$. One can fuse these excitations into $M_{1} \oo M_{2}$ or $M_{2} \oo M_{1}$ by moving $M_{1}$ to $M_{2}$. 
		The order of the tensor product depends on the relative position of the excitations. Moving $M_{1}$ around $M_{2}$ first relates the two products with the braiding $M_{1} \oo M_{2}\to M_{2} \oo M_{1}$ in $\mathcal{Z(A)}$.
	\item[\textbf{Braiding in the boundary:}] Moving an excitation $M\in \mathcal{Z(A)}$ from a bulk region to an excitation $N\in \mathcal{W}_{a}$ in the boundary fuses the excitations to $F_{\to a}M  \oo N$ or $N \oo F_{\to a}M $.  The order of the tensor product depends on the relative position of the excitations. Moving $M$ around $N$ first relates the two products with the half-braiding $M \oo F_{\to a}M \to F_{\to a}M \oo N$ of $\mathcal{Z}(\mathcal{W}_{a})$ with $\mathcal{W}_{a}$.
	\item[\textbf{Braiding in the defect:}] Braiding bulk excitations with defect excitations follows the same rules as braiding bulk excitations with boundary excitations, \emph{mutatis mutandis}. One simply needs to exchange $W_{a}$ with $W_{d}$, $\mathcal{Z}(A)$ with either $\mathcal{Z}(\mathcal{A}_{1})$ or $\mathcal{Z}(\mathcal{A}_{2})^{rev}$ and $F_{\to a}$ with either $F_{\to d}$ or $F_{d \leftarrow}$. \end{itemize}

To construct a Kitaev model with topological boundaries and defects we require the Hopf algebraic counterparts corresponding to the categorical data in Definition~\ref{definition:TVTQFTDefect}.
We determine these counterparts by utilizing Tannaka duality as follows.
We take one fiber functor $\omega_{\mathcal{C}}$ for every fusion category $\mathcal{C}$ labeling a topological boundary, topological defect or three-dimensional region. 
If a bulk region labeled with $\mathcal{A}$ is adjacent to a boundary labeled with $\mathcal{W}_{a}$, we obtain two fiber functor for $\mathcal{Z(A)}$, given by the two sequences
\begin{align*}
	\mathcal{Z(A)} \to \mathcal{A} \overset{\omega_{A}}{\to} \mathrm{Vect}_{\C},
	\qquad \mathcal{Z(A)} \overset{F_{\to a}}{\to} \mathcal{Z}(\mathcal{W}_{a}) \to \mathcal{W}_{a} \overset{\omega_{\mathcal{W}_{a}}}{\to} \mathrm{Vect}_{\C}
\end{align*}
To determine a Hopf-algebraic counterpart to the braided equivalences $F_{\to a}$ we need to impose a compatibility on these fiber functors (cf. Lemma~\ref{lemma:FibreFunctorTensorEquivalence}). Concretely we require that these two functors are naturally isomorphic as tensor functor. 

For defects we similarly obtain two sets of fiber functors
\begin{align*}
	&\mathcal{Z}(\mathcal{A}_{1}) \boxtimes \mathcal{Z}(\mathcal{A}_{2})
	\to \mathcal{A}_{1} \boxtimes \mathcal{A}_{2} 
	\overset{\omega_{\mathcal{A}_{1}}\boxtimes \omega_{\mathcal{A}_{2}} }{\to}
	\mathrm{Vect}_{\C}\boxtimes \mathrm{Vect}_{\C} 
	\overset{\otimes}{\to} \mathrm{Vect}_{\C}
	\\
	&\mathcal{Z}(\mathcal{A}_{1}) \boxtimes \mathcal{Z}(\mathcal{A}_{2})
	\overset{F_{\to d \leftarrow}}{\to}
		\mathcal{W}_{d} 
		\overset{\omega_{\mathcal{W}_{d}}}{\to}
		\mathrm{Vect}_{\C} 
\end{align*}
and again we impose that these are isomorphic as tensor functors.

\begin{proposition}
	The data from Definition~\ref{definition:TVTQFTDefect} is in Tannaka duality with the following Hopf-algebraic data (summarized in the table below the proof):
	\begin{compactenum}
		\item A semisimple, finite-dimensional Hopf algebra $H=\mathrm{End}(\omega_{\mathcal{A}})$ for every three-dimensional region labeled with $\mathcal{A}$. Excitations in this region are modules over $D(H)$.
		\item A semisimple, finite-dimensional Hopf algebra $K=\mathrm{End}(\omega_{\mathcal{W}})$ for every topological boundary or topological surface defect labeled with $\mathcal{W}$. Excitations in the topological boundary or topological surface defect are modules over $K$.
		\item For every pair $(H,K)$ of a topological boundary labeled with $K$ with an adjacent three-dimensional region labeled with $H$: 
			A twist $J$ of $D(H)$ such that $D(H)_{J} \cong D(K)$ as quasitriangular Hopf algebras.
		\item For every triple $(H_{1},H_{2},K)$ of a topological defect surface labeled with $K$ separating two three-dimensional regions labeled with $H_{1}$ and $H_{2}$: A twist $J$ of $D(H) \oo D(K)^{rev}$ such that $\left( D(H_{1}) \oo D(H_{2})^{rev} \right)_{J} \cong D(K)$ as quasitriangular Hopf algebras, where $H^{rev}$ denotes the quasitriangular Hopf algebra $H$ with the opposite $R$-matrix.
	\end{compactenum}
	\label{proposition:FSVKitaevTranslation}	
\end{proposition}

\begin{proof}
	In Lemma~\ref{lemma:FibreFunctorClassic} we have seen that fusion categories correspond to semisimple, finite-dimensional Hopf algebras and that centers of fusion categories correspond to the Drinfel'd doubles of said Hopf algebras.
	In Lemma~\ref{lemma:FibreFunctorTensorEquivalence} we have seen that braided tensor equivalences correspond to braided twist equivalences of Hopf algebras.
	Applying both lemmata to the data for Turaev-Viro-TQFTs with boundary conditions and surface defects yields the proposition.
\end{proof}
\begin{tabular}[center]{|p{3cm}|p{5.5cm}|p{6.5cm}|}
	\hline
	&TV-TQFT & Kitaev model
	\\\hline\hline
	bulk region&  center $\mathcal{Z}(\mathcal{A})$ of fusion category $\mathcal{A}$ & Drinfel'd double $D(H)$ of semisimple, finite-dimensional Hopf algebra $H$
	\\\hline
	boundary component & fusion category $\mathcal{W}_{a}$ &semisimple, finite-dimensional Hopf algebra $K$
	\\\hline
	bulk $\to$ boundary &braided tensor equivalence $\mathcal{Z}(\mathcal{A}) \to \mathcal{Z}(\mathcal{W}_{a})$ & twist $J$ and isomorphism $D(H)_{J} \cong D(K)$ of quasitriangular Hopf algebras
	\\\hline
	codimension-one defect & fusion category $\mathcal{W}_{d}$ & semisimple, finite-dimensional Hopf algebra $K_{d}$
	\\\hline
	bulk $\to$ defect & braided tensor equivalence $\mathcal{Z}(\mathcal{A}_{1})\boxtimes \mathcal{Z}(\mathcal{A}_{2})^{rev} \to \mathcal{Z}(\mathcal{W}_{d})$ 
	& twist $F$ and isomorphism  $(D(H_{1})  \oo  D(H_{2}) )_{F} \cong D(K) $ of
quasitriangular Hopf algebras
	\\\hline
	\end{tabular}

\section{Conditions for a Kitaev model with topological defects and boundaries}

\label{section:TranslationKitaev}

In this section we derive the counterparts of the conditions in Section \ref{section:FSVSummary} for Kitaev models with topological defects and topological boundaries.
We start by highlighting how the Hopf algebra $D(H)$ appears in the Kitaev model without defects and boundaries.
In the model with defects and boundaries, the Drinfel'd double $D(H)$ is then replaced with the Hopf algebraic data from Proposition~\ref{proposition:FSVKitaevTranslation} at the edges and sites in defects and boundaries. 


In the Kitaev model without defects and boundaries the Hopf algebra $D(H)$ appears in the following manner: 
\begin{itemize}
	\item There are local operators which define a $D(H)$-module structure on the extended space $\HSpace$. These are the face and vertex operators for a site $s$ from  Theorem~\ref{theorem:Drinfel'dActionHSpace}.
	\item For suitable paths $\rho:u\to v$, the holonomy  is a $D(H)$-module homomorphism $\Hol_{\rho}: D(H)^{*} \oo \HSpace \to \HSpace$, where $ \oo $ is the tensor product of $D(H)$-modules (see Corollary~\ref{corollary:HolonomyModuleHom}), $D(H)^{*}$ is equipped with the left coregular action and $\HSpace$ with the action defined by the starting site $u$ of $\rho$. 
\end{itemize}
Note that the coalgebra structure of $D(H)$ only plays a role for the second point, as the tensor product $D(H)^* \oo \HSpace$ is defined in terms of its comultiplication $\Delta$.
	
	We now give a slightly simpler description of the Hopf algebraic data from Proposition~\ref{proposition:FSVKitaevTranslation} for bulk regions, boundaries and defects. 
	We no longer list the Hopf algebras $D(K)$ for boundaries from~\ref{proposition:FSVKitaevTranslation}.3, as they are isomorphic to $D(H)_{F}$.
	We similarly omit the Hopf algebras $D(K)$ for defects.
A Kitaev model with topological boundaries and defects is then given by the following algebraic data:
\begin{compactenum}[({D}1):]
	\item A Drinfel'd double $D(H_{b})$ of a semisimple finite-dimensional Hopf algebra $H_{b}$ for every bulk region $b$.
		\label{D1}
	\item A twist $F_{a}$ of $D(H_{b})$ for every boundary $a$ adjacent to a bulk region $b$.
		\label{D2}
	\item A twist $F_{d}$ of $D(H_{b_{1}}) \oo D(H_{b_{2}})$ for every defect $d$ separating two boundary regions $b_{1}$ and $b_{2}$.
		\label{D3}
\end{compactenum}

	We now compare with the Kitaev model without defects and boundaries and propose that a Kitaev model with defects and boundaries should have the following local operators
	\begin{compactenum}[({O}1):]
		\item For every bulk region $b$ and every boundary adjacent to $b$: Local operators which define a $D(H_{b})$-module structure on the extended space $\HSpace$.
			\label{O1}
		\item For every defect $d$ separating two boundary regions $b_{1}$ and $b_{2}$.: Local operators which define a $D(H_{b_{1}}) \oo D(H_{b_{2}})$-module structure on the extended space $\HSpace$.
			\label{O2}
	\end{compactenum}
	In the model with defects and boundaries, we will use holonomies to describe the creation and fusion of excitations. As in the Kitaev model without defects, we define them as endomorphisms of $\HSpace$ assigned to paths in the underlying thickened ribbon graph $D(\Gamma)$. As we will only consider paths inside a bulk region  and adjacent boundary or defect lines, these holonomies depend on the  data assigned to the bulk regions. For a path $\rho$ in a bulk region $b$ and adjacent boundary and defect lines, we define them as linear maps
	
	\begin{align*}
		\widetilde{\Hol}_{b,\rho}: D(H_{b})^{*} \oo \HSpace \to \HSpace.
	\end{align*}
	 As in Corollary \ref{corollary:HolonomyModuleHom}, we require that this map 
	is a $D(H_{b})$-module homomorphism
	when $D(H_b)^*$ is equipped with one of the two $D(H_b)$-module structures associated with  $\rho$, as in Corollary \ref{corollary:HolonomyModuleHom}, and where the tensor product in $ D(H_b)^*\oo\HSpace $ depends on whether $\rho$ starts or ends in a bulk region, a boundary or a defect:
	\begin{compactenum}[({H}1):]
		\item \label{H1}  If $\rho$ starts in the bulk region $b$, then $ \oo $ is the tensor product of $D(H_{b})$-modules and $\HSpace$ is equipped with the local $D(H_{b})$-module structure associated to the starting site of $\rho$ from~(O\ref{O1}).

		\item If $\rho$ starts in a boundary $a$ adjacent to $b$, then  $ \oo$ is the tensor product of $D(H_{b})_{F_{a}}$-modules, $\HSpace$ is equipped with the local $D(H_{b})$-module structure associated to the starting site of $\rho$ from~(O\ref{O1}).
			\label{H2}
			
		\item If $\rho$ starts in a defect $d$ separating bulk regions $b_{1}$ and $b_{2}$ and $b=b_{1}$, then $D(H_{b})=D(H_{b_{1}})\subseteq D(H_{b_{1}}) \oo D(H_{b_{2}})$ is a subalgebra. We equip $D(H_{b})^{*}$ with the trivial $D(H_{b_{2}})$-action and equip $\HSpace$ with the local  $D(H_{b_{1}}) \oo D(H_{b_{2}})$-module structure associated to the starting site of $\rho$ from~(O\ref{O2}). We let $ \oo$ be the tensor product of $\left( D(H_{b_{1}}) \oo D(H_{b_{2}}) \right)_{F_{d}}$-modules.
			\label{H3}
	\end{compactenum}

	In Section~\ref{section:FSVSummary} we also described the movement of excitations in a Turaev-Viro-TQFT. 
	The holonomy along a path $\rho$ in the Kitaev model without defects and boundaries generates excitations at the end points of $\rho$ (see Corollary~\ref{corollary:HolonomyFusion}), but it does not move excitations from one site to another. 
	To implement the process of moving excitations we associate to each simple path $\rho$  in $D(\Gamma)$ that traverses a bulk region $b$ and adjacent defects and boundaries 
	 an additional operator on $\HSpace$,  the \emph{transport operator}  $T_{\rho}$. 
	In analogy to the conditions for moving excitations in a Turaev-Viro TQFT from section~\ref{section:FSVSummary} we require that $T_{\rho}$ fulfills the following conditions:
	\begin{compactenum}[({T}1):]	
	\item \textbf{Fusion:} 
		If $\rho$ ends to the left of a site (see Definition~\ref{def:leftorright}), then $T_{\rho}$ restricts to a map:
	\begin{align}
		T_{\rho}:\HSpace(u,M,v,N) \to \HSpace(u,\C, v,M \oo N)
			\label{eq:fusion}
		\end{align}
		\label{T1}
		Here $\HSpace(u,M,v,N)$ is defined analogously to \eqref{eq:MultipleExcitationSpace}, $M$ and $N$ are $D(H_{b})$-modules and $v$ and $ \oo $ is the tensor product from (H\ref{H1}),(H\ref{H2}) or~(H\ref{H3}). 
		
		If $\rho$ instead ends to the right of a site, then the order of the tensor product is reversed. We illustrate this by Figure~\ref{figure:FuseExcitations}.
	\item \textbf{Braiding:} For a path $\rho:u\to v$ in a bulk region $b$ such that $\rho$ ends to the left of $v$ there is another path $\rho':u\to v$, constructed from $\rho$, that ends to the right of $v$ such that
		\begin{align}
			T_{\rho'}= T_{\rho} \circ R
			\label{eq:KitaevBraidingCondition}
		\end{align}
		Here $R$ denotes an action of the $R$-matrix of $D(H_{b})$ on the sites $u$ and $v$, if $v$ is in the bulk region $b$.
		If instead $v$ is in an adjacent boundary or defect, then $R$ instead denotes an action with the twisted $R$-matrix of $D(H_{b})$.
		\label{T2}
	\end{compactenum}

\begin{figure}[H]
	\centering
	\scalebox{0.35}{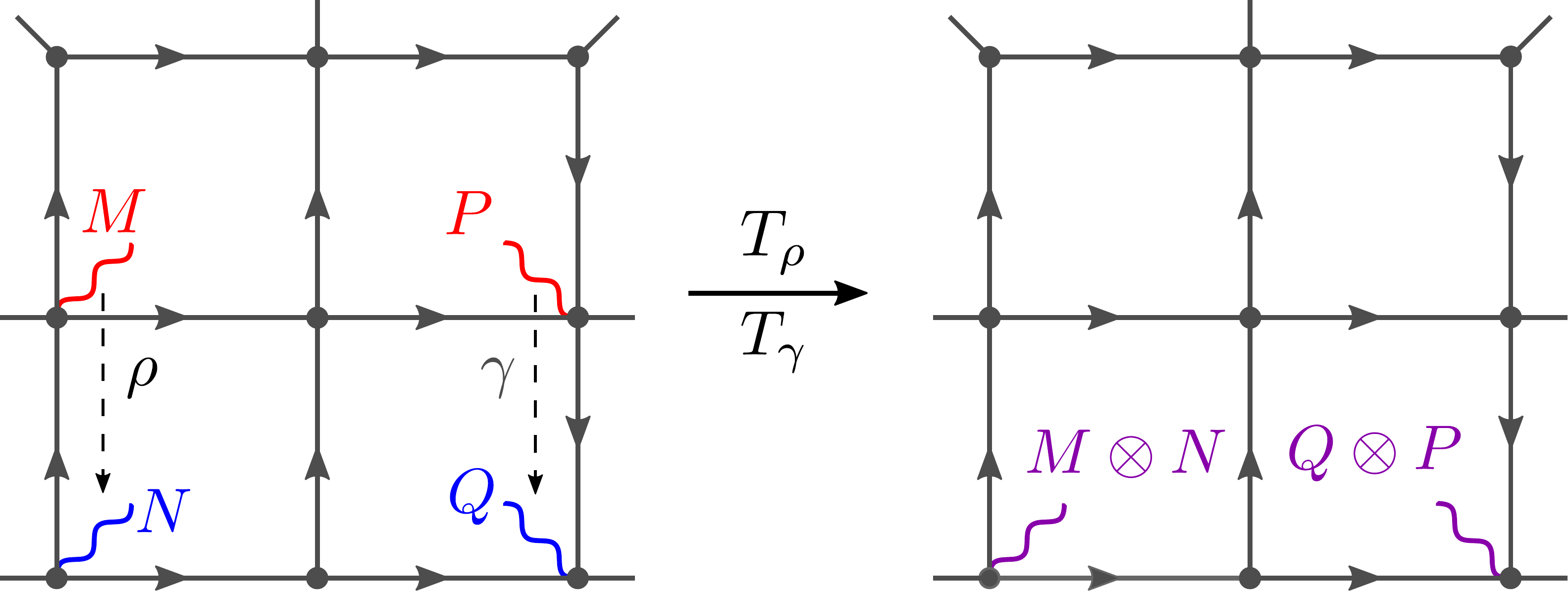}
	\caption{Moving an excitation of type $M$ along $\rho:u\to v$ to an excitation of type $N$ fuses them to an excitation of type  $M \oo N$, as $\rho$ ends to the left of $v$.  Moving an excitation of type $P$ along $\gamma$ fuses $P$ with $Q$ to an excitation of type $Q \oo P$, as $\gamma:w\to x$ ends to the right of $x$.}
	\label{figure:FuseExcitations}
\end{figure}

\section{Kitaev models with topological defects and  boundaries}
\label{sec:defect model}

Codimension 1 and 2 defect structures in Kitaev lattice models based on group algebras of finite groups were first investigated in \cite{BD}. The defect data for Kitaev lattice models based on unitary quantum groupoids was then identified in \cite{KK}, and an explicit Kitaev model with defects based on Hopf algebras was then constructed in \cite{K}. In  \cite{KMM} defects were generalized to higher Kitaev models based on crossed modules for semisimple Hopf algebras. 

In this section, we construct a Kitaev model with \emph{topological} defects and boundaries, that satisfies the conditions derived in \cite{FSV}. More specifically, we construct a Kitaev model with defects and boundaries, which exhibits the structures and satisfies the conditions from Section~\ref{section:TranslationKitaev}, which are the Hopf algebraic counterparts of the conditions in \cite{FSV}. 

In a bulk region our model behaves identically to the Kitaev model without defects and boundaries. We also show that the model without defects and boundaries can be obtained as a special case of our model  by choosing trivial defects.

We show that the processes of moving and braiding excitations described in \cite{FSV} can be implemented in our model as linear endomorphisms of its extended space assigned to the  the relevant paths.

Just as the Kitaev model without defects and boundaries, our model is based on a ribbon graph. We require additional structure on these ribbon graphs to describe oriented surfaces with defects and boundaries. This structure is introduced in Subsection~\ref{subsec:ribdefect}. In Subsection~\ref{subsection:HilbertSpace}, we then define the extended space of the model, its vertex and face operators and the holonomies. We also give some basic properties of these operators. In Subsection~\ref{subsec:fusionbraiding} we describe the fusion, the transport and the braiding of excitations. We show that our implementation of these processes has the properties formulated in \cite{FSV} and summarized in Section~\ref{section:TranslationKitaev}.
In Subsection \ref{subsec:transparent} we then show that our model reduced to the Kitaev model without defects and boundaries for trivial choices of defect data.

\subsection{Ribbon graphs with defect lines and boundaries}
\label{subsec:ribdefect}

For the Kitaev model with defects and boundaries we require additional structure on ribbon graphs to describe defect and boundary lines. These defect and boundary lines  are \emph{oriented cyclic} subgraphs, i.e. connected subgraphs in which every vertex has exactly one incoming and one outgoing edge end. Examples for ribbon graphs with defects and boundaries are given in Figure~\ref{fig:boundarygraph} and Figure~\ref{fig:defect}. 

\begin{definition}
	\label{definition:RibbonGraphDefect}
A \emph{ribbon graph with defects and boundaries} is a ribbon graph $\Gamma=(V,E)$ with the following additional data
\begin{itemize}
	\item a non-empty family $\left( \Gamma_{b} =(V_{b},E_{b})\right)_{b \in B}$ of connected subgraphs, the \emph{bulk regions} of $\Gamma$, with edges in these subgraphs called \emph{bulk edges},
	
	\item a family $\left( \Gamma_{a} \right)_{a \in A}$ of  connected, oriented cyclic subgraphs, the \emph{boundary lines}, with vertices and edges in these subgraphs called \emph{boundary vertices} and  \emph{boundary edges},
	\item for every $a\in A$ a bulk region $b_{a} \in B$, called the \emph{bulk region bordered by $a$,}
	\item a family $\left( \Gamma_{d} \right)_{d \in D}$ of connected, oriented cyclic subgraphs, the \emph{defect lines}, with vertices and edges in these subgraphs  called \emph{defect vertices} and  \emph{defect edges},
	\item for every $d\in D$ a pair $(b_{dL},b_{dR}) \in B \times B$ of bulk regions with $b_{dL}\neq b_{dR}$,  the \emph{bulk region to the left (resp. to the right) of $d$.} 
\end{itemize}
This data has to satisfy the following conditions:
\begin{itemize}
	\item For $d,d' \in D \cup A$ with $d\neq d'$, the graphs $\Gamma_{d}$ and $\Gamma_{d'}$ do not share any vertices.
	\item Every edge of $\Gamma$ is in exactly one of the subgraphs $\Gamma_{i}$ for $i\in B \cup D \cup A$.
	\item For $d \in D$, every vertex $v$ of $\Gamma_{d}$ also is a vertex of $\Gamma_{b_{dL}}$ and $\Gamma_{b_{dR}}$. The edges of $v$ in $\Gamma_{b_{dL}}$ (resp. $\Gamma_{b_{dR}}$) are to the left (resp. to the right) of the orientation of $\Gamma_{d}$. 
	\item Every vertex is at least in one bulk region and at most in two bulk regions. If a vertex $v$ is in two bulk regions $b_1, b_2$, then
	there is a defect line $d$ containing $v$ such that $b_1,b_2$ are to the left and right of $d$.
	
	\item For $a \in A$, every vertex $v$ of $\Gamma_{a}$ also is a vertex of $\Gamma_{b_{a}}$. The edges of $v$ in $\Gamma_{b_{a}}$ are to the left of the orientation of $\Gamma_{a}$.
\end{itemize}
\end{definition}

\begin{figure}[H]
\begin{center}
\begin{tikzpicture}[scale=.7]
\draw[line width=1pt, color=black, ->, >=stealth] (-2,0)--(0,0);
\draw[line width=1pt, color=black, ] (2,0)--(0,0);
\draw[line width=1pt, color=black, ->, >=stealth] (-2,4)--(0,4);
\draw[line width=1pt, color=black, ] (2,4)--(0,4);
\draw[line width=1pt, color=black, ->, >=stealth] (-2,0)--(-2,2);
\draw[line width=1pt, color=black, ] (-2,4)--(-2,2);
\draw[line width=1pt, color=black, ] (2,0)--(2,2);
\draw[line width=1pt, color=black, ->, >=stealth ] (2,4)--(2,2);
\draw[color=black, fill=black] (-2,0) circle (.2);
\draw[color=black, fill=black] (2,0) circle (.2);
\draw[color=black, fill=black] (-2,4) circle (.2);
\draw[color=black, fill=black] (2,4) circle (.2);
\draw[line width=1pt, color=black, ] (-2,0)--(-1.5,.5) node[anchor=west]{$s$};
\draw[line width=1pt, color=black, ->, >=stealth] (-4,-2)--(-3,-1);
\draw[line width=1pt, color=black,] (-2,0)--(-3,-1);
\draw[line width=1pt, color=black, ->, >=stealth] (2,0)--(3,-1);
\draw[line width=1pt, color=black,] (4,-2)--(3,-1);
\draw[line width=1pt, color=black, ->, >=stealth] (4,6)--(3,5);
\draw[line width=1pt, color=black,] (3,5)--(2,4);
\draw[line width=1pt, color=black, ->, >=stealth] (-2,4)--(-3,5);
\draw[line width=1pt, color=black,] (-3,5)--(-4,6);
\draw[color=red, fill=red] (-4,-2) circle (.2);
\draw[color=red, fill=red] (-4,6) circle (.2);
\draw[color=red, fill=red] (4,-2) circle (.2);
\draw[color=red, fill=red] (4,6) circle (.2);
\draw[line width=1.5pt, color=red] (-4,-2)--(-3,-1.5) node[anchor=west]{$t$};
\draw[line width=1.5pt, color=red, ->, >=stealth] (-4,6)--(-4,2);
\draw[line width=1.5pt, color=red,] (-4,-2)--(-4,2);
\draw[line width=1.5pt, color=red, ->, >=stealth] (4,-2)--(4,2);
\draw[line width=1.5pt, color=red,] (4,2)--(4,6);
\draw[line width=1.5pt, color=red, ->, >=stealth] (4,6)--(0,6);
\draw[line width=1.5pt, color=red,] (0,6)--(-4,6);
\draw[line width=1.5pt, color=red, ->, >=stealth] (-4,-2)--(0,-2);
\draw[line width=1.5pt, color=red,] (0,-2)--(4,-2);
\node at (0,2) {$b$};
\node at (4.2,2)[color=red, anchor=west]{$c$};
\end{tikzpicture}
\end{center}
\caption{Bulk graph $b$ with a boundary component $a$, a bulk site $s$ and a boundary site $t$.}
\label{fig:boundarygraph}
\end{figure}
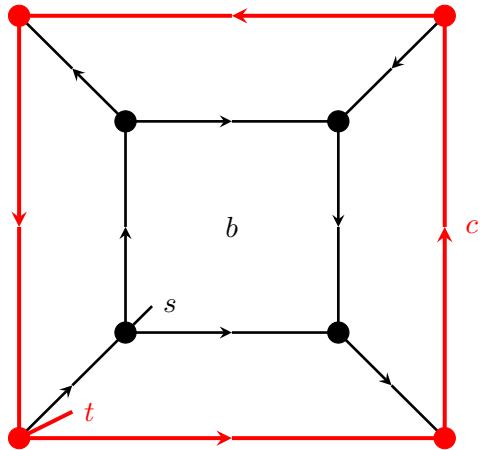

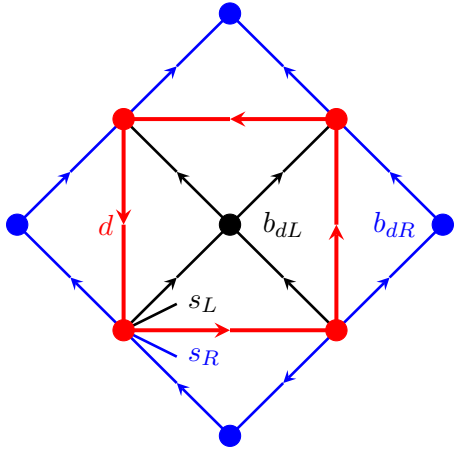
\begin{figure}[H]
\begin{center}
\begin{tikzpicture}[scale=.7]
\draw[line width=1.5pt, color=red, ->,>=stealth] (-2,-2)--(0,-2);
\draw[line width=1.5pt, color=red] (0,-2)--(2,-2);
\draw[line width=1.5pt, color=red, ->,>=stealth] (2,-2)--(2,0);
\draw[line width=1.5pt, color=red] (2,0)--(2,2);
\draw[line width=1.5pt, color=red, ->,>=stealth] (2,2)--(0,2);
\draw[line width=1.5pt, color=red] (0,2)--(-2,2);
\draw[line width=1.5pt, color=red, ->,>=stealth] (-2,2)--(-2,0);
\draw[line width=1.5pt, color=red] (-2,0)--(-2,-2);
\draw[line width=1pt, color=black, ->,>=stealth] (0,0)--(-1,1);
\draw[line width=1pt, color=black] (-1,1)--(-2,2);
\draw[line width=1pt, color=black, ->,>=stealth] (0,0)--(1,1);
\draw[line width=1pt, color=black] (1,1)--(2,2);
\draw[line width=1pt, color=black, ->,>=stealth] (-2,-2)--(-1,-1);
\draw[line width=1pt, color=black] (-1,-1)--(0,0);
\draw[line width=1pt, color=black, ->,>=stealth] (2,-2)--(1,-1);
\draw[line width=1pt, color=black] (1,-1)--(0,0);
\draw[line width=1pt, color=blue, ->,>=stealth] (2,2)--(1,3);
\draw[line width=1pt, color=blue] (1,3)--(0,4);
\draw[line width=1pt, color=blue, ->,>=stealth] (-2,2)--(-1,3);
\draw[line width=1pt, color=blue] (-1,3)--(0,4);
\draw[line width=1pt, color=blue, ->,>=stealth] (-4,0)--(-3,1);
\draw[line width=1pt, color=blue] (-3,1)--(-2,2);
\draw[line width=1pt, color=blue, ->,>=stealth] (-2,-2)--(-3,-1);
\draw[line width=1pt, color=blue] (-3,-1)--(-4,0);
\draw[line width=1pt, color=blue, ->,>=stealth] (0,-4)--(-1,-3);
\draw[line width=1pt, color=blue] (-1,-3)--(-2,-2);
\draw[line width=1pt, color=blue, ->,>=stealth] (2,-2)--(1,-3);
\draw[line width=1pt, color=blue] (1,-3)--(0,-4);
\draw[line width=1pt, color=blue, ->,>=stealth] (2,-2)--(3,-1);
\draw[line width=1pt, color=blue] (3,-1)--(4,0);
\draw[line width=1pt, color=blue, ->,>=stealth] (4,0)--(3,1);
\draw[line width=1pt, color=blue] (3,1)--(2,2);
\draw[color=black, line width=1pt] (-2,-2)--(-1, -1.5) node[anchor=west]{$s_L$};
\draw[color=blue, line width=1pt] (-2,-2)--(-1, -2.5) node[anchor=west]{$s_R$};
\draw[color=red, fill=red] (-2,-2) circle (.2);
\draw[color=red, fill=red] (-2,2) circle (.2);
\draw[color=red, fill=red] (2,-2) circle (.2);
\draw[color=red, fill=red] (2,2) circle (.2);
\draw[color=black, fill=black] (0,0) circle (.2);
\draw[color=blue, fill=blue] (0,4) circle (.2);
\draw[color=blue, fill=blue] (0,-4) circle (.2);
\draw[color=blue, fill=blue] (-4,0) circle (.2);
\draw[color=blue, fill=blue] (4,0) circle (.2);
\node at (1,0)[color=black]{$b_{dL}$};
\node at (2.5,0)[anchor=west, color=blue]{$b_{dR}$};
\node at (-2,0)[color=red, anchor=east]{$d$};
\end{tikzpicture}
\end{center}
\caption{Defect $d$ separating bulk regions $b_{dL}$ and $b_{dR}$ with a pair $(s_L,s_R)$ of defect sites.}
\label{fig:defect}
\end{figure}

The additional structure of a ribbon graph with defects and boundaries leads to a distinction of bulk sites, boundary sites and defect sites in its thickening.

\begin{definition}\label{def:sitesdefect} Let $\Gamma$ be a ribbon graph with defects and boundaries and $D(\Gamma)$
its thickening.
\begin{compactenum}
	\item A \emph{bulk site} in $b\in B$ is a site $s$, whose face path and vertex path only consist of edges in $D(\Gamma_{b})$.
	\item A \emph{pair of defect sites} in $d\in D$ at a defect vertex $v$ is the pair $(s_{L},s_{R})$ of sites $s_{L}$ and $s_{R}$ at $v$ such that $s_{L}$ ($s_{R}$) is directly to the left (to the right) of the outgoing defect edge, viewed in the direction of its orientation, as shown in see Figure~\ref{fig:defect}. The sites $s_{L}$ and $s_{R}$ are called \emph{defect sites} in $b_{dL}$ and $b_{dR}$, respectively.
	\item A \emph{boundary site} in $a \in A$ is a site $t$ at a boundary vertex $v$ in $\Gamma_{a}$ directly to the left of the outgoing boundary edge of $v$, as shown in Figure~\ref{fig:boundarygraph}.
	\end{compactenum}
\end{definition}	
	
Paths in a ribbon graph with defects and boundaries can be characterized by their behavior with respect to defect lines and boundaries. In the following, we only consider simple paths that do not cross over defect or boundary edges.	

\begin{definition}\label{def:permissible}	
 A simple path $\rho \in D(\Gamma)$ is called \emph{permissible}, if $\rho$ does not contain edges of the form $e^{\pm t},e^{\pm s}$ with $e$ a defect or boundary edge.
\end{definition}

%

\subsection{The extended space and operators}
\label{subsection:HilbertSpace}

	 In this section we define the basic structures of the Kitaev model with defects and boundaries: The extended space, the vertex and face operators, holonomies and a transport operator. We then show that this model fulfills the conditions stated in Section~\ref{section:TranslationKitaev}. 

	 The ingredients for the model are a ribbon graph with boundaries and defects $\Gamma$ together with the algebraic data derived in Section \ref{section:TranslationKitaev}, (D\ref{D1})-(D\ref{D3}):
\begin{itemize}
	\item A semisimple, finite-dimensional Hopf algebra $H_{b}$ to every bulk region $b\in B$.
	\item A twist $F_{a}$ of $D(H_{b})$  to every boundary line $a \in A$.
	\item A twist $F_{d}$ of $D(H_{b_{dL}}) \oo D(H_{b_{dR}})$ to every defect line $d \in D$. 
\end{itemize}
We now associate a twist $F$ of $D(H_{b})$ to every site $s$ in a bulk region $b$, the \emph{twist at $s$}. 
We distinguish whether $s$ is (i) a bulk site in $b$, (ii) a boundary site in a boundary line $a$ adjacent to $b$ or (iii) a defect site in a defect line $d$ adjacent to $b$. In case (i) $F$ is the trivial twist. In case (ii), $F$ is the twist $F_{a}$. In case (iii)  we have $b\in\left\{b_{dL}, b_{dR}  \right\}$ and $F$ is the the projection of the twist $F_{d}$ onto the Hopf subalgebra $D(H_{b}) \subseteq D(H_{b_{dL}}) \oo D(H_{b_{dR}})$ as in Example~\ref{example:ProjectedTwist}.

To construct the extended space, we assign to every edge $e\in E$ a vector space $H_{e}$, where
\begin{itemize}
	\item $H_{e}=H_{b}$, if $e\in E_{b}$ is an edge in a bulk region $b\in B$.
	\item $H_{e}=H_{b_{a}}$, if $e\in E_{a}$ is a boundary edge in a boundary line $a\in A$.
	\item $H_{e}=H_{b_{dL}} \oo H_{b_{dR}}$, if $e\in E_{d}$ is a defect edge in a defect line $d\in D$.
\end{itemize}
We define the \emph{extended space} of the model as the vector space 
\begin{align}\label{eq:hilbdef}
	\HSpace := \tensor_{e\in E}H_{e}
\end{align}

For every bulk region $b$, $h \oo \alpha\in D( H_b)^{*}$, and every simple path $\rho$  that is either (a) a permissible path in $b$ and the adjacent defect and boundary lines or (b) a vertex path around a defect vertex, we define
 a holonomy map
\begin{align}
	\Hol_{b,\rho}^{h \oo \alpha } :\HSpace\to \HSpace.
	\label{eq:HolonomyBulk}
\end{align}
{Note that the index $b$ is superfluous in case (a), but necessary in case (b) to distinguish the two bulk regions separated by a defect line. As in the case of a Kitaev model without defects, these holonomies are defined in terms of edge operators associated to the edges $e\in E$ of the underlying ribbon graph $\Gamma=(V,E)$, but we have to distinguish four cases:
\begin{compactenum}[(i)]
	\item $e\in E_{b}$ or $e\in E_{a}$, where $a\in A$ is a boundary line adjacent to $b$,
	\item $e\in E_{d}$, where $b$ is the bulk region to the left of the defect line $d$,
	\item $e\in E_{d}$, where $b$ is the bulk region to the right of the defect line $d$,
	\item $e$ is in a different bulk region or in a defect or boundary line not adjacent to $b$.
\end{compactenum}
 In all cases, we define the triangle operators as the holonomies $\Hol_{b,e^{\nu}}^{h \oo \alpha} $  along the paths $e^{\nu}$ for $\nu\in \left\{\pm s,\pm t,\pm L,\pm R  \right\}$ for $\alpha\oo h\in D(H_b^*)$. The four cases differ only in the vector space $H_e$ assigned to the edge $e$. In case (i) one has  $H_e=H_b$, in case (ii) $H_e=H_b\oo H_{b'}$, where $b'$ is the bulk region to the right of $d$, in case (iii) $H_e=H_{b'}  \oo H_{b}$, where $b'$ is the bulk region to the left of $d$, and in case (iv) $H_e=H_{b'}$ or $H_e=H_{b'}\oo H_{b''}$ for $b\neq b'$ and $b\neq b''$. Case (iv) is only required to inductively define holonomies for vertex paths around a defect vertex, and the associated holonomy is taken to be trivial. As in the model without boundaries and defects, we require that the triangle operators act trivially on all tensor factors in $\HSpace=\bigotimes_{e\in E} H_e$ except the one associated with $e\in E$.

\begin{definition}\label{def:edgeopsdefect} The \emph{triangle operators} for an edge $e\in E$ are the holonomies $H^{h\oo \alpha}_{b, e^\nu}$ defined by
\begin{align*}
	\Hol_{b,e^{-\nu}}^{h \oo \alpha} := \Hol_{b,e^{\nu}}^{S(h \oo \alpha)}\qquad  \nu\in \left\{ L,R,s,t \right\},
\end{align*}
where $S$ is the antipode of $D(H_{b})^{*}$, and by
\begin{description}
\item[case (i):] $m\in H_b$
\begin{align*}
	\Hol_{b,e^{s}}^{h \oo \alpha} (m) &= \alpha(1) \cdot mh, &
	\Hol_{b,e^{-t}}^{h \oo \alpha} (m) &=\alpha(1) \cdot S(h)m
	\\
	\Hol_{b,e^{R}}^{h \oo \alpha} (m) &= \varepsilon(h) \langle \alpha , m_{(2)} \rangle m_{(1)}, &
	\Hol_{b,e^{-L}}^{h \oo \alpha} (m) &= \varepsilon(h)\langle \alpha , S(m_{(1)}) \rangle m_{(2)},
\end{align*}

\item[case (ii):] $m\oo n\in H_b\oo H_{b'}$
\begin{align*}
	\Hol_{b,e^{s}}^{h \oo \alpha} (m \oo n) &=\alpha(1)  mh \oo n, &
	\Hol_{b,e^{t}}^{h \oo \alpha} (m \oo n) &=\alpha(1)  hm \oo n
	\\
	\Hol_{b,e^{L}}^{h \oo \alpha} (m \oo n) &=\varepsilon(h) \langle \alpha , m_{(1)} \rangle  m_{(2)} \oo n, &
	\Hol_{b,e^{R}}^{h \oo \alpha} (m \oo n) &=\varepsilon(h)\alpha(1) m\oo n,
\end{align*}

\item[case (iii):] $m\oo n\in H_{b'}\oo H_{b}$
\begin{align*}
	\Hol_{b,e^{s}}^{h \oo \alpha} (m \oo n) &=\alpha(1)  m \oo nh, &
	\Hol_{b,e^{t}}^{h \oo \alpha} (m \oo n) &=\alpha(1)  m \oo hn
	\\
	\Hol_{b,e^{R}}^{h \oo \alpha} (m \oo n) &=\varepsilon(h) \langle \alpha , n_{(2)} \rangle  m \oo n_{(1)}, &
	\Hol_{b,e^{L}}^{h \oo \alpha} (m \oo n) &=\varepsilon(h) \alpha(1) \rangle  m \oo n,
\end{align*}
\item[case (iv):] $m\oo n\in H_{b'}\oo H_{b''}$
\begin{align*}
	\Hol_{b,e^{\nu}}^{h \oo \alpha}(m\oo n) = \varepsilon(h)\alpha(1) m\oo n\quad \nu\in\{s,t,R,L\},
\end{align*}
\end{description}
\end{definition}

%
%
%
%

 Note that the formulas for case (i) coincide with the formulas in Definition \ref{definition:HolonomyBasic} for the triangle operators in a model without defects and boundaries. 
Definition \ref{def:edgeopsdefect} thus generalizes the triangle operators from Definition \ref{definition:HolonomyBasic} for models without defects and boundaries. 
As in the model without defects, see Definition \ref{definition:HolonomyComposite}, we define the holonomies along  simple paths $\rho$ by decomposing them into subpaths that form a left joint. 

\begin{definition}
	\label{definition:HolonomyCompositeDefect}
	Let $\rho=\rho_{1}\circ\rho_{2}$ a simple path in reduced form with $\rho_{2}:s \to t^{\eta_{2}}$ and $\rho_{1}:t^{\eta_{1}}\to u$ with $\eta_{1},\eta_{2}\in \left\{ L,R \right\}$. Write $\beta:=h \oo \alpha\in D(H_{b})^{*}$ and $\Delta_{D(H_{b})^{*}}(\beta)=\beta_{(1)} \oo \beta_{(2)}$. Then we define
\begin{align}
	&\Hol_{b,\rho}^{\beta}  = \Hol_{b,\rho_{1}}^{\beta_{(1)}}\circ \Hol_{b,\rho_{2}}^{\beta_{(2)}} \quad &\text{if $(\rho_{1},\rho_{2})_{\emptyset}$ or $(\rho_{2},\rho_{1}^{-1})_{\prec}$}
	\label{eq:HolonomyRecursionStandardDefect}
	\\
	&\Hol_{b,\rho}^{\beta}  = \Hol_{b,\rho_{2}}^{\beta_{(2)}}\circ \Hol_{b,\rho_{1}}^{\beta_{(1)}} &\text{if $(\rho_{1}^{-1},\rho_{2})_{\prec}$}
	\label{eq:HolonomyRecursionLeftJoinDefect}
\end{align}
\end{definition}

By an argument that is fully analogous to the proof of Lemma \ref{lemma:HolonomyWellDefined} one shows that the resulting holonomy is independent of the decomposition of $\rho$ into subpaths $\rho_1,\rho_2$ in Definition \ref{definition:HolonomyCompositeDefect}. By considering vertex and face paths associated with sites in   $\Gamma$, we can then define the associated vertex and face operators as the holonomies along these paths, thus generalizing Definition \ref{definition:VertexFaceOperators} for the model without defects and boundaries. We also associate generalized vertex and face operators to pairs of defect sites.

\begin{definition}\label{def:vertexfacedefect}
Let $\Gamma$ be a ribbon graph with defects and boundaries.
\begin{enumerate}
\item The \emph{vertex and face operator}  for a site $s$ in a bulk region $b$ 
 are the linear maps
	\begin{align}
		A_{b,s}^{h}=\Hol_{b,v}^{h \oo \varepsilon} :\, &\HSpace \to \HSpace 
		&
		B_{b,s}^{\alpha}= \Hol_{b,f}^{1 \oo \alpha}:\, &\HSpace \to \HSpace,
		\label{eq:FaceOperatorBulk}
	\end{align}
where $f$ and $v$ are the face and vertex path based at $s$ and $\alpha \oo h\in D(H_{b})$. We use the notation
$BA_{b,s}^{\alpha \oo h}= B_{s}^{\alpha}A_{s}^{h} :\, \HSpace \to \HSpace$.

\item The \emph{vertex and face operator} for a pair $(s_{L},s_{R})$ of defect sites in a defect line $d$ are
	\begin{align}
		BA_{d, (s_{L},s_{R})}^{h \oo k}:= BA_{b_{dL},s_{L}}^{h}\circ BA_{b_{dR},s_{R}}^{k}:\HSpace\to \HSpace,
		\label{eq:FaceVertexOperatorDefect}
	\end{align}
where $h \oo k\in D(H_{b_{dL}}) \oo D(H_{b_{dR}})$.
\end{enumerate}
	\end{definition}
	
An example of the action of a vertex and face operator at a  defect vertex vertex is given in Figure~\ref{fig:vertexdefect}.

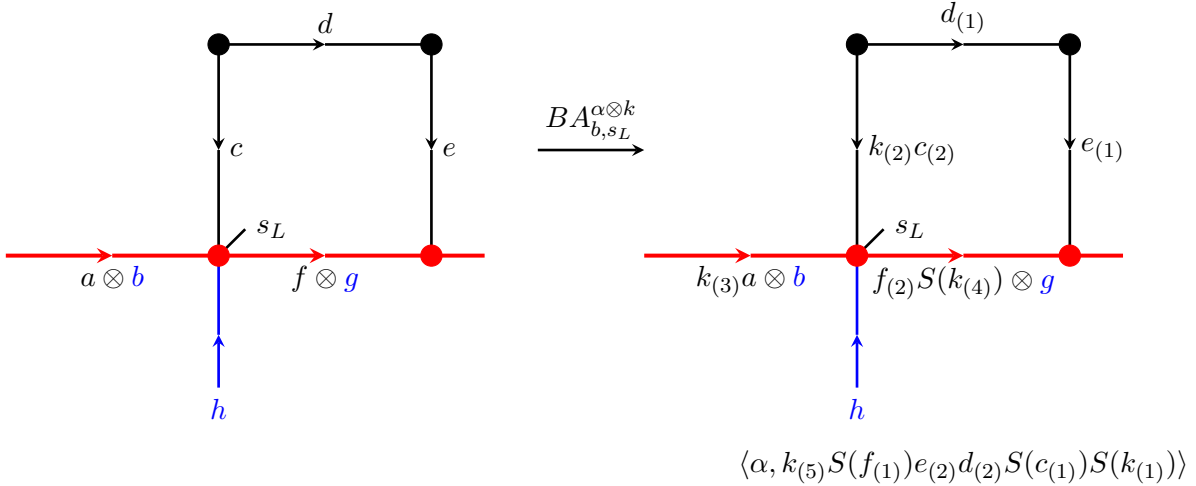
\begin{figure}[H]
\begin{center}
\begin{tikzpicture}[scale=.7]
\begin{scope}[shift={(0,0)}]
\draw[line width=1.5pt, color=red, ->,>=stealth] (-4,0)--(-2,0);
\draw[line width=1.5pt, color=red] (-2,0)--(0,0);
\draw[line width=1.5pt, color=red, ->,>=stealth] (0,0)--(2,0);
\draw[line width=1.5pt, color=red] (2,0)--(5,0);
\draw[line width=1pt, color=black,->,>=stealth] (0,4)--(0,2);
\draw[line width=1pt, color=black,] (0,2) node[anchor=west]{$c$}--(0,0);
\draw[line width=1pt, color=black,->,>=stealth] (0,4)--(2,4) node[anchor=south]{$d$};
\draw[line width=1pt, color=black,] (2,4)--(4,4);
\draw[line width=1pt, color=black,->,>=stealth] (4,4)--(4,2) node[anchor=west]{$e$};
\draw[line width=1pt, color=black,] (4,2)--(4,0);
\draw[line width=1pt, color=blue,->,>=stealth] (0,-2.5)node[anchor=north]{$h$}--(0,-1.5);
\draw[line width=1pt, color=blue,] (0,-1.5)--(0,0);
\draw[line width=1pt, color=black](0,0)--(.5,.5) node[anchor=west]{$s_L$};
\draw[color=red, fill=red] (0,0) circle (.2);
\draw[color=red, fill=red] (4,0) circle (.2);
\draw[color=black, fill=black] (0,4) circle (.2);
\draw[color=black, fill=black] (4,4) circle (.2);
\node at (2,0) [anchor=north]{$f\oo {\color{blue}g}$};
\node at (-2,0)[anchor=north]{$a\oo{\color{blue} b}$};
\end{scope}
\draw[line width=1pt, color=black,->,>=stealth] (6,2)--(8,2);
\node at (7,2)[anchor=south]{$BA^{\alpha\oo k}_{b,s_L}$};
\begin{scope}[shift={(12,0)}]
\draw[line width=1.5pt, color=red, ->,>=stealth] (-4,0)--(-2,0);
\draw[line width=1.5pt, color=red] (-2,0)--(0,0);
\draw[line width=1.5pt, color=red, ->,>=stealth] (0,0)--(2,0);
\draw[line width=1.5pt, color=red] (2,0)--(5,0);
\draw[line width=1pt, color=black,->,>=stealth] (0,4)--(0,2);
\draw[line width=1pt, color=black,] (0,2) node[anchor=west]{$k_{(2)}c_{(2)}$}--(0,0);
\draw[line width=1pt, color=black,->,>=stealth] (0,4)--(2,4) node[anchor=south]{$d_{(1)}$};
\draw[line width=1pt, color=black,] (2,4)--(4,4);
\draw[line width=1pt, color=black,->,>=stealth] (4,4)--(4,2) node[anchor=west]{$e_{(1)}$};
\draw[line width=1pt, color=black,] (4,2)--(4,0);
\draw[line width=1pt, color=blue,->,>=stealth] (0,-2.5)node[anchor=north]{$h$}--(0,-1.5);
\draw[line width=1pt, color=blue,] (0,-1.5)--(0,0);
\draw[line width=1pt, color=black](0,0)--(.5,.5) node[anchor=west]{$s_L$};
\draw[color=red, fill=red] (0,0) circle (.2);
\draw[color=red, fill=red] (4,0) circle (.2);
\draw[color=black, fill=black] (0,4) circle (.2);
\draw[color=black, fill=black] (4,4) circle (.2);
\node at (2,0) [anchor=north]{$f_{(2)}S(k_{(4)})\oo {\color{blue}g}$};
\node at (-2,0)[anchor=north]{$k_{(3)}a\oo{\color{blue} b}$};
\node at (2,-4.5)[anchor=south]{$\langle \alpha, k_{(5)}S(f_{(1)})e_{(2)}d_{(2)}S(c_{(1)})S(k_{(1)})\rangle$};
\end{scope}
\end{tikzpicture}
\end{center}
\caption{Action of the vertex and face operator at a defect vertex.}
\label{fig:vertexdefect}
\end{figure}
Note that  the first part of Definition \ref{def:vertexfacedefect} considers general sites of $\Gamma$, and just specifies the bulk region, in which they are located. This is not to be confused with the \emph{bulk sites} from Definition \ref{def:sitesdefect}.  It  follows directly from   Definition \ref{def:edgeopsdefect} of the triangle operators and  Definition \ref{definition:HolonomyCompositeDefect} of the holonomies, that the vertex and face operators from Definition \ref{def:vertexfacedefect} reduce to the ones from Definition \ref{definition:VertexFaceOperators}, whenever the site $s$ is a bulk site. Together with the definition of the extended space in \eqref{eq:hilbdef} this yields the following example.

	\begin{example}
		\label{example:StandardModelSpecialCase}
		If the ribbon graph $\Gamma$ only has a single bulk region and no defect or boundary lines, then the extended space, the vertex and face operators and the holonomies of the Kitaev model with boundaries and defects coincide with those of the Kitaev model without boundaries and defects.
	\end{example}

	After defining the the generalized vertex and face operators, we can now show  that our model satisfies the condition~(O\ref{O1}) and~(O\ref{O2}) from Section \ref{section:TranslationKitaev}:
\begin{proposition}
	\label{proposition:OperatorActionDefect}
	Let $s,t$ be sites in $\Gamma$, $b\in B$ a bulk region, $a\in A$ a boundary line adjacent to $b$ and $d\in D$ a defect line adjacent to $b$. Then
	\begin{compactenum}
		\item $BA_{b,s}: D(H_{b}) \oo \HSpace\to \HSpace$ defines an action of $D(H_{b})$, if $s$ is a bulk site in $b$, a boundary site in $a$ or a defect site in $b$.
			\label{item1}
		\item $BA_{d,s_{L},s_{R}}: D(H_{b_{dL}}) \oo D(H_{b_{dR}}) \oo \HSpace \to \HSpace$ defines an action of $D(H_{b_{dL}}) \oo D(H_{b_{dR}})$, if $(s_{L},s_{R})$ is a pair of defect sites in $d$.
			\label{item3}
	\end{compactenum}
\end{proposition}
\begin{proof}
	Proof of~\ref{item1}:
	For a bulk site or a boundary site $s$, the operators $A_{b,s},B_{b,s}$ are defined identically to the ones of a Kitaev model without defects and boundaries. They thus define an action of $D(H_{s})$, as proven in \cite{BMCA}. The proof is analogous for defect sites $s$, as they only act on the defect edges and the bulk edges in $b$ (cf. Figure~\ref{fig:vertexdefect}). The statement~\ref{item3} follows as a corollary of~\ref{item1} by using~\eqref{eq:FaceVertexOperatorDefect}.
\end{proof}

We also obtain a direct generalization of the second statement in Theorem \ref{ht:siterep} on the vertex and face operators of disjoint sites. An analogous result holds for pairs of defect sites at a defect vertex.

\begin{proposition}
	Let $s,t$ be sites in bulk regions $b_{s}$ and $b_{t}$ and let $h\in D(H_{b_{s}}), k\in D(H_{b_{t}})$. Then $BA_{b_{s},s}^{h}$ and $BA_{b_{t},t}^{k}$ commute, if either
	\begin{itemize}
		\item $s$ and $t$ are disjoint sites.
		\item $(s,t)$ is a pair of defect sites.
\end{itemize}
\label{proposition:OperatorsCommute}
\end{proposition}
\begin{proof}
	The proof for the first case is analogous to the one for Kitaev models without defects and boundaries, see for instance \cite[Lemma~5.9]{Me}. 
	For the second case we note that for a pair of defect sites $(s,t)$ the map $BA_{b_{s},s}^{h}$ only acts on copies of $H_{b_{s}}$ in the tensor product $\HSpace = \tensor_{e\in E} H_{e}$, while $BA_{b_{t},t}^{k}$ only acts on copies of $H_{b_{t}}$. They thus commute trivially.
\end{proof}

We will now investigate how holonomies along permissible paths in the thickening of a ribbon graph with defects and boundaries generate and fuse excitations at their endpoints. The first step   is to study the interaction of these holonomies with the vertex and face operators for sites at the endpoints of the path. This yields a direct generalization of Lemma
 \ref{lemma:RibbonHolonomyCommutators} for models with defects. Just as Lemma~\ref{lemma:RibbonHolonomyCommutators} it shows that the action of its holonomy generates two inverse excitations at the path's start and target site and fuses them with the excitations already present at those sites.

\begin{corollary}
	Let $\rho:s_{1}^{\eta_{1}}\to s_{2}^{\eta_{2}}$ be a permissible path in the bulk region $b$ and adjacent defects and boundaries. Then the identities from Lemma~\ref{lemma:RibbonHolonomyCommutators} hold, i.e. we have
	\begin{align}
		&BA_{b,s_{1}}^{k}\circ \Hol_{b,\rho}^{\alpha} =  \Hol_{b,\rho}^{ k_{(2)}  \vartriangleright \alpha  } \circ BA_{b,s_{1}}^{k_{(1)}}, \quad \text{if $\eta_{1}=R$}
		\label{eq:HolonomyCommutatorsStartRightDefect}
		\\
		&BA_{b,s_{1}}^{k}\circ \Hol_{b,\rho}^{\alpha} =  \Hol_{b,\rho}^{k_{(1)}   \vartriangleright\alpha } \circ BA_{b,s_{1}}^{k_{(2)}} , \quad \text{if $\eta_{1}=L$}
		\label{eq:HolonomyCommutatorsStartLeftDefect}
		\\
		&BA_{b,s_{2}}^{k} \circ \Hol_{b,\rho}^{\alpha} =  \Hol_{b,\rho}^{   \alpha  \vartriangleleft S\left( k_{(1)} \right)} \circ BA_{b,s_{2}}^{k_{(2)}},\quad \text{if $\eta_{2}=L$}
		\label{eq:HolonomyCommutatorsEndLeftDefect}
		\\
		&BA_{b,s_{2}}^{k} \circ \Hol_{b,\rho}^{\alpha} =  \Hol_{b,\rho}^{\alpha \vartriangleleft S(k_{(2)})} \circ BA_{b,s_{2}}^{k_{(1)}},\quad \text{if $\eta_{2}=R$}.
		\label{eq:HolonomyCommutatorsEndRightDefect}
	\end{align}
	\label{corollary:DefectBoundaryHolonomyFaceVertex}
\end{corollary}
\begin{proof}
	The proof is analogous to the proof of Lemma~\ref{lemma:RibbonHolonomyCommutators}, since we can again write $A_{b,s_{i}}^{k}$ and $B_{b,s_{i}}^{\alpha}$ as holonomies along the face and vertex path of $s_{i}$. These paths again form left and right joints with the holonomies. Applying~\eqref{eq:LeftJoint} and~\eqref{eq:RightJoint} then produces the identities above.
\end{proof}

Corollary~\ref{corollary:DefectBoundaryHolonomyFaceVertex} shows, that the maps $\Hol_{b,\rho}^{ \alpha}$ for a simple path $\rho$ do not satisfy the conditions (H\ref{H2}) and (H\ref{H3}) from Section \ref{section:TranslationKitaev} for paths that start or end at boundary or defect sites.
This is not surprising, since in the definition of the holonomies we did not take into account the twists associated with defect and boundary components.  

For instance, a path $\rho:s_{1}^{L}\to s_{2}$ starting to the left of a boundary site $s_{1}$ defines a module homomorphism $D(H_{b})^{*} \oo \HSpace\to \HSpace$, where $ \oo $ is the tensor product of $D(H_{b})$-modules (cf. Corollary~\ref{corollary:HolonomyFusion}).
Instead we require module homomorphisms $D(H_{b})^{*} \oo_{F} \HSpace\to \HSpace$, where $F$ is the twist at $s_{1}$ and $ \oo_{F}$ the tensor product of $D(H_{b})_{F}$-modules. 

By Proposition \ref{proposition:FunctorTranslation}, the $D(H_b)$-modules   $D(H_b)^*\oo\HSpace$ and $D(H_b)^*\oo_F \HSpace$ are isomorphic as $D(H_b)$-modules and the isomorphism is given by an action of the twist $F$. For the holonomies, we have to use different twists for the cases where the path starts or ends at the left or right of a defect or boundary site.
This is due to the fact that  these cases are associated with four different $D(H_{b})$-module structures  on $D(H_{b})^{*} \oo \HSpace$ that are given in
Corollary~\ref{corollary:HolonomyFusion}.
For each of these $D(H_b)$-module structures,  there is a different isomorphism relating the tensor product to a twisted tensor product.

We start with paths  that start or end at a boundary site.
For this let $\rho:s^{\sigma}\to t^{\tau}$ a permissible path with $\sigma,\tau\in \left\{ L,R \right\}$ that starts or ends in the boundary region labeled with the twist $F$ of $D(H_{b})$
and denote $F_{L}=F$ and $F_{R}= F_{21}= F^{(2)} \oo F^{(1)}$.
\begin{itemize}
		\item If the starting site $s^{\sigma}$ is boundary site, we define
	\begin{align}\label{eq:twistphidef}
		\varphi_{s_{\sigma}}: &D(H_{b})^{*} \oo_{F} \HSpace \to D(H_{b})^{*}\oo \HSpace,\quad
		\alpha \oo n \mapsto  \left(F_{\sigma}^{(-1)} \vartriangleright \alpha \right)  \oo  BA_{b,s}^{F_{\sigma}^{(-2)}}(n)
	\end{align}
	\item If the target site $t^{\tau}$, is a boundary site, we define
		\begin{align}\label{eq:twistpsidef}
			\psi_{t_{\tau}}: &D(H_{b})^{*} \oo_{F} \HSpace \to D(H_{b})^{*}\oo \HSpace,\quad
			\alpha \oo n \mapsto  \left( \alpha  \vartriangleleft S\left( F_{\tau}^{(-1)} \right)\right) \oo  BA_{b,t}^{F_{\tau}^{(-2)}}(n)
		\end{align}
\end{itemize}

We  now consider paths that start or end in a defect site. For this let $\rho:s^{\sigma}\to t^{\tau}$ a permissible path with $\sigma,\tau\in \left\{ L,R \right\}$ that starts or ends on a defect $d$ labeled with the twist $F$ of $D(H_{b_{dL}})\oo D(H_{b d_R})$. Then  $s\in\{s_R, s_L\}$,  where $(s_{L},s_{R})$ is a pair of defect sites in  $d$, or 
$t\in\{t_L, t_R\}$  for a pair $(t_{L},t_{R})$  of defect sites in $d$. 
 Abusing notation we write  $ \vartriangleright, \vartriangleleft$  for the left and right coregular action of $D(H_{b_{dL}}) \oo D(H_{b_{dR}})$ on the submodules $D(H_{b_{dL}})^{*} \oo 1 , 1 \oo  D(H_{b_{dR}})^{*} \subseteq D(H_{b_{dL}})^{*} \oo D(H_{b_{dR}})^{*}$.
\begin{itemize} 
\item If the starting site $s^{\sigma}$ is a defect site, we define
\begin{align}\label{eq:twistphidef2}
	\varphi_{s_{\sigma}}: &D(H_{b})^{*} \oo_{F} \HSpace \to D(H_{b})^{*} \oo \HSpace,\quad
	\alpha \oo n \mapsto  \left(F_{\sigma}^{(-1)} \vartriangleright \alpha \right)  \oo  BA_{d,s_{L},s_{R}}^{F_{\sigma}^{(-2)}}(n),
\end{align}
\item  If the target site $t^{\tau}$ is a defect site, we define
\begin{align}\label{eq:twistpsidef2}
	\psi_{t_{\tau}}: &D(H_{b})^{*} \oo_{F} \HSpace \to D(H_{b})^{*} \oo \HSpace,\quad
	\alpha \oo n \mapsto  \left( \alpha  \vartriangleleft S\left(F_{\tau}^{(-1)}\right)\right) \oo  BA_{d,t_{L},t_{R}}^{F_{\tau}^{(-2)}}(n)
\end{align}
\end{itemize}

To a permissible path $\rho:s^{\sigma}\to t^{\tau}$  with $\sigma,\tau\in \left\{ L,R \right\}$ that starts or ends on a bulk site, we assign the identity morphism. More specifically
\begin{itemize}
\item If the starting site $s^\sigma$ is a bulk site, we set
\begin{align}\label{eq:twistphidef3}
	\varphi_{s_{\sigma}}=\psi_{s_{\sigma}} = \id :D(H_{b})^{*}  \oo \HSpace \to D(H_{b})^{*} \oo \HSpace
\end{align}
\item If the target site $t^\tau$ is a bulk site we set
\begin{align}\label{eq:twistpsidef3}
	\psi_{t_{\tau}} = \id :D(H_{b})^{*}  \oo \HSpace \to D(H_{b})^{*} \oo \HSpace.
\end{align}
\end{itemize}

We thus assigned different isomorphisms to all possible constellations of endpoints of permissible paths in $D(\Gamma)$. The twisted holonomy is then obtained by modifying the holonomy of a path $\rho$ by applying the isomorphisms associated with its starting and target end. 

\begin{definition}
	\label{definition:HolonomyTwisted}
	Let $\rho:s^{\sigma}\to t^{\tau}$ a permissible path in the bulk region $b$ with $\sigma,\tau\in \left\{ L,R \right\}$.  The \emph{twisted holonomy} along $\rho$ is the map
	\begin{align}
		\widetilde{\Hol}_{b,\rho} = \Hol_{b,\rho} \circ\left( \varphi_{s^{\sigma}}\circ\psi_{t^{\tau}} \right) 
	\label{eq:HolonomyTwistedDefinition}
	\end{align}
\end{definition}

\begin{example}
If  $\rho:s^{L}\to t^{R}$, then 
\begin{align}
	\label{eq:HolonomyTwistedLeftRight}
	\widetilde{\Hol}_{b,\rho}^{\alpha} =
	\Hol_{b,\rho}^{ F^{(-1)} \vartriangleright\alpha \vartriangleleft S(G^{(-2)})} BA_{s}^{F^{(-2)}} BA_{t}^{G^{(-1)}} 
\end{align}
where either (i) $F$ is the twist at $s$ and $BA_{s}=BA_{b,s}$ if $s$ is a bulk or boundary site, or (ii) $F$ is the twist $F_{d}$ and $BA_{s}=BA_{d,s_{L},s_{R}}$ if $s$ is a defect site in $d$, and similar for $BA_{t}$.
If instead $\rho$ starts to the right of $s$, one has to exchange $F^{(-1)}$ and $F^{(-2)}$ and if $\rho$ ends to the left of $t$ one has to exchange $G^{(-1)}$ and $G^{(-2)}$ in~\eqref{eq:HolonomyTwistedLeftRight}.
\end{example}

  It is clear from the definitions that the twisted holonomies along a path $\rho$ coincide with the untwisted ones whenever the path $\rho$ starts and ends at a bulk site. 

  We now give a counterpart of Corollary \ref{corollary:HolonomyModuleHom} for the \emph{twisted} holonomies from Definition \ref{definition:HolonomyTwisted}. 
  It shows that the twisted holonomies generate excitations at the endpoints of the path $\rho$. If these sites already carry excitations, they are fused with the newly generated excitations. If the site in question is a boundary or defect site, then the fusion uses a twisted tensor product, i.e. the twisted holonomies satisfy the conditions~(H\ref{H1}) to~(H\ref{H3}) from Section \ref{section:TranslationKitaev}.
%
%
%
\begin{proposition}
	\label{proposition:HolonomyTwistedCondition}
	Let $\rho:s_{1}\to s_{2}$ a permissible path between disjoint sites in a bulk region or adjacent defects and boundaries and denote $F$ the twist at $s_{1}$ and $G$ the twist at $s_{2}$.  Then $\widetilde{\Hol}_{b,\rho}$ defines module homomorphisms
\begin{align}
	D(H_{b})^{*}_{ \vartriangleright}  \oo_{F} \HSpace_{s_{1}} \to \HSpace_{s_{1}} \quad&\text{if $\rho$ starts to the left of $s_{1}$}\\
	 D(H_{b})^{*}_{ \vartriangleright} \oo_{F}^{cop}\HSpace_{s_{1}}  \to \HSpace_{s_{1}}\quad&\text{if $\rho$ starts to the right of $s_{1}$}\\
	D(H_{b})^{*}_{ \vartriangleleft}  \oo_{G} \HSpace_{s_{2}} \to \HSpace_{s_{2}}\quad&\text{if $\rho$ ends to the left of $s_{2}$}\\
	 D(H_{b})^{*}_{ \vartriangleleft} \oo_{G}^{cop} \HSpace_{s_{2}} \to \HSpace_{s_{2}}\quad&\text{if $\rho$ ends to the right of $s_{2}$}
\end{align}
	Here  $\HSpace_{s_{i}}$ is the extended space with the action defined by the site $s_{i}$, $D(H_b)^{*}_{ \vartriangleright}$ is the vector space $D(H_b)^{*}$ with the left regular action of $D(H_b)$ and $D(H_b)^{*}_{ \vartriangleleft}$ is the vector space $D(H_b)^{*}$ with the left action of $D(H_b)$ defined by~\eqref{eq:RightCoregularLeftAction}.
\end{proposition}

\begin{proof}
	We prove the first statement for the case where $s_{1}$ is a boundary site and and $\rho$ starts to the left of $s_{1}$ and ends to the right of $s_{2}$.
	The other cases and the proof for defect sites is analogous.
	We compute 	
	\begin{align*}
		BA_{b,s_{1}}^{h}\circ \widetilde{\Hol}_{b,\rho}^{\alpha} 
		&\overset{\eqref{eq:HolonomyTwistedLeftRight}}{=} 
		BA_{b,s_{1}}^{h}\Hol_{b,\rho}^{ F^{(-1)} \vartriangleright\alpha \vartriangleleft S(G^{(-2)})} BA_{b,s_{1}}^{F^{(-2)}} BA_{b,s_{2}}^{G^{(-1)}} 
		\\&
			\overset{\eqref{eq:HolonomyCommutatorsStartLeftDefect}}{=}
		\Hol_{b,\rho}^{ \left(h_{(1)} F^{(-1)}\right) \vartriangleright\alpha \vartriangleleft S(G^{(-2)})} 
		BA_{b,s_{1}}^{h_{(2)}} BA_{b,s_{1}}^{F^{(-2)}} BA_{b,s_{2}}^{G^{(-1)}} 
		\\&=
		\Hol_{b,\rho}^{ \left(F^{(-3)}F^{(1)}h_{(1)} F^{(-1)}\right) \vartriangleright\alpha \vartriangleleft S(G^{(-2)})} 
		BA_{b,s_{1}}^{F^{(-4)}}BA_{b,s_{1}}^{F^{(2)}h_{(2)}F^{(-2)}} BA_{b,s_{2}}^{G^{(-1)}} 
		\\&
		\overset{\ref{lemma:TwistedHopfAlgebra}}{=}
		\Hol_{b,\rho}^{ \left(F^{(-3)} h_{(F1)}\right) \vartriangleright\alpha \vartriangleleft S(G^{(-2)})} 
		BA_{b,s_{1}}^{F^{(-4)}}BA_{b,s_{1}}^{h_{(F2)}} BA_{b,s_{2}}^{G^{(-1)}}
		\\
		&\overset{\eqref{eq:HolonomyTwistedLeftRight}}{=} 
		\widetilde{\Hol}_{b,\rho}^{  h_{(F1)} \vartriangleright\alpha} 
		BA_{b,s_{1}}^{h_{(F2)}}
	\end{align*}
	From this we conclude, that $\widetilde{\Hol}_{b,\rho}$ is a module homomorphism from $D(H_{b})^{*}_{ \vartriangleright}  \oo_{F} \HSpace_{s_{1}}$ to $\HSpace_{s_{1}}$.
\end{proof}

\subsection{Moving and braiding of excitations}
\label{subsec:fusionbraiding}

The movement and braiding of excitations of the Kitaev model was first considered in \cite{Ki} for the model based on the group algebra of a finite group. The article \cite{BMCA} comments on the generalization to a finite-dimensional semisimple Hopf algebra $H$, but does not give a detailed and explicit description. 

In this section we describe how to move and braid excitations in the Kitaev model with topological defects and boundaries based on semisimple finite-dimensional Hopf algebras. 
Our description is based on the rules for fusion and braiding of excitations in a Turaev-Viro TQFT with topological defects and boundaries from~\cite{FSV} outlined in Section~\ref{section:FSVSummary} and  translated into conditions for our model in Section~\ref{section:TranslationKitaev}.
Our results also apply to a Kitaev model without defects and boundaries based on a semisimple finite-dimensional Hopf algebra, as this is just a special case (cf. Example~\ref{example:StandardModelSpecialCase}).

Moving an excitation along a path $\rho:s_{1}\to s_{2}$ in $D(\Gamma)$ in a bulk region $b$ from the site $s_{1}$ to $s_{2}$ should leave no excitation at the site $s_{1}$. Holonomies do not satisfy this condition, as they generate excitations at both endpoints of $\rho$. 
This is remedied by applying the Haar integrals of  $D(H_b)$ and $D(H_b)^*$. When inserted into the holonomy along $\rho$, the Haar integral of $D(H_b)^{*}$ generates pairs of dual excitations at the sites $s_1$ and $s_2$ and fuses them with the existing excitations at these sites. The Haar integral of $D(H_b)$ then destroys the excitation at $s_1$ and moves it to $s_2$.

\begin{definition}
	\label{definition:TransportOperator}
	Let $\rho:s_{1}\to s_{2}^{\eta}$ a permissible path in the bulk region $b$ with $\eta\in\left\{ L,R \right\}$, $F$ the twist at $s_{2}$ and $\lambda\in D(H_{b}), \smallint \in D(H_{b})^{*}$ the Haar integrals of $D(H_{b})$ and $D(H_{b})^{*}$.
	The \emph{transport operator} is the linear map map
	\begin{align}
		&T_{\rho}:\HSpace\to \HSpace,\quad
		m\mapsto  BA_{b,s_{1}}^{\lambda}\circ \Hol_{b,\rho}^{\smallint \vartriangleleft S(F^{(-2)})} \circ BA_{b,s_{2}}^{F^{(-1)}} (m), \quad\text{if $\eta=R$}
		\label{eq:TransportOperatorBulkRight}
		\\
		&T_{\rho}:\HSpace\to \HSpace,\quad
		m\mapsto  BA_{b,s_{1}}^{\lambda}\circ \Hol_{b,\rho}^{\smallint \vartriangleleft S(F^{(-1)})} \circ BA_{b,s_{2}}^{F^{(-2)}} (m) , \quad\text{if $\eta=L$}
		\label{eq:TransportOperatorBulkLeft}
	\end{align}
\end{definition}

%
%
We now present a concrete example for the transport operator in a Kitaev model without defects and boundaries.
\begin{example}
	Let $\lambda\in H$, $\smallint\in H^{*}$ be the Haar integrals of $H$ and $H^{*}$ and consider the  graph below with a single bulk region $b$ and three edges, labeled with elements $a,b,c \in H$. Then the transport operator can be computed by first applying the operator $\Hol_{\rho}^{\lambda \oo \smallint}$: 
	\begin{align*}
	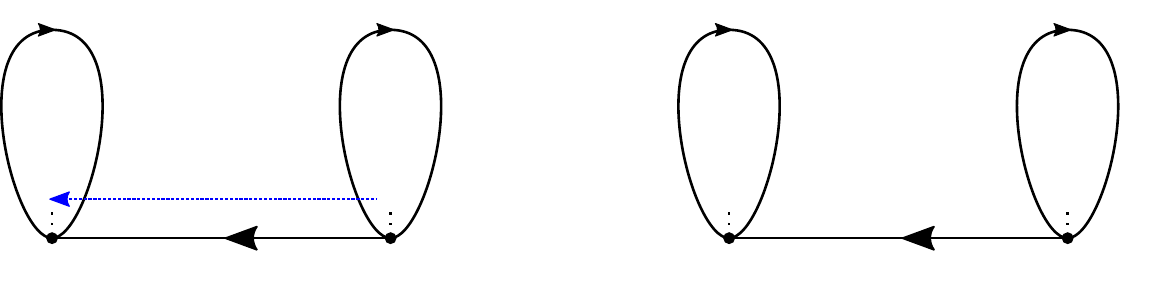
	\end{align*}
	Applying the operators $B^{\smallint}_{s_{1}}$ and $A^{\lambda}_{s_{1}}$ afterwards we obtain
	\begin{align*}
	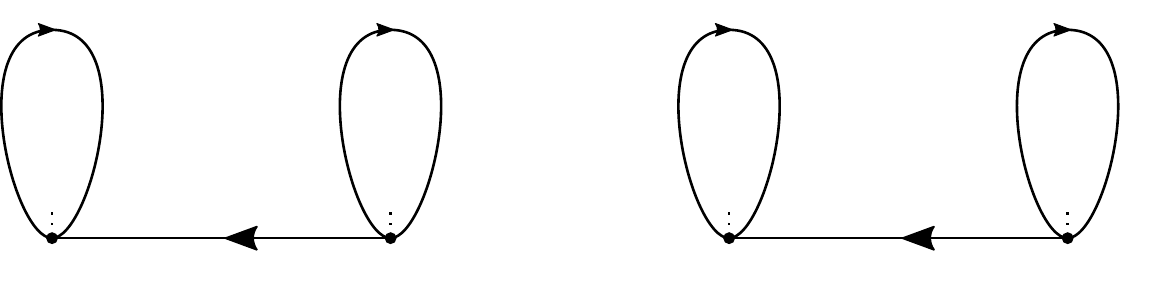
	\end{align*}
	Here we used the identity $\smallint(\lambda_{(2)}c) S(\lambda_{(1)}) = c$, which can be derived from the properties of the Haar integrals $\lambda$ and $\smallint$. 
\end{example}

In the remainder of this section we prove that the transport operator satisfies the conditions~(T\ref{T1}) and~(T\ref{T2}) from Section \ref{section:TranslationKitaev}. 
For this, we first show a useful commutation property of the transport operator and holonomies. Commuting the transport operator $T_{\rho}$ with a holonomy along the path $\gamma$ extends the path $\gamma$ to $\rho\circ\gamma$ - under suitable conditions on the paths.
\begin{lemma}
	\label{lemma:HolonomyExtension}
	Let $\rho:s_{1}\to s_{2}, \gamma:s_{0}\to s_{1}$ and $\rho\circ\gamma:s_{0}\to s_{2}$ be permissible paths between disjoint sites such that both $\gamma$ and $\rho\circ\gamma$ end to the right of a site and either (cf. Figure~\ref{fig:AllowedCompositions})
	\begin{compactenum}
		\item $\rho$ starts to the left of $s_{1}$, or
		\item $(\gamma,\rho^{-1})_{\prec}$, or
		\item $\rho^{-1}$ is a subpath of $\gamma$.
	\end{compactenum}
	Denote $F$ the twist at $s_{2}$. Then 
	\begin{align}
		T_{\rho}\circ  \Hol_{b,\gamma}^{ \alpha \vartriangleleft S(F^{(-2)})} \circ BA_{b,s_{1}}^{F^{(-1)}}  	
		&= \Hol_{b,\rho\circ\gamma}^{\alpha  \vartriangleleft S(F^{(-2)})}\circ BA_{b,s_{2}}^{F^{(-1)}} \circ T_{\rho}  
		\label{eq:HolonomyStretching}
	\end{align}
\end{lemma}

\begin{figure}[H]
\begin{center}
\begin{tikzpicture}[scale=1]
\draw[line width=1pt, ->,>=stealth](0,-2)--(2,-2);
\draw[line width=1pt, ->,>=stealth](0,0)--(2,0);
\draw[line width=1pt, ->,>=stealth](0,2)--(2,2);
\draw[line width=1pt, ->,>=stealth](0,-2)--(0,-1);
\draw[line width=1pt, ->,>=stealth](0,-1)--(0,1);
\draw[line width=1pt](0,1)--(0,2);
\draw[line width=1.5pt, color=red,->,>=stealth] (0,3)--(0,2.5);
\draw[line width=1.5pt, color=red] (0,2)--(0,2.5);
\draw[line width=1.5pt, color=red,->,>=stealth] (0,2)--(-1,2);
\draw[line width=1pt, color=black](0,-2)--(.3,-1.7)node[anchor=west]{$s_0$};
\draw[line width=1pt, color=black](0,0)--(.3,.3)node[anchor=west]{$s_1$};
\draw[line width=1pt, color=black](0,2)--(.3,2.3)node[anchor=south west]{$s_2$};
\draw[line width=1pt, color=blue,->,>=stealth] (.6,-1.4)--(.6,-.5) node[anchor=west]{$\gamma$};
\draw[line width=1pt, color=blue] (.6,-.5)--(.6,.1);

\draw[line width=1pt, color=violet,->,>=stealth] (.6,.6)--(.6,1.5) node[anchor=west]{$\rho$};
\draw[line width=1pt, color=violet] (.6,1.5)--(.6,2.2);
\draw[color=red, fill=red] (0,2) circle (.13) ;
\draw[color=black, fill=black] (0,0) circle (.13) ;
\draw[color=black, fill=black] (0,-2) circle (.13) ;
\end{tikzpicture}
\qquad
\qquad
\begin{tikzpicture}[scale=1]
\draw[line width=1pt, color=black,->,>=stealth](0,0)--(0,1);
\draw[line width=1pt, color=black,->,>=stealth](0,1)--(0,3);
\draw[line width=1pt, color=black, ->,>=stealth](0,0)--(1,0);
\draw[line width=1pt, color=black,->,>=stealth](0,0)--(-1,0);
\draw[line width=1pt, color=black,](-2,0)--(-1,0);
\draw[line width=1pt, color=black, ->,>=stealth](0,2)--(1,2);
\draw[line width=1.5pt, color=red,->,>=stealth] (-2,1.5)--(-2,1);
\draw[line width=1.5pt, color=red,->,>=stealth] (-2,1)--(-2,-1);
\draw[line width=1pt, color=black] (-2,0)--(-1.7,.3) node[anchor=south]{$s_2$};
\draw[line width=1pt, color=black] (0,0)--(.3,.3) node[anchor=west]{$s_0$};
\draw[line width=1pt, color=black] (0,2)--(.3,2.3) node[anchor=south]{$s_1$};
\draw[line width=1pt, color=blue,->,>=stealth] (.7,.6)--(.7, 1.5) node[anchor=west]{$\gamma$};
\draw[line width=1pt, color=blue,] (.7,1.5)--(.7, 2.3);
\draw[line width=1pt, color=violet,->,>=stealth] (.4,2.3)--(.4, 1);
\draw[line width=1pt, color=violet,] (.4,1)--(.4, .5);
\draw[line width=1pt, color=violet,->,>=stealth] (.4,.5)--(-1, .5) node[anchor=south]{$\rho$};
\draw[line width=1pt, color=violet,] (-1,.5)--(-1.4, .5);
\draw[color=red, fill=red] (-2,0) circle (.13) ;
\draw[color=black, fill=black] (0,0) circle (.13) ;
\draw[color=black, fill=black] (0,2) circle (.13) ;
\end{tikzpicture}
\qquad
\qquad
\begin{tikzpicture}[scale=1]
\draw[line width=1pt, color=black, ->,>=stealth] (1.5,-2)--(1,-2);
\draw[line width=1pt, color=black] (1,-2)--(0,-2);
\draw[line width=1pt, color=black, ->,>=stealth] (1.5,0)--(1,0);
\draw[line width=1pt, color=black] (1,0)--(0,0);
\draw[line width=1pt, color=black, ->,>=stealth] (1.5,2)--(1,2);
\draw[line width=1pt, color=black] (1,2)--(0,2);
\draw[line width=1pt, color=black, ->,>=stealth] (0,-2)--(0,-1);
\draw[line width=1pt, color=black] (0,-1)--(0,0);
\draw[line width=1.5pt, color=red, ->,>=stealth] (0,3)--(0,1);
\draw[line width=1.5pt, color=red] (0,1)--(0,0);
\draw[line width=1.5pt, color=red, ->,>=stealth] (0,0)--(-1,0);
\draw[line width=1.5pt, color=red] (-1,0)--(-1.5,0);
\draw[line width=1pt, color=black] (0,-2)--(.3,-1.7)node[anchor=west]{$s_0$};
\draw[line width=1pt, color=black] (0,0)--(.3,.3)node[anchor=south]{$s_2$};
\draw[line width=1pt, color=black] (0,2)--(.3,2.3) node[anchor=south west]{$s_1$};
\draw[line width=1pt, color=blue, ->,>=stealth] (.6,-1.5)--(.6,-.7) node[anchor=west]{$\gamma$};
\draw[line width=1pt, color=blue, ] (.6,-.7)--(.6,1) ;
\draw[line width=1pt, color=blue, ] (.6,1)--(.6,2.2);
\draw[line width=1pt, color=violet, ->,>=stealth] (.9,2.2)--(.9,1.5) node[anchor=west]{$\rho$};
\draw[line width=1pt, color=violet, ] (.9,1.5)--(.9,.2);
\draw[color=black, fill=black] (0,-2) circle (.13);
\draw[color=red, fill=red] (0,0) circle (.13);
\draw[color=red, fill=red] (0,2) circle (.13);
\end{tikzpicture}
\end{center}
\caption{From left to right: Paths $\gamma:s_{0}\to s_{1}, \rho:s_{1}\to s_{2}$ such that 1. $\rho$ starts to the left of $s_{1}$, 2.$(\gamma,\rho^{-1})_{\prec}$, 3. $\rho^{-1}$ is a subpath of $\gamma$.}
\label{fig:AllowedCompositions}
\end{figure}
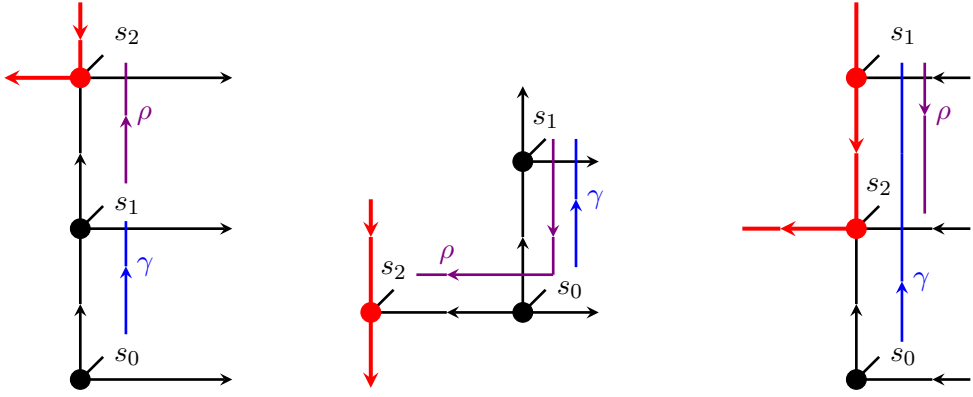
\begin{proof}
	We first make some observations about $\rho$ in case~1 and case~3. 
	If $\rho$ starts to the left of $s_{1}$, then $(\rho,\gamma)_{\emptyset}$, since $\rho\circ\gamma$ would not be a simple path otherwise.
	If $\rho^{-1}$ is a subpath of $\gamma$, then $(\rho, \rho\circ\gamma)_{\emptyset}$, since $\rho\circ\gamma$ would end to the left of $s_{2}$ otherwise.
	
	$\bullet$ Step 1: ~We first show that in all three cases we have
	\begin{align}
		\Hol_{b,\rho}^{\smallint  \vartriangleleft S(h)} \Hol_{b,\gamma}^{\alpha} 
		=\Hol_{b,\rho\circ\gamma}^{\alpha \vartriangleleft S(h_{(2)})} \Hol_{b,\rho}^{\smallint  \vartriangleleft S(h_{(1)})},
		\label{eq:TransportHelper1}
	\end{align}
	where $\lambda\in D(H_{b})$ and $\smallint \in D(H_{b})^{*}$ are the Haar integrals of $D(H)$ and $D(H)^{*}$ and $h\in D(H_{b})$.

1.~If $\rho$ starts to the left of $s_{1}$, we have
	\begin{align*}
		\Hol_{b,\rho}^{\smallint  \vartriangleleft S(h)} 
		\Hol_{b,\gamma}^{\alpha } 
		&\overset{\eqref{eq:HaarIntegral}}{=}
		\Hol_{b,\rho}^{ (\alpha_{(1)} \smallint) \vartriangleleft S(h)} 
		\Hol_{b,\gamma}^{\alpha_{(2)}}
		\overset{\eqref{eq:LeftRightHolonomy}}{=}
		\Hol_{b,\rho}^{ \alpha_{(1)}\vartriangleleft S(h_{(2)})} 
		\Hol_{b,\rho}^{ \smallint \vartriangleleft S(h_{(1)})} 
		\Hol_{b,\gamma}^{\alpha_{(2)}}
		\\
		&\overset{\eqref{eq:NonOverlap}}{=}
		\Hol_{b,\rho}^{ \alpha_{(1)}\vartriangleleft S(h_{(2)})} 
		\Hol_{b,\gamma}^{\alpha_{(2)}}
		\Hol_{b,\rho}^{ \smallint \vartriangleleft S(h_{(1)})} 
		\overset{\eqref{eq:HolonomyRecursionStandard}}{=}
		\Hol_{b,\rho\circ\gamma}^{ \alpha\vartriangleleft S(h_{(2)})} 
		\Hol_{b,\rho}^{ \smallint \vartriangleleft S(h_{(1)})} 
	\end{align*}
	2.~If $(\gamma, \rho^{-1})_{\prec}$, we obtain
	\begin{align*}
		\Hol_{b,\rho}^{\smallint  \vartriangleleft S(h)} 
		\Hol_{b,\gamma}^{\alpha } 
		&\overset{\eqref{eq:HaarIntegral}}{=}
		\Hol_{b,\rho}^{\alpha_{(1)}\smallint  \vartriangleleft S(h)} 
		\Hol_{b,\gamma}^{\alpha_{(2)}}
		\overset{\eqref{eq:RightRightHolonomy}}{=}
		\Hol_{b,\rho}^{R^{(1)} \vartriangleright \alpha_{(1)} \vartriangleleft S(h_{(2)})} 
		\Hol_{b,\rho}^{R^{(2)} \vartriangleright \smallint  \vartriangleleft S(h_{(1)})} 
		\Hol_{b,\gamma}^{\alpha_{(2)}}
		\\
		&\overset{\eqref{eq:HolonomyReversal}}{=}
		\Hol_{b,\rho}^{R^{(1)} \vartriangleright \alpha_{(1)} \vartriangleleft S(h_{(2)})} 
		\Hol_{b,\rho^{-1}}^{h_{(1)} \vartriangleright S(\smallint)  \vartriangleleft  S( R^{(2)} )} 
		\Hol_{b,\gamma}^{\alpha_{(2)}}
		\\
		&\overset{\eqref{eq:LeftJoint}}{=}
		\Hol_{b,\rho}^{R^{(1)} \vartriangleright \alpha_{(1)} \vartriangleleft S(h_{(2)})} 
		\Hol_{b,\gamma}^{\alpha_{(2)} \vartriangleleft R^{(3)}}
		\Hol_{b,\rho^{-1}}^{ h_{(1)} \vartriangleright S(\smallint)  \vartriangleleft  S( R^{(2)}) R^{(4)} }
		\\
		&\overset{(*)}{=}
		\Hol_{b,\rho}^{R^{(1)} R^{(3)} \vartriangleright \alpha_{(1)} \vartriangleleft S(h_{(2)})} 
		\Hol_{b,\gamma}^{\alpha_{(2)} }
		\Hol_{b,\rho^{-1}}^{ h_{(1)} \vartriangleright S(\smallint)  \vartriangleleft  S( R^{(2)}) R^{(4)} }
		\\
		&\overset{(**)}=
		\Hol_{b,\rho}^{\alpha_{(1)} \vartriangleleft S(h_{(2)})} 
		\Hol_{b,\gamma}^{\alpha_{(2)} }
		\Hol_{b,\rho^{-1}}^{ h_{(1)} \vartriangleright S(\smallint) }
		\\
		&\overset{\eqref{eq:HolonomyRecursionStandard}}=
		\Hol_{b,\rho\circ\gamma}^{\alpha \vartriangleleft S(h_{(2)})} 
		\Hol_{b,\rho^{-1}}^{ h_{(1)} \vartriangleright S(\smallint) }
		\overset{\eqref{eq:HolonomyReversal}}{=}
		\Hol_{b,\rho\circ\gamma}^{\alpha \vartriangleleft S(h_{(2)})} 
		\Hol_{b,\rho}^{\smallint  \vartriangleleft S(h_{(1)}) },
	\end{align*}
	where we used the identity $\left( k  \vartriangleright \alpha_{(1)}  \right)  \oo \alpha_{(2)} = \alpha_{(1)} \oo  \left( \alpha_{(2)}  \vartriangleleft k \right)$ in $(*)$ and the identity $\left( \id  \oo S \right)(R)=R^{-1}$ in $(**)$.\\

	3.~If $\rho^{-1}$ is a subpath of $\gamma$, then we have
	\begin{align*}
		\Hol_{b,\rho}^{\smallint  \vartriangleleft S(h)} 
		\Hol_{b,\gamma}^{\alpha } 
		&\overset{\eqref{eq:HaarIntegral}}{=}
		\Hol_{b,\rho}^{ (\alpha_{(1)} \smallint) \vartriangleleft S(h)} 
		\Hol_{b,\gamma}^{\alpha_{(2)}}
		\overset{\eqref{eq:LeftRightHolonomy}}{=}
		\Hol_{b,\rho}^{ \alpha_{(1)}\vartriangleleft S(h_{(2)})} 
		\Hol_{b,\rho}^{ \smallint \vartriangleleft S(h_{(1)})} 
		\Hol_{b,\gamma}^{\alpha_{(2)}}
		\\
		&\overset{\eqref{eq:NonOverlap}}{=}
		\Hol_{b,\rho}^{ \alpha_{(1)}\vartriangleleft S(h_{(2)})} 
		\Hol_{b,\gamma}^{\alpha_{(2)}}
		\Hol_{b,\rho}^{ \smallint \vartriangleleft S(h_{(1)})} 
		\overset{\eqref{eq:HolonomyRecursionStandardDefect}}{=}
		\Hol_{b,\rho\circ\gamma}^{ \alpha\vartriangleleft S(h_{(2)})} 
		\Hol_{b,\rho}^{ \smallint \vartriangleleft S(h_{(1)})}. 
	\end{align*}
	$\bullet$ Step 2: We use \eqref{eq:TransportHelper1} to prove~\eqref{eq:HolonomyStretching}. For this, we have three cases, namely that $\rho$ is a left-right, right-right or right-left path. The case that $\rho$ is a left-left path does not arise, since this would imply that $\rho\circ\gamma$ ends to the left of  a site, in contradiction to the assumptions. 
	The other  three cases correspond to case 1.~2.~and 3.~in Lemma \ref{lemma:HolonomyExtension}, respectively, and are depicted in Figure \ref{fig:AllowedCompositions}.
	 
	 1.~If $\rho$ is a left-right path, we have
	\begin{align*}
		T_{\rho}\circ \Hol_{b,\gamma}^{ \alpha  \vartriangleleft S( F^{(-2)})}BA_{b,s_{1}}^{F^{(-1)}}
		&\overset{\eqref{eq:TransportOperatorBulkRight}}{=}
		BA^{\lambda}_{b,s_{1}}
		\Hol_{b,\rho}^{\smallint  \vartriangleleft S( F^{(-4)})} 
		BA_{b,s_{2}}^{F^{(-3)}} 
		\Hol_{b,\gamma}^{ \alpha  \vartriangleleft S(F^{(-2)})}
		BA_{b,s_{1}}^{F^{(-1)}}
		\\
		&\overset{\eqref{eq:NonOverlap}}=
		BA^{\lambda}_{b,s_{1}}
		\Hol_{b,\rho}^{\smallint  \vartriangleleft S(F^{(-4)})} 
		\Hol_{b,\gamma}^{ \alpha  \vartriangleleft S(F^{(-2)})}
		BA_{b,s_{2}}^{F^{(-3)}} 
		BA_{b,s_{1}}^{F^{(-1)}}
		\\
		&\overset{\eqref{eq:TransportHelper1}}{=}
		BA^{\lambda}_{b,s_{1}}
		\Hol_{b,\rho\circ\gamma}^{ \alpha  \vartriangleleft S(F^{(-4)}_{(2)}F^{(-2)})}
		\Hol_{b,\rho}^{\smallint  \vartriangleleft S(F^{(-4)}_{(1)})} 
		BA_{b,s_{2}}^{F^{(-3)}} 
		BA_{b,s_{1}}^{F^{(-1)}}
		\\
		&\overset{\eqref{eq:MiddleJoint}}=
		\Hol_{b,\rho\circ\gamma}^{ \alpha  \vartriangleleft S(F^{(-4)}_{(2)}F^{(-2)})}
		BA^{\lambda}_{b,s_{1}}
		\Hol_{b,\rho}^{\smallint  \vartriangleleft S(F^{(-4)}_{(1)})} 
		BA_{b,s_{2}}^{F^{(-3)}} 
		BA_{b,s_{1}}^{F^{(-1)}}
		\\
		&\overset{\eqref{eq:HolonomyCommutatorsStartLeft}}{=}
		\Hol_{b,\rho\circ\gamma}^{ \alpha  \vartriangleleft S(F^{(-4)}_{(2)}F^{(-2)})}
		BA_{b,s_{1}}^{\lambda F^{(-1)}_{(2)}}
		\Hol_{b,\rho}^{S(F^{-1}_{(1)})  \vartriangleright\smallint  \vartriangleleft S(F^{(-4)}_{(1)})} 
		BA_{b,s_{2}}^{F^{(-3)}}
		\\
		&\overset{\eqref{eq:HaarIntegral}}=
		\Hol_{b,\rho\circ\gamma}^{ \alpha  \vartriangleleft S(F^{(-4)}_{(2)}F^{(-2)})}
		BA^{\lambda}_{b,s_{1}}
		\Hol_{b,\rho}^{S(F^{-1})  \vartriangleright \smallint  \vartriangleleft S(F^{(-4)}_{(1)})} 
		BA_{b,s_{2}}^{F^{(-3)}}
		\\
		&\overset{(*)}{=}
		\Hol_{b,\rho\circ\gamma}^{ \alpha  \vartriangleleft S(F^{(-4)}_{(2)}F^{(-2)})}
		BA^{\lambda}_{b,s_{1}}
		\Hol_{b,\rho}^{\smallint  \vartriangleleft S( F^{(-4)}_{(1)}F^{(-1)})} 
		BA_{b,s_{2}}^{F^{(-3)}}
		\\
		&\overset{\eqref{eq:TwistDelta}}{=}
		\Hol_{b,\rho\circ\gamma}^{ \alpha  \vartriangleleft S(F^{(-4)})}
		BA^{\lambda}_{b,s_{1}}
		\Hol_{b,\rho}^{\smallint  \vartriangleleft S(F^{(-3)}_{(2)} F^{(-2)})} 
		BA_{b,s_{2}}^{F^{(-3)}_{(1)}F^{(-1)}}
		\\
		&\overset{\eqref{eq:HolonomyCommutatorsEndRight}}{=}
		\Hol_{b,\rho\circ\gamma}^{ \alpha  \vartriangleleft S(F^{(-4)})}
		BA^{\lambda}_{b,s_{1}}
		BA_{b,s_{2}}^{F^{(-3)}}
		\Hol_{b,\rho}^{\smallint  \vartriangleleft S( F^{(-2)})} 
		BA_{b,s_{2}}^{F^{(-1)}}
		\\
		&\overset{\ref{proposition:OperatorsCommute}}=
		\Hol_{b,\rho\circ\gamma}^{ \alpha  \vartriangleleft S(F^{(-4)})}
		BA_{b,s_{2}}^{F^{(-3)}}
		BA^{\lambda}_{b,s_{1}}
		\Hol_{b,\rho}^{\smallint  \vartriangleleft S (F^{(-2)})} 
		BA_{b,s_{2}}^{F^{(-1)}}
		\\
		&\overset{\eqref{eq:TransportOperatorBulkRight}}{=}
		\Hol_{b,\rho\circ\gamma}^{ \alpha  \vartriangleleft S(F^{(-4)})}
		BA_{b,s_{2}}^{F^{(-3)}}
		T_{\rho}
	\end{align*}
	where we used the identity $h  \vartriangleright \smallint = \smallint  \vartriangleleft h $ for $h\in D(H_{b})$ in (*).

	2.~and 3.:~The computations  for right-right and right-left paths  are analogous. If $\rho$ ends to the left of $s_1$, one simply has to exchange $F^{(-4)}$ and $F^{(-3)}$ and use~\eqref{eq:HolonomyCommutatorsEndLeft} instead of~\eqref{eq:HolonomyCommutatorsEndRight}. If $\rho$ starts to the right of $s_1$, one has to use~\eqref{eq:HolonomyCommutatorsStartRight} instead of~\eqref{eq:HolonomyCommutatorsStartLeft}.

\end{proof}

The following proposition shows that the transport operator $T_{\rho}$ does indeed move excitations along a permissible path $\rho:s_{1}\to s_{2}$ from $s_{1}$ to $s_{2}$. It fuses the excitations $M_{1}$ at $s_{1}$ and $M_{2}$ at $s_{2}$ via the (twisted) tensor product at the site $s_2$. In other words, 
 $T_{\rho}$ satisfies condition~(T\ref{T1}) from Section~\ref{section:TranslationKitaev}. This result can also be seen as a counterpart of~Proposition~\ref{proposition:HolonomyTwistedCondition}, which describes the fusion of the $D(H_b)$-module $D(H_b)^*$ with the excitation at $s_1$ by the holonomy along $\rho$. 
 
\begin{proposition}[Fusion]
	\label{proposition:TransportOperatorTensorProduct}
	Let $\rho:s_{1} \to s_{2}^{\eta}$ a permissible path between disjoints sites in a bulk region $b$ with $\eta\in \left\{ L,R \right\}$ and denote $F$ the twist at the site $s_{2}$.
	Then $T_{\rho}$ induces a $D(H_{b})$-linear map
	\begin{align}
		\label{eq:TransportOperatorTensorProduct}
		T_{\rho}: \HSpace\left( s_{1}, M_{1}, s_{2},M_{2} \right)
		\to
		\HSpace\left( s_{1}, \C, s_{2}, M_{2} \oo_{F} M_{1} \right) \quad\text{if $\eta=R$}.
		\\
		\label{eq:TransportOperatorTensorProductLeft}
		T_{\rho}: \HSpace\left( s_{1}, M_{1}, s_{2},M_{2} \right)
		\to
		\HSpace\left( s_{1}, \C, s_{2}, M_{1}  \oo_{F} M_{2}\right) \quad\text{if $\eta=L$}.
	\end{align}
	Here $M_{1},M_{2}$ are $D(H_{b})$-modules and $\HSpace\left( s_{1}, M_{1}, s_{2},M_{2} \right)$ is defined as in~\eqref{eq:MultipleExcitationSpace}.
\end{proposition}

\begin{proof}
	We only prove the first claim as the second claim follows by analogous computation. In this case we need to show the two identities
	\begin{align}
		BA_{b,s_{1}}^{k}\circ T_{\rho} = \varepsilon(k) T_{\rho}, \quad
		BA_{b,s_{2}}^{k} \circ T_{\rho} = T_{\rho} \circ 	BA_{b,s_{2}}^{k_{(F1)}} \circ	BA_{b,s_{1}}^{k_{(F2)}} 
		\qquad \text{for $k\in D(H)$},
		\label{eq:H2}
	\end{align}
	where $F\Delta F^{-1}: D(H)\to D(H)\oo D(H)$, $k\mapsto k_{(F1)}\oo k_{(F2)}$ is the twisted comultiplication.
	The first identity in~\eqref{eq:H2} follows by direct computation:
	\begin{align*}
		BA_{b,s_{1}}^{k}\circ T_{\rho} 
		&\overset{\eqref{eq:TransportOperatorBulkRight}}{=}
		BA_{b,s_{1}}^{k} BA_{b,s_{1}}^{\lambda} \Hol_{b,\rho}^{ \smallint  \vartriangleleft F^{(-2)}} BA_{b,s_{2}}^{F^{(-1)}}
		\\
		&= \varepsilon(k) BA_{b,s_{1}}^{\lambda} \Hol_{b,\rho}^{ \smallint  \vartriangleleft F^{(-2)}} BA_{b,s_{2}}^{F^{(-1)}}
		= \varepsilon(k) T_{\rho}
	\end{align*}
	For the second identity in~\eqref{eq:H2} we additionally assume that $\rho$ starts to the left of $s_{1}$ and ends to the right of $s_2$. In this case, we have
	\begin{align*}
		BA_{b,s_{2}}^{k}T_{\rho}
		&\overset{\eqref{eq:TransportOperatorBulkRight}}{=}
		BA_{b,s_{2}}^{k} BA_{b,s_{1}}^{\lambda} \Hol_{b,\rho}^{ \smallint  \vartriangleleft F^{(-2)}} BA_{b,s_{2}}^{F^{(-1)}}
		\overset{\ref{proposition:OperatorsCommute}}{=}
		BA_{b,s_{1}}^{\lambda} BA_{b,s_{2}}^{k} \Hol_{b,\rho}^{ \smallint  \vartriangleleft F^{(-2)}} BA_{b,s_{2}}^{F^{(-1)}}
		\\
		&\overset{\eqref{eq:HolonomyCommutatorsEndRightDefect}}{=}
		BA_{b,s_{1}}^{\lambda} \Hol_{b,\rho}^{ \smallint  \vartriangleleft S(k_{(2)} F^{(-2)})} BA_{b,s_{2}}^{k_{(1)}F^{(-1)}}
		\overset{\ref{lemma:TwistedHopfAlgebra}}{=}
		BA_{b,s_{1}}^{\lambda} \Hol_{b,\rho}^{ \smallint  \vartriangleleft S(F^{(-2)}k_{(F2)})} BA_{b,s_{2}}^{F^{(-1)}k_{(F1)}}
		\\
		&\overset{(*)}{=}
		BA_{b,s_{1}}^{\lambda} \Hol_{b,\rho}^{ S(k_{(F2)}) \vartriangleright \smallint  \vartriangleleft S(F^{(-2)})} BA_{b,s_{2}}^{F^{(-1)}k_{(F1)}}
		\\&\overset{\eqref{eq:HolonomyCommutatorsStartLeftDefect}}{=}
		BA_{b,s_{1}}^{\lambda S(k_{(F2)(2)})} \Hol_{b,\rho}^{  \smallint  \vartriangleleft S(F^{(-2)})} BA_{b,s_{1}}^{ k_{(F2)(1)}} BA_{b,s_{2}}^{F^{(-1)}k_{(F1)}}
		\\
		&\overset{\eqref{eq:HaarIntegral}}{=}
		BA_{b,s_{1}}^{\lambda } \Hol_{b,\rho}^{  \smallint  \vartriangleleft S(F^{(-2)})} BA_{b,s_{1}}^{ k_{(F2)}} BA_{b,s_{2}}^{F^{(-1)}k_{(F1)}}
		\\&\overset{\ref{proposition:OperatorsCommute}}{=}
		BA_{b,s_{1}}^{\lambda } \Hol_{b,\rho}^{  \smallint  \vartriangleleft S(F^{(-2)})}  BA_{b,s_{2}}^{F^{(-1)}}BA_{b,s_{2}}^{k_{(F1)}} BA_{b,s_{1}}^{ k_{(F2)}}
		= T_{\rho}BA_{b,s_{2}}^{k_{(F1)}} BA_{b,s_{1}}^{ k_{(F2)}}.
	\end{align*}
	In $(*)$ we have used the identity
	\begin{align*}
		\smallint  \vartriangleleft S(h) = \langle {\smallint}_{(1)} , S(h) \rangle {\smallint}_{(2)}
		= \langle {\smallint}_{(2)} , S(h) \rangle {\smallint}_{(1)}
		= S(h) \vartriangleright \smallint \qquad \text{for $h\in D(H)$ }
	\end{align*}
	which follows from the cyclicity of the Haar integral $\smallint$.
	The proof for the case, where $\rho$ starts to the right of $s_{1}$, is analogous, but uses 	identity \eqref{eq:HolonomyCommutatorsStartRightDefect} instead of \eqref{eq:HolonomyCommutatorsStartLeftDefect}. The same holds for the case, where $\rho$ ends to the right of $s_2$, but with  \eqref{eq:HolonomyCommutatorsEndLeftDefect} instead of \eqref{eq:HolonomyCommutatorsEndRightDefect}. 
\end{proof}

We now show that the fusion of multiple excitations satisfies an associativity condition. 
This associativity reflects the associativity of the tensor product of $D(H_{b})$-modules. For three $D(H_{b})$-modules $M_{1},M_{2},M_{3}$, the two tensor products
$(M_{1} \oo M_{2}) \oo M_{3}$ and  $M_{1}  \oo \left( M_{2} \oo M_{3}\right)$
coincide as $D(H_{b})$-modules. 
In the Kitaev model, we have two ways of fusing three excitations of type $M_{1},M_{2},M_{3}$ in the bulk region $b$ into an excitation of type $M_{1} \oo M_{2} \oo M_{3}$, as shown in Figure~\ref{figure:BulkAssociativity}.
\begin{figure}[H]
	\centering
	\scalebox{0.25}{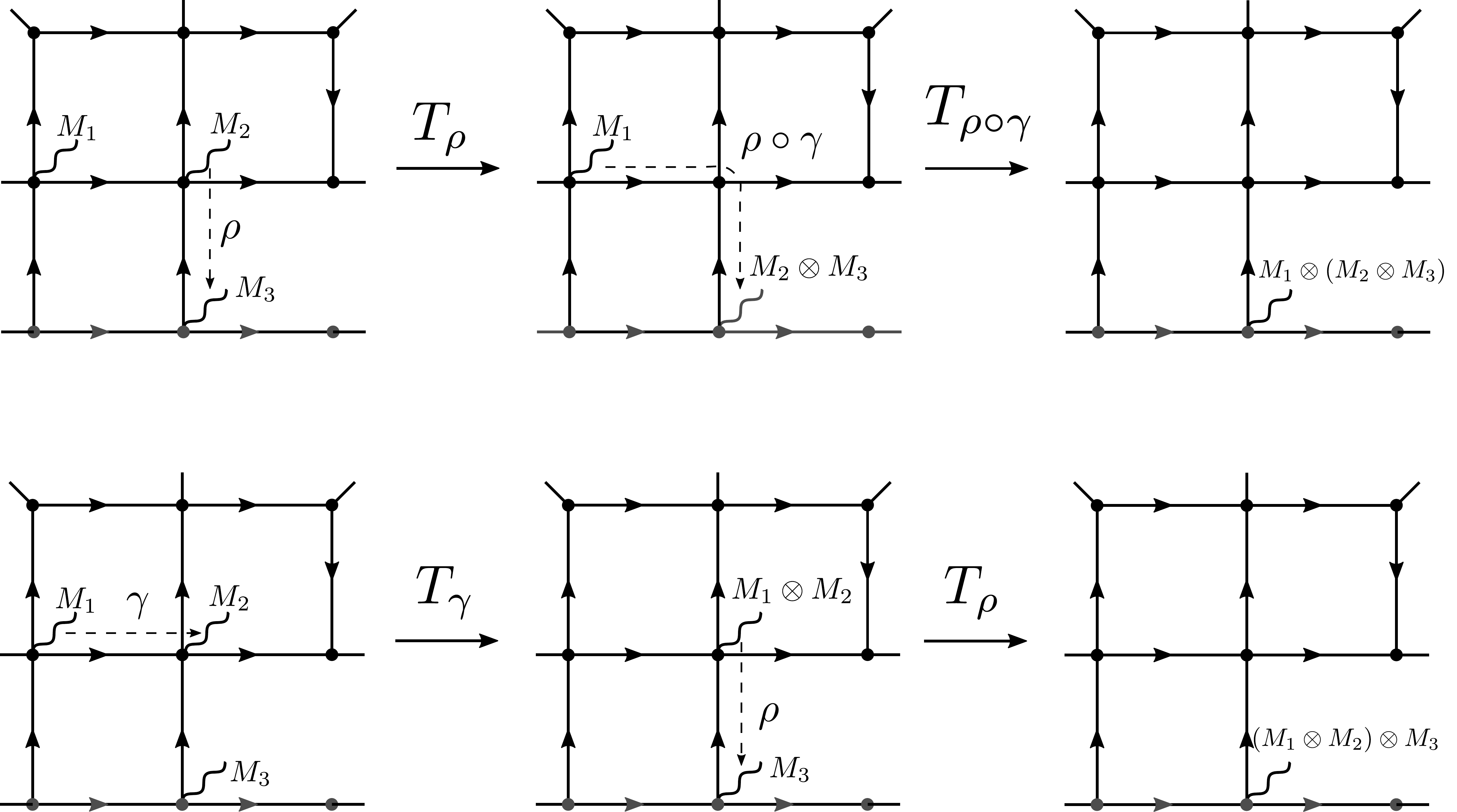}
	\caption{The two ways of fusing three bulk excitations $M_1,M_2,M_3$.}
	\label{figure:BulkAssociativity}
\end{figure}
As  the tensor product of $D(H_b)$-modules is associative, these two procedures must coincide, i.e. define the same map $\HSpace\to\HSpace$. 
If all three excitations are in a defect line or a boundary line, the same condition should hold after replacing the tensor product $ \oo $ with a twisted tensor product $ \oo_{F}$.
\\
If the excitations $M_{1}$ and $M_{2}$ are in the bulk and $M_{3}$ is in an adjacent boundary or defect line, then we again have two ways of fusing the excitations illustrated by Figure~\ref{figure:BoundaryAssociativity}.
\begin{figure}[H]
	\centering
	\scalebox{0.25}{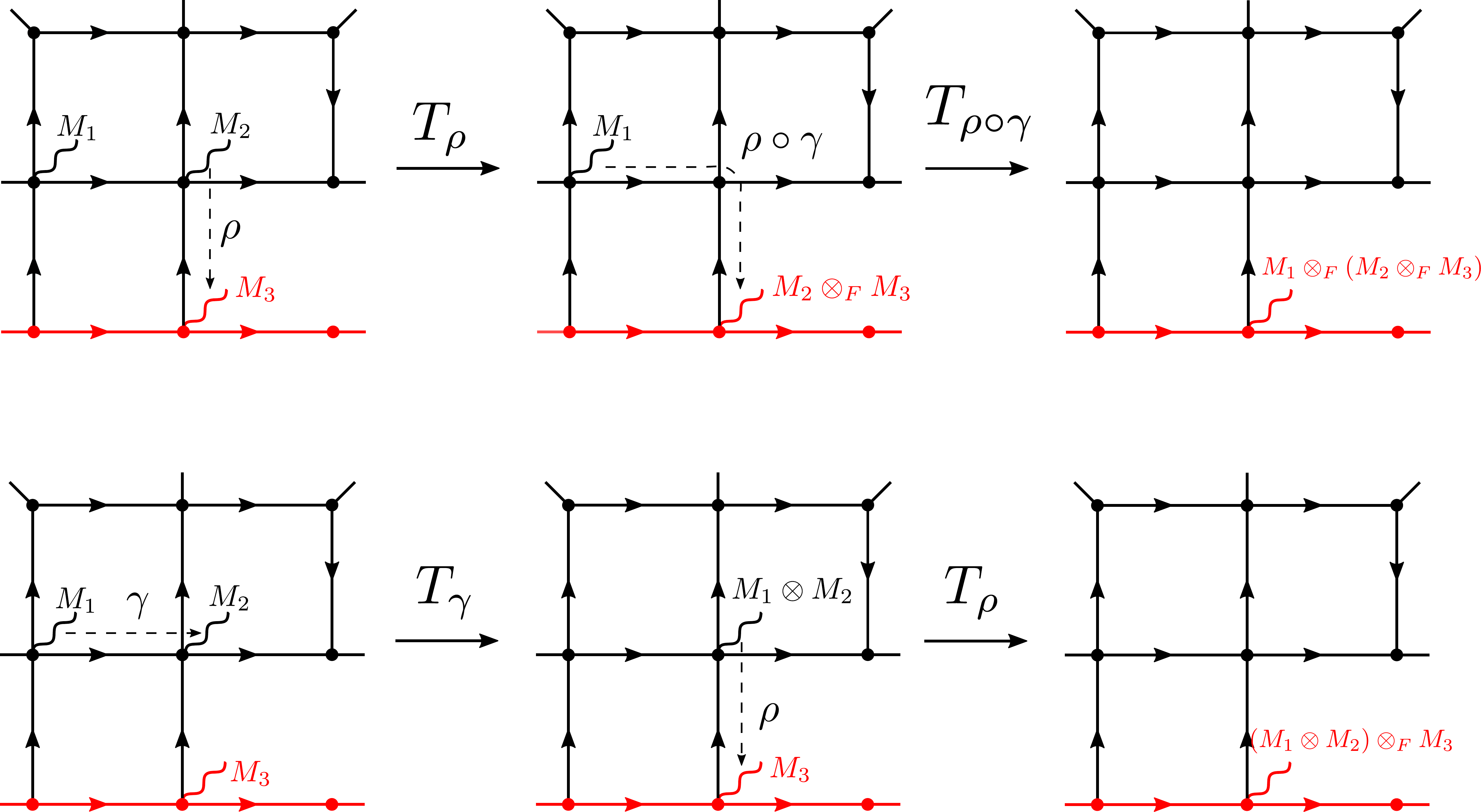}
	\caption{The two ways of fusing bulk excitations $M_1,M_2$ and a boundary excitation $M_3$.}
	\label{figure:BoundaryAssociativity}
\end{figure}
The two $D(H_{b})$-modules $(M_{1} \oo M_{2}) \oo_{F} M_{3}$ and $M_{1} \oo_{F} (M_{2}\oo_{F} M_{3})$  are isomorphic with an isomorphism given by the action of the twist $F$ from~\eqref{eq:TwistCoherenceData} on $M_{1}$ and $M_{2}$. The two procedures in Figure~\ref{figure:BoundaryAssociativity} should therefore coincide up to a twist acting on $M_{1}$ and $M_{2}$.
The following proposition shows that the transport operator $T_{\rho}$ indeed fulfills this associativity condition:
\begin{proposition}[Associativity]
	\label{proposition:TransportAssociativity}
	Let $\gamma:s_{0}\to s_{1}, \rho:s_{1}\to s_{2}$ be permissible paths satisfying the conditions of Lemma~\ref{lemma:HolonomyExtension}. Denote by $F$ the twist at $s_{2}$ and by $G$ the twist at $s_{1}$ and let either $F=G$ or $G$ trivial.
	Then we have
	\begin{align}
		T_{\rho\circ\gamma}\circ T_{\rho} &=  T_{\rho}\circ T_{\gamma}\quad&\text{if $F=G$.}
		\label{eq:TransportAssociativitySameRegion}
		\\
		T_{\rho\circ\gamma}\circ T_{\rho} &= T_{\rho}\circ  T_{\gamma}\circ  BA_{b,s_{0}}^{F^{(-2)}} \circ BA_{b,s_{1}}^{F^{(-1)}} \quad&\text{if $G$ trivial.}
		\label{eq:TransportAssociativityBulkToDefect}
	\end{align}
\end{proposition}
\begin{proof}
	For $F=G$ we show the claim by direct computation
	\begin{align*}
		T_{\rho}\circ T_{\gamma}  
		&\overset{\eqref{eq:TransportOperatorBulkRight}}{=}
		T_{\rho}\circ BA_{b,s_{0}}^{\lambda} \circ \Hol_{b,\gamma}^{\smallint  \vartriangleleft S( F^{(-2)})} \circ BA_{b,s_{1}}^{F^{(-1)}}
		\overset{\ref{proposition:OperatorsCommute}}{=}
		BA_{b,s_{0}}^{\lambda} \circ T_{\rho} \circ \Hol_{b,\gamma}^{\smallint  \vartriangleleft S(F^{(-2)})}\circ BA_{b,s_{1}}^{F^{(-1)}}
		\\
		&\overset{\eqref{eq:HolonomyStretching}}{=}
		BA_{b,s_{0}}^{\lambda} \circ \Hol_{b,\rho\circ\gamma}^{\smallint \vartriangleleft S(F^{(-2)})}\circ BA_{b,s_{2}}^{F^{(-1)}} \circ T_{\rho}
		\overset{\eqref{eq:TransportOperatorBulkRight}}{=}
		 T_{\rho\circ\gamma} \circ T_{\rho}
	\end{align*}
	For $G$ trivial we first assume that $\gamma$ starts to the left of $s_{0}$ and compute 
	\begin{align} 
		T_{\gamma}\circ  BA_{b,s_{0}}^{F^{(-2)}} \circ BA_{b,s_{1}}^{F^{(-1)}} 
		&\overset{\eqref{eq:TransportOperatorBulkRight}}{=}
		BA_{b,s_{0}}^{\lambda} \circ \Hol_{b,\gamma}^{ \smallint} \circ BA_{b,s_{0}}^{F^{(-2)}} \circ BA_{b,s_{1}}^{F^{(-1)}}
		\nonumber\\
		&\overset{\eqref{eq:HolonomyCommutatorsStartLeftDefect}}{=}
		BA_{b,s_{0}}^{\lambda F^{(-2)}_{(2)}}\circ \Hol_{b,\gamma}^{ S(F^{(-2)}_{(1)})  \vartriangleright \smallint}  \circ BA_{b,s_{1}}^{F^{(-1)}}
		\nonumber\\
		&=
		BA_{b,s_{0}}^{\lambda} \circ \Hol_{b,\gamma}^{ S(F^{(-2)})  \vartriangleright \smallint}  \circ BA_{b,s_{1}}^{F^{(-1)}}
		\nonumber\\
		&\overset{(*)}{=}
		BA_{b,s_{0}}^{\lambda} 
		\circ \Hol_{b,\gamma}^{ \smallint  \vartriangleleft S(F^{(-2)})}
		\circ BA_{b,s_{1}}^{F^{(-1)}}
		\label{eq:H1}
	\end{align}
	where we used the identity
	$ h  \vartriangleright \smallint = \smallint  \vartriangleleft h $ for $h\in D(H_{b})$ in $(*)$. If $\gamma$ starts to the right of $s_{0}$, the identity~\eqref{eq:H1} follows analogously, but we have to use~\eqref{eq:HolonomyCommutatorsStartRightDefect} instead of~\eqref{eq:HolonomyCommutatorsStartRightDefect}.
	Inserting~\eqref{eq:H1} into the right hand side of~\eqref{eq:TransportAssociativityBulkToDefect}, we obtain the second term of the proof of the case $F=G$. 
	Proceeding identically, we then obtain~\eqref{eq:TransportAssociativityBulkToDefect}.
\end{proof}

When fusing two excitations $M_{1},M_{2}$ with $T_{\rho}$ into an excitation $M_{1} \oo_{F} M_{2}$, the order of $M_1$ and $M_2$ in the tensor product depends on whether $\rho:s_{1}\to s_{2}$ ends to the left or the right of $s_{2}$ (cf. Proposition~\ref{proposition:TransportOperatorTensorProduct}). 
The two $D(H_{b})$-modules $M_{1} \oo M_{2}$ and $M_{2} \oo M_{1}$ are related by the braiding of $D(H_{b})\mathrm{-Mod}$. 
The following proposition shows that the transport operator satisfies condition~(T\ref{T2}) from Section \ref{section:TranslationKitaev}. 
In other words, for every permissible path $\rho: s_1\to s_2^R$ in a bulk region, there is a canonical path $\rho': s_1\to s_2^L$, obtained by composing $\rho$ with a face path, such 
that the maps $T_{\rho}$ and $T_{\rho'}$ are also related by that braiding. 
If $s_{2}$ is a defect or boundary site, then they instead are related by a twisted braiding. 

\begin{proposition}[Braiding]
	Let $\rho:s_{1}^{L}\to s_{2}^{R}$ a permissible path in $b$, $f$ the face path of $s_{2}$, $F$ the twist at $s_{2}$ and $R_{F}\in D(H_{b}) \oo D(H_{b})$ the twisted $R$-matrix of $D(H_{b})$ from \eqref{eq:TwistedRMatrix}. Then we have
	\begin{align}
		T_{f^{-1}\circ\rho} = T_{\rho}\circ BA_{b,s_{1}}^{R_{F}^{(2)}}\circ BA_{b,s_{2}}^{R_{F}^{(1)}}
		\label{eq:TransportOperatorBraidingBulk}
	\end{align}
\end{proposition}

\begin{proof}
	We denote the $R$-matrix of $D(H_{b})$ by $R=\varepsilon \oo x \oo X \oo 1$. We then show the claim by direct computation:
	\begin{align*}
		T_{\rho}\circ BA_{b,s_{1}}^{R_{F}^{(2)}}\circ BA_{b,s_{2}}^{R_{F}^{(1)}}
		&\overset{\eqref{eq:TransportOperatorBulkRight}}{=}
		BA_{b,s_{1}}^{\lambda} \Hol_{b,\rho}^{ \smallint  \vartriangleleft S(F^{(-2)})} BA_{b,s_{2}}^{F^{(-1)}}BA_{b,s_{1}}^{R_{F}^{(2)}} BA_{b,s_{2}}^{R_{F}^{(1)}}
		\\
		&\overset{\eqref{eq:HolonomyCommutatorsStartLeftDefect}}=
		BA_{b,s_{1}}^{\lambda R^{(2)}_{F,(2)}} \Hol_{b,\rho}^{ S(R_{F,(1)}^{(2)})  \vartriangleright \smallint  \vartriangleleft S(F^{(-2)})} BA_{b,s_{2}}^{F^{(-1)}R_{F}^{(1)}}
		\\
		&\overset{\eqref{eq:HaarIntegral}}=
		BA_{b,s_{1}}^{\lambda} \Hol_{b,\rho}^{ S(R_F^{(2)})  \vartriangleright \smallint  \vartriangleleft S(F^{(-2)})} BA_{b,s_{2}}^{F^{(-1)}R_{F}^{(1)}}
		\\
		&\overset{(*)}=
		BA_{b,s_{1}}^{\lambda} \Hol_{b,\rho}^{ \smallint  \vartriangleleft S(F^{(-2)} R_{F}^{(2)})} BA_{b,s_{2}}^{F^{(-1)}R_{F}^{(1)}}
		\\
		&\overset{\eqref{eq:TwistedRMatrix}}{=}
		BA_{b,s_{1}}^{\lambda} \Hol_{b,\rho}^{ \smallint  \vartriangleleft S(R^{(2)} F^{(-1)})} BA_{b,s_{2}}^{R^{(1)} F^{(-2)}}
		\\
		&=
		BA_{b,s_{1}}^{\lambda} \Hol_{b,\rho}^{ \smallint  \vartriangleleft S(R^{(2)} F^{(-1)})} BA_{b,s_{2}}^{R^{(1)}} BA_{b,s_{2}}^{F^{(-2)}}
		\\
		&\overset{\eqref{eq:Rmatrix}}=
		BA_{b,s_{ 1}}^{\lambda} \Hol_{b,\rho}^{ \smallint  \vartriangleleft S((\varepsilon  \oo x) F^{(-1)})} BA_{b,s_{2}}^{X \oo 1}  BA_{b,s_{2}}^{F^{(-2)}}
		\\
		&\overset{\eqref{eq:FaceOperatorBulk}}{=}
		BA_{b,s_{1}}^{\lambda} \Hol_{b,\rho}^{ \smallint  \vartriangleleft S( (\varepsilon  \oo x) F^{(-1)})} \Hol_{b,f}^{X \oo 1}  BA_{b,s_{2}}^{F^{(-2)}}
		\\
		&\overset{\eqref{eq:CoregularRightAction}}=
		\langle {\smallint}_{(1)} , S(F^{(-1)}) S(\varepsilon \oo x)  \rangle 
		BA_{b,s_{1}}^{\lambda} \Hol_{b,\rho}^{ \smallint_{(2)} } \Hol_{b,f}^{1 \oo X}  BA_{b,s_{2}}^{F^{(-2)}}
		\\
		&\overset{(**)}=
		\langle {\smallint}_{(1)} , S(F^{(-1)}) \rangle 
		BA_{b,s_{1}}^{\lambda} \Hol_{b,\rho}^{ \smallint_{(3)} } \Hol_{b,f}^{\smallint_{(2)}}  BA_{b,s_{2}}^{F^{(-2)}}
		\\
		&\overset{\eqref{eq:HolonomyRecursionLeftJoinDefect}}=
		BA_{b,s_{1}}^{\lambda} \Hol_{b,f^{-1}\circ\rho}^{ \smallint  \vartriangleleft S(F^{(-1)})}  BA_{b,s_{2}}^{F^{(-2)}}
		\\
		&=T_{f^{-1}\circ\rho},
	\end{align*}
	where we used 
	the identity 
	$ h  \vartriangleright \smallint = \smallint  \vartriangleleft h $ for $h\in D(H_{b})$ in $(*)$ and  
	in $(**)$ we used the identity
	\begin{align*}
		\Hol_{b,f}^{\beta} = \langle h,\varepsilon \rangle \Hol_{b,f}^{ 1 \oo \alpha} = 
		\langle h \oo \beta , \varepsilon \oo x \rangle \Hol_{b,f}^{1 \oo X} 
	\end{align*}
	 for the face path $f$ and $\beta:=h \oo \alpha\in D(H_{b})^{*}$ that follows because $x$ and $X$ stand for dual bases.
\end{proof}

\subsection{Transparent defects}
\label{subsec:transparent}

In this section we show that the Kitaev model without defects arises as a special case of the model with defects, if one considers defects  labeled with trivial defect data, called \emph{transparent defects} in the following. 
These are defects between two bulk regions labeled with the same Hopf algebra $H$, that are decorated with trivial defect data, namely the twist from Example \ref{example:Drinfel'dQuadrupleTwist} that relates the Drinfel'd double of a factorizable Hopf algebra $K$ to the tensor product Hopf algebra $K\oo K$.

\begin{definition}
	\label{definition:TransparentDefect}
	A \emph{transparent defect} is a defect $d$ 
	between bulk regions $b_{dL},b_{dR}$ such that
		\begin{compactenum}
			\item the bulk regions $b_{dL},b_{dR}$ are labeled with the same Hopf algebra $H=H_{b_{dL}}=H_{b_{dR}}$,
			\item the defect $d$ is labeled with the twist of $D(H) \oo D(H)$ from Example~\ref{example:Drinfel'dQuadrupleTwist} for $K=D(H)$.
		\end{compactenum}
	\end{definition}

We now show that the extended space of the model with a transparent defect and of the model without the defects are related by a module homomorphism for the action defined by any site of the graph except the defect sites.  For this, we consider the Kitaev models with defects and boundaries  related by removing a defect line $d$ labeled with transparent defect data. 

Removing such a transparent defect line  involves two steps:  (i) modifying the underlying ribbon graph $\Gamma$ to obtain a graph $\Gamma'$, (ii) applying a linear map $\mathcal R:\HSpace\to \HSpace'$ from the extended space $\HSpace $ for $\Gamma$ to the extended space $\HSpace'$ of $\Gamma'$. The modified graph $\Gamma'$ is  the ribbon graph with defects and boundaries obtained by  removing the defect line $d$ and all the edges of  the associated cyclic subgraph $\Gamma_{d}$,
but not its vertices, and by identifying the bulk regions $b_{dL}$ and $b_{dR}$.
	
We consider the linear map $\mathcal R: \HSpace\to \HSpace'$ that acts on the tensor factors of  $H \oo H$ for edges of $d$ by
	\begin{align}\label{eq:rdef}
 m\oo n \mapsto   \smallint\left( mS(n) \right)
	\end{align}
	and  as the identity map on the tensor factors associated to other edges, see Figure~\ref{fig:removingdefect}.
	
	\begin{figure}[H]
\begin{center}
\begin{tikzpicture}[scale=.7]
\begin{scope}[shift={(0,0)}]
\draw[line width=1.5pt, color=red, ->,>=stealth] (-2,0)--(-1,0) node[anchor=south]{$a\oo b$};
\draw[line width=1.5pt, color=red,->,>=stealth] (-1,0)--(2,0) node[anchor=south]{$c\oo d$};
\draw[line width=1.5pt, color=red, ->,>=stealth] (2,0)--(5,0) node[anchor=south]{$p\oo q$};
\draw[line width=1.5pt, color=red,] (5,0)--(6,0); 
\draw[line width=1pt, color=black] (0,-3)--(0,3);
\draw[line width=1pt, color=black] (4,-3)--(4,3);
\draw[line width=1pt, color=black] (-1,2)--(5,2);
\draw[line width=1pt, color=black] (-1,-2)--(5,-2);
\draw[color=red, fill=red] (0,0) circle (.2);
\draw[color=red, fill=red] (4,0) circle (.2);
\draw[color=black, fill=black] (0,2) circle (.2);
\draw[color=black, fill=black] (0,-2) circle (.2);
\draw[color=black, fill=black] (4,2) circle (.2);
\draw[color=black, fill=black] (4,-2) circle (.2);
\end{scope}
\draw[line width=1pt, color=black,->,>=stealth] (7,0)--(9,0);
\node at (8,0)[anchor=south]{$R$};
\begin{scope}[shift={(13,0)}]
\node at (-1,0) [anchor=east, color=red]{$\int(aS(b))$};
\node at (2,0) [color=red]{$\int (cS(d))$};
\node at (5,0) [anchor=west, color=red]{$\int (pS(q))$};
\draw[line width=1pt, color=black] (0,-3)--(0,3);
\draw[line width=1pt, color=black] (4,-3)--(4,3);
\draw[line width=1pt, color=black] (-1,2)--(5,2);
\draw[line width=1pt, color=black] (-1,-2)--(5,-2);
\draw[color=black, fill=black] (0,0) circle (.2);
\draw[color=black, fill=black] (4,0) circle (.2);
\draw[color=black, fill=black] (0,2) circle (.2);
\draw[color=black, fill=black] (0,-2) circle (.2);
\draw[color=black, fill=black] (4,2) circle (.2);
\draw[color=black, fill=black] (4,-2) circle (.2);
\end{scope}
\end{tikzpicture}
\end{center}
\caption{Removing a defect}
\label{fig:removingdefect}
\end{figure}
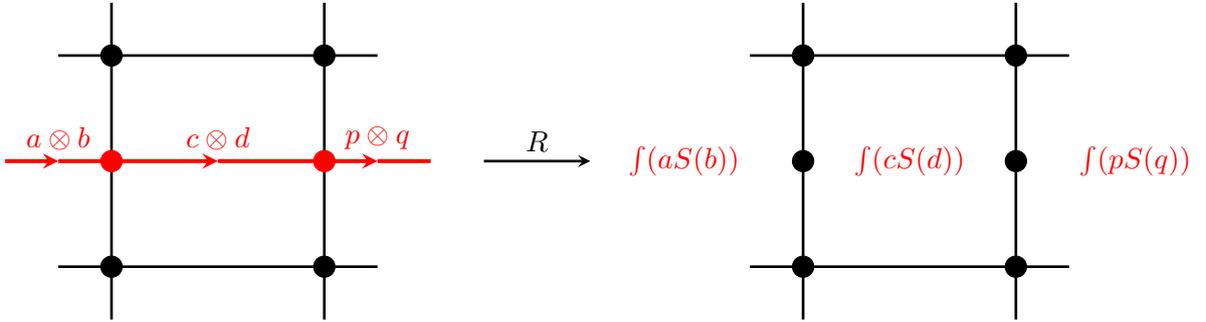
	It is clear from this definition that $\mathcal{R}$ commutes with the holonomies along all paths $\rho$  in $D(\Gamma)$ that do not traverse edges of $\Gamma_{d}$. In particular, $\mathcal{R}$ is a module homomorphism with respect to the $D(H_b)$-module structures and  $D(H_{b_{dL}})\oo D(H_{b_{dR}})$-module structures  from Proposition \ref{proposition:OperatorActionDefect} that are associated with sites that are disjoint to all sites in $d$.

	

	We now consider a pair of defect sites $(s_{L},s_{R})$ in $d$. Removing the defect turns $(s_{L},s_{R})$ into a single site $s \in \Gamma'$, as shown in Figures~\ref{fig:vertextransparentdefect} and Figure~\ref{fig:facevertexbulk}.
	The pair $(s_{L},s_{R})$  is associated with 
	 the $D(H)$-action on $\HSpace$ from Proposition \ref{proposition:OperatorActionDefect}, 2.~defined by
		\begin{align}\label{eq:sitepair}
			BA_{d,s_{L},s_{R}}^{t_{(1)} \oo t_{(2)}}:\HSpace\to\HSpace \quad\text{ for $t\in D(H)$},
		\end{align}
		see Figure~\ref{fig:vertextransparentdefect} and Figure~\ref{fig:facetransparentdefect} for a concrete example.
		The bulk site $s\in \Gamma'$ is associated with $D(H)$-action on $\HSpace'$ from from Proposition \ref{proposition:OperatorActionDefect}, 1.~defined by
		\begin{align}\label{eq:singlesite}
			BA_{s,b}^{t}:\HSpace'\to\HSpace' \quad\text{ for $t\in D(H)$},
		\end{align}
		for a concrete example, see Figure~\ref{fig:facevertexbulk}. It turns out that the map $\mathcal R:\HSpace\to\HSpace'$ from \eqref{eq:rdef} is a module homomorphism, when $\HSpace$ and $\HSpace'$ are equipped with these  two actions.
		
\begin{figure}[H]
\begin{center}
\begin{tikzpicture}[scale=.7]
\begin{scope}[shift={(0,0)},scale=1.3]
\draw[line width=1.5pt, color=red, ->,>=stealth] (-2,0)--(-1,0) node[anchor=south]{$a\oo b\quad$};
\draw[line width=1.5pt, color=red,->,>=stealth] (-1,0)--(2,0) node[anchor=south]{$\quad c\oo d$};
\draw[line width=1.5pt, color=red, ->,>=stealth] (2,0)--(5,0);
\draw[line width=1pt, color=black,->,>=stealth] (0,-3)--(0,-1);
\draw[line width=1pt, color=black] (0,-1)node[anchor=east]{$f$}--(0,1)node[anchor=east]{$e$};
\draw[line width=1pt, color=black,->,>=stealth] (0,3)--(0,1);
\draw[line width=1pt, color=black] (4,-3)--(4,3);
\draw[line width=1pt, color=black] (-1,2)--(5,2);
\draw[line width=1pt, color=black] (-1,-2)--(5,-2);
\draw[line width=1pt, color=black] (0,0)--(.5,.5) node[anchor=west]{$s_L$};
\draw[line width=1pt, color=black] (0,0)--(.5,-.5) node[anchor=west]{$s_R$};
\draw[color=red, fill=red] (0,0) circle (.15);
\draw[color=red, fill=red] (4,0) circle (.15);
\draw[color=black, fill=black] (0,2) circle (.15);
\draw[color=black, fill=black] (0,-2) circle (.15);
\draw[color=black, fill=black] (4,2) circle (.15);
\draw[color=black, fill=black] (4,-2) circle (.15);
\end{scope}
\draw[line width=1pt, color=black,->,>=stealth] (7,0)--(9,0);
\node at (8,0)[anchor=south]{$BA^{\epsilon\oo h}_{d,s_L,s_R}$};
\begin{scope}[shift={(13,0)}, scale=1.3]
\draw[line width=1.5pt, color=red, ->,>=stealth] (-2,0)--(-1,0) node[anchor=south]{$h_{(2)}a\oo h_{(5)}b\qquad$};
\draw[line width=1.5pt, color=red,->,>=stealth] (-1,0)--(2,0) node[anchor=south]{$\; c S(h_{(3)})\oo d S(h_{(4)})$};
\draw[line width=1.5pt, color=red, ->,>=stealth] (2,0)--(5,0);
\draw[line width=1pt, color=black,->,>=stealth] (0,-3)--(0,-1);
\draw[line width=1pt, color=black] (0,-1)node[anchor=east]{$h_{(6)}f$}--(0,1)node[anchor=east]{$h_{(1)}e$};
\draw[line width=1pt, color=black,->,>=stealth] (0,3)--(0,1);
\draw[line width=1pt, color=black] (4,-3)--(4,3);
\draw[line width=1pt, color=black] (-1,2)--(5,2);
\draw[line width=1pt, color=black] (-1,-2)--(5,-2);
%
%
\draw[color=red, fill=red] (0,0) circle (.15);
\draw[color=red, fill=red] (4,0) circle (.15);
\draw[color=black, fill=black] (0,2) circle (.15);
\draw[color=black, fill=black] (0,-2) circle (.15);
\draw[color=black, fill=black] (4,2) circle (.15);
\draw[color=black, fill=black] (4,-2) circle (.15);
\end{scope}
\end{tikzpicture}
\end{center}
\caption{$H$-action defined by a vertex and face operator at a transparent defect}
\label{fig:vertextransparentdefect}
\end{figure}

\begin{figure}[H]
\begin{center}
\begin{tikzpicture}[scale=.7]
\begin{scope}[shift={(0,0)},scale=1.3]
\draw[line width=1.5pt, color=red, ->,>=stealth] (-1,0)--(-.5,0);
\draw[line width=1.5pt, color=red,->,>=stealth] (-.5,0)--(2,0) node[anchor=south]{$\quad c\oo d$};
\draw[line width=1.5pt, color=red, ->,>=stealth] (2,0)--(5,0);
\draw[line width=1pt, color=black,->,>=stealth] (0,-3)--(0,-1);
\draw[line width=1pt, color=black] (0,-1)node[anchor=east]{$f$}--(0,1)node[anchor=east]{$e$};
\draw[line width=1pt, color=black,->,>=stealth] (0,3)--(0,1);

\draw[line width=1pt, color=black,->,>=stealth] (4,3)--(4,1) node[anchor=west]{$k$};
\draw[line width=1pt, color=black,->,>=stealth] (4,1)--(4,-1)node[anchor=west]{$m$};
\draw[line width=1pt, color=black,] (4,-1)--(4,-3);
\draw[line width=1pt, color=black, ->,>=stealth] (-1,2)--(2,2) node[anchor=north]{$g$};
\draw[line width=1pt, color=black] (2,2)--(5,2);
\draw[line width=1pt, color=black, ->,>=stealth] (5,-2)--(2,-2) node[anchor=south]{$n$};
\draw[line width=1pt, color=black] (-1,-2)--(2,-2);
\draw[line width=1pt, color=black] (0,0)--(.5,.5) node[anchor=west]{$s_L$};
\draw[line width=1pt, color=black] (0,0)--(.5,-.5) node[anchor=west]{$s_R$};
\draw[color=red, fill=red] (0,0) circle (.15);
\draw[color=red, fill=red] (4,0) circle (.15);
\draw[color=black, fill=black] (0,2) circle (.15);
\draw[color=black, fill=black] (0,-2) circle (.15);
\draw[color=black, fill=black] (4,2) circle (.15);
\draw[color=black, fill=black] (4,-2) circle (.15);
\end{scope}
\draw[line width=1pt, color=black,->,>=stealth] (7,0)--(9,0);
\node at (8,0)[anchor=south]{$BA^{\alpha\oo 1}_{d,s_L,s_R}$};
\begin{scope}[shift={(13,0)}, scale=1.3]
\draw[line width=1.5pt, color=red, ->,>=stealth] (-1,0)--(-.5,0);
\draw[line width=1.5pt, color=red,->,>=stealth] (-.5,0)--(2,0) node[anchor=south]{$\quad c_{(2)}\oo d_{(1)}$};
\draw[line width=1.5pt, color=red, ->,>=stealth] (2,0)--(5,0);
\draw[line width=1pt, color=black,->,>=stealth] (0,-3)--(0,-1);
\draw[line width=1pt, color=black] (0,-1)node[anchor=east]{$f_{(1)}$}--(0,1)node[anchor=east]{$e_{(2)}$};
\draw[line width=1pt, color=black,->,>=stealth] (0,3)--(0,1);

\draw[line width=1pt, color=black,->,>=stealth] (4,3)--(4,1) node[anchor=west]{$k_{(1)}$};
\draw[line width=1pt, color=black,->,>=stealth] (4,1)--(4,-1)node[anchor=west]{$m_{(1)}$};
\draw[line width=1pt, color=black,] (4,-1)--(4,-3);
\draw[line width=1pt, color=black, ->,>=stealth] (-1,2)--(2,2) node[anchor=north]{$g_{(1)}$};
\draw[line width=1pt, color=black] (2,2)--(5,2);
\draw[line width=1pt, color=black, ->,>=stealth] (5,-2)--(2,-2) node[anchor=south]{$n_{(1)}$};
\draw[line width=1pt, color=black] (-1,-2)--(2,-2);
%
%
\draw[color=red, fill=red] (0,0) circle (.15);
\draw[color=red, fill=red] (4,0) circle (.15);
\draw[color=black, fill=black] (0,2) circle (.15);
\draw[color=black, fill=black] (0,-2) circle (.15);
\draw[color=black, fill=black] (4,2) circle (.15);
\draw[color=black, fill=black] (4,-2) circle (.15);
\node at (2,-3.5)[color=black]{$\langle \alpha, f_{(2)}n_{(2)}m_{(2)}d_{(2)}S(c_{(1)})k_{(2)}g_{(2)}S(e_{(1)})$};
\end{scope}
\end{tikzpicture}
\end{center}
\caption{$H^*$-action defined by a vertex and face operator at a transparent defect.}
\label{fig:facetransparentdefect}
\end{figure}

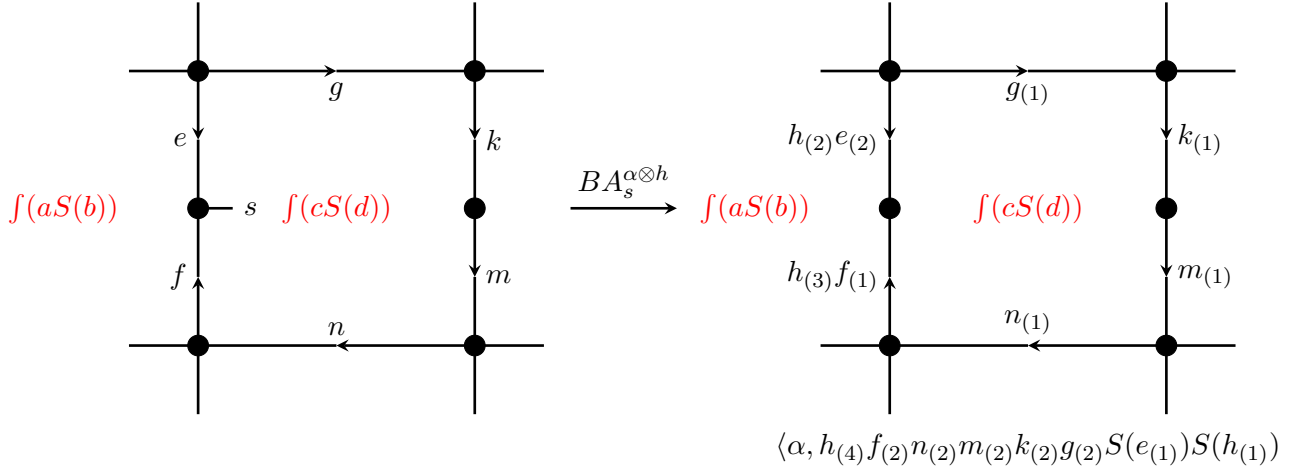
\begin{figure}[H]
\begin{center}
\begin{tikzpicture}[scale=.7]
\begin{scope}[shift={(0,0)},scale=1.3]
%
\draw[line width=1pt, color=black,->,>=stealth] (0,-3)--(0,-1);
\draw[line width=1pt, color=black] (0,-1)node[anchor=east]{$f$}--(0,1)node[anchor=east]{$e$};
\draw[line width=1pt, color=black,->,>=stealth] (0,3)--(0,1);

\draw[line width=1pt, color=black,->,>=stealth] (4,3)--(4,1) node[anchor=west]{$k$};
\draw[line width=1pt, color=black,->,>=stealth] (4,1)--(4,-1)node[anchor=west]{$m$};
\draw[line width=1pt, color=black,] (4,-1)--(4,-3);
\draw[line width=1pt, color=black, ->,>=stealth] (-1,2)--(2,2) node[anchor=north]{$g$};
\draw[line width=1pt, color=black] (2,2)--(5,2);
\draw[line width=1pt, color=black, ->,>=stealth] (5,-2)--(2,-2) node[anchor=south]{$n$};
\draw[line width=1pt, color=black] (-1,-2)--(2,-2);
\draw[line width=1pt, color=black] (0,0)--(.5,0) node[anchor=west]{$s$};
\draw[color=black, fill=black] (0,0) circle (.15);
\draw[color=black, fill=black] (4,0) circle (.15);
\draw[color=black, fill=black] (0,2) circle (.15);
\draw[color=black, fill=black] (0,-2) circle (.15);
\draw[color=black, fill=black] (4,2) circle (.15);
\draw[color=black, fill=black] (4,-2) circle (.15);
\node at (-1,0) [anchor=east, color=red]{$\int(aS(b))$};
\node at (2,0) [color=red]{$\int (cS(d))$};
\end{scope}
\draw[line width=1pt, color=black,->,>=stealth] (7,0)--(9,0);
\node at (8,0)[anchor=south]{$BA^{\alpha\oo h}_{s}$};
\begin{scope}[shift={(13,0)}, scale=1.3]
%
\draw[line width=1pt, color=black,->,>=stealth] (0,-3)--(0,-1);
\draw[line width=1pt, color=black] (0,-1)node[anchor=east]{$h_{(3)}f_{(1)}$}--(0,1)node[anchor=east]{$h_{(2)}e_{(2)}$};
\draw[line width=1pt, color=black,->,>=stealth] (0,3)--(0,1);

\draw[line width=1pt, color=black,->,>=stealth] (4,3)--(4,1) node[anchor=west]{$k_{(1)}$};
\draw[line width=1pt, color=black,->,>=stealth] (4,1)--(4,-1)node[anchor=west]{$m_{(1)}$};
\draw[line width=1pt, color=black,] (4,-1)--(4,-3);
\draw[line width=1pt, color=black, ->,>=stealth] (-1,2)--(2,2) node[anchor=north]{$g_{(1)}$};
\draw[line width=1pt, color=black] (2,2)--(5,2);
\draw[line width=1pt, color=black, ->,>=stealth] (5,-2)--(2,-2) node[anchor=south]{$n_{(1)}$};
\draw[line width=1pt, color=black] (-1,-2)--(2,-2);
%
%
\draw[color=black, fill=black] (0,0) circle (.15);
\draw[color=black, fill=black] (4,0) circle (.15);
\draw[color=black, fill=black] (0,2) circle (.15);
\draw[color=black, fill=black] (0,-2) circle (.15);
\draw[color=black, fill=black] (4,2) circle (.15);
\draw[color=black, fill=black] (4,-2) circle (.15);
\node at (2, -3.5)[color=black]{$\langle \alpha, h_{(4)}f_{(2)}n_{(2)}m_{(2)}k_{(2)}g_{(2)}S(e_{(1)})S(h_{(1)})$};
\node at (-1,0) [anchor=east, color=red]{$\int(aS(b))$};
\node at (2,0) [color=red]{$\int (cS(d))$};
\end{scope}
\end{tikzpicture}
\end{center}
\caption{$D(H)$-action defined by a vertex and face operator in the bulk after removing the defect line.}
\label{fig:facevertexbulk}
\end{figure}

\begin{proposition}
		\label{proposition:RemoveTransparentHomomorphism}
		The map $\mathcal{R}:\HSpace\to\HSpace'$ is a $D(H)$-module homomorphism with respect to  the $D(H)$-actions defined by \eqref{eq:sitepair} and \eqref{eq:singlesite}:
		\begin{align*}
		\mathcal R\circ BA^{t_{(1)}\oo t_{(2)}}_{d,s_L,s_R}=BA^t_s\circ \mathcal R.
		\end{align*}
		
\end{proposition}
\begin{proof}
		For $t=\alpha \oo h\in D(H)$ we have 	
		\begin{align}
			BA_{d,s_{L},s_{R}}^{t_{(1)} \oo t_{(2)}} 
			&=
			BA_{b_{dL},s_{L}}^{t_{(1)} } \circ BA_{b_{dR},s_{R}}^{t_{(2)}} 
			=
			B_{b_{dL},s_{L}}^{\alpha_{(2)}}\circ A_{b_{dL},s_{L}}^{h_{(1)}}\circ B_{b_{dR},s_{R}}^{\alpha_{(1)}}  \circ A_{b_{dR},s_{R}}^{h_{(2)}} \nonumber
			\\
			&\overset{\ref{proposition:OperatorsCommute}}{=}
			B_{b_{dL},s_{L}}^{\alpha_{(2)}}\circ B_{b_{dR},s_{R}}^{\alpha_{(1)}} \circ A_{b_{dL},s_{L}}^{h_{(1)}} \circ A_{b_{dR},s_{R}}^{h_{(2)}} 
			\label{eq:TransparentDefectAction}
		\end{align}
		We now show the claim for the example from Figure~\ref{fig:removingdefect}.  
		We start by computing $\mathcal R\circ BA_{d,s_{L},s_{R}}^{t_{(1)} \oo t_{(2)}} $. Using the formulas from  Figure~\ref{fig:vertextransparentdefect} and \ref{fig:facetransparentdefect} we obtain:
		\begin{align*}
			&BA_{d,s_{L},s_{R}}^{t_{(1)} \oo t_{(2)}} \left( a \oo b \oo c \oo d \oo e \oo f \oo g  \oo k \oo  m \oo n\right)
			\\
			&\;
			\overset{\eqref{eq:TransparentDefectAction}}{=}
			B_{b_{dL},s_{L}}^{\alpha_{(2)}}\circ B_{b_{dR},s_{R}}^{\alpha_{(1)}} \circ A_{b_{dL},s_{L}}^{h_{(1)}} \circ A_{b_{dR},s_{R}}^{h_{(2)}} 
			\left( a \oo b \oo c \oo d \oo e \oo f \oo g  \oo k \oo  m \oo n\right)
			\\
			&\overset{\ref{fig:vertextransparentdefect}}{=} B_{b_{dL},s_{L}}^{\alpha_{(2)}}\circ B_{b_{dR},s_{R}}^{\alpha_{(1)}} 
			( h_{(2)}a  \oo h_{(5)}b \oo c_{(2)}S(h_{(3)}) \oo d_{(1)}S(h_{(4)}) 
			\\
			&\qquad\qquad\qquad\qquad\qquad\qquad\qquad\qquad\oo h_{(1)}e_{(2)} \oo h_{(6)} f_{(1)} \oo g_{(1)} \oo k_{(1)} \oo m_{(1)} \oo n_{(1)} )
			\\&
			\overset{\ref{fig:facetransparentdefect}}=
			\langle \alpha ,h_{(10)} f_{(2)}n_{(2)}m_{(2)}d_{(2)}S(h_{(6)})h_{(5)} S(c_{(1)})k_{(2)}g_{(2)}S(e_{(1)}) S(h_{(1)}) \rangle 
			\\
			&\qquad h_{(3)}a  \oo h_{(8)}b \oo c_{(2)}S(h_{(4)}) \oo d_{(1)}S(h_{(7)}) \oo h_{(2)}e_{(2)} \oo h_{(9)} f_{(1)} \oo g_{(1)} \oo k_{(1)} \oo m_{(1)} \oo n_{(1)}
			\\
			&\;=\langle \alpha ,h_{(8)} f_{(2)}n_{(2)}m_{(2)}d_{(2)}S(c_{(1)})k_{(2)}g_{(2)}S(e_{(1)}) S(h_{(1)}) \rangle 
			\\
			&\qquad h_{(3)}a  \oo h_{(6)}b \oo c_{(2)}S(h_{(4)}) \oo d_{(1)}S(h_{(5)}) \oo h_{(2)}e_{(2)} \oo h_{(7)} f_{(1)} \oo g_{(1)} \oo k_{(1)} \oo m_{(1)} \oo n_{(1)}
		\end{align*}

	Applying the map $\mathcal{R}$ to this expression yields 
	\begin{align}
		&\mathcal{R}\circ BA_{d,s_{L},s_{R}}^{t_{(1)} \oo t_{(2)}} \left( a \oo b \oo c \oo d \oo e \oo f \oo g  \oo k \oo  m \oo n\right)\nonumber
		\\ \nonumber
			&\;=
			\langle \alpha ,h_{(8)} f_{(2)}n_{(2)}m_{(2)}d_{(2)}S(c_{(1)})k_{(2)}g_{(2)}S(e_{(1)}) S(h_{(1)}) \rangle  
			\cdot\langle \smallint ,  h_{(3)}a  S( h_{(6)}b)  \rangle 
			\\ \nonumber
			&\qquad
			\langle \smallint ,   c_{(2)}S(h_{(4)}) h_{(5)}S(d_{(1)}) \rangle 
			  h_{(2)}e_{(2)} \oo h_{(7)} f_{(1)} \oo g_{(1)} \oo k_{(1)} \oo m_{(1)} \oo n_{(1)}
			  \\ \nonumber
			&\stackrel{(*)}=
			\langle \alpha ,h_{(6)} f_{(2)}n_{(2)}m_{(2)}d_{(2)}S(c_{(1)})k_{(2)}g_{(2)}S(e_{(1)}) S(h_{(1)}) \rangle  
			\cdot\langle \smallint ,  h_{(3)}a S(b) S( h_{(4)})  \rangle 
			\\ \nonumber
			&\qquad\langle \smallint , c_{(2)}S(d_{(1)}) \rangle 
			  h_{(2)}e_{(2)} \oo h_{(5)} f_{(1)} \oo g_{(1)} \oo k_{(1)} \oo m_{(1)} \oo n_{(1)}
			  \\ \nonumber
			&\stackrel{(**)}=
			\langle \alpha ,h_{(4)} f_{(2)}n_{(2)}m_{(2)}d_{(2)}S(c_{(1)})k_{(2)}g_{(2)}S(e_{(1)}) S(h_{(1)}) \rangle  
			\cdot\langle \smallint ,  a S(b)  \rangle 
			\\ \nonumber
			&\qquad\langle \smallint , c_{(2)}S(d_{(1)}) \rangle 
			  h_{(2)}e_{(2)} \oo h_{(3)} f_{(1)} \oo g_{(1)} \oo k_{(1)} \oo m_{(1)} \oo n_{(1)}
			  \\ \nonumber
			&\stackrel{(***)}=
			\langle \alpha ,h_{(4)} f_{(2)}n_{(2)}m_{(2)}k_{(2)}g_{(2)}S(e_{(1)}) S(h_{(1)}) \rangle  
			\cdot\langle \smallint ,  a S(b)  \rangle 
			\\
			&\qquad\langle \smallint , cS(d) \rangle 
			  h_{(2)}e_{(2)} \oo h_{(3)} f_{(1)} \oo g_{(1)} \oo k_{(1)} \oo m_{(1)} \oo n_{(1)}.
			  \label{eq:RemoveDefectVertexFaceOp}
	\end{align}
Here, we used in $(*)$	 the defining property of the antipode $S$, in $(**)$ used the cyclicity property
		$\langle \smallint , pq \rangle = \langle \smallint , qp \rangle $
	of the Haar integral $\smallint$ and the properties of the antipode, and in $(***)$
	 the identity
	$
		\langle \smallint , p_{(2)} \rangle  p_{(1)} = \langle \smallint , p \rangle 1
	$
	for the Haar integral $\smallint$ applied to $p=cS(d)$.

The result of applying the operator $BA_s^t\circ \mathcal R$ instead is given in
 Figure~\ref{fig:facevertexbulk}. Comparing~\eqref{eq:RemoveDefectVertexFaceOp} to this result we find both terms to be identical. This shows that $\mathcal{R}$ is a $D(H)$-module homomorphism in this example.	
 The proof for a pair of defect sites in a general graph is analogous.
 In the general case, we may have a different number of edges at the vertex and the two faces of $s_{L}$ and $s_{R}$. Still, we have two defect edges at the vertex which are still labeled with elements $a \oo b, c \oo d\in H \oo H$. Applying $BA_{d,s_{L},s_{R}}^{t_{(1)} \oo t_{(2)}}$ we obtain the following term
 \begin{align*}
	 &BA_{d,s_{L},s_{R}}^{t_{(1)} \oo t_{(2)}} \left( a \oo b \oo c \oo d \oo \cdots \right)
	 \\&\quad=\langle \alpha , h_{(max)} \cdots d_{(2)}S(c_{(1)})\cdots S(h_{(1)}) \rangle 
	 \cdot h_{(n)}a \oo h_{(n+3)}b \oo c_{(2)}S(h_{(n+1)}) \oo d_{(1)}S(h_{(n+2)}) \oo \cdots
 \end{align*}
 where $n\in \mathbb{N}$ depends on the number of edges at the vertex of $(s_{L},s_{R})$ and $\cdots$ stands for terms coming from the other edges at the faces and vertices. 
 Applying $\mathcal{R}$ and proceeding analogously to the example, the terms coming from the defect edges cancel. 
 The result again coincides with the one obtained by computing $BA_s^t\circ \mathcal R$.
	\end{proof}
	
	In this article we have constructed a Kitaev model with topological boundaries and defects which satisfies the axioms~(D\ref{D1}) to~(T\ref{T2}) from Section~\ref{section:TranslationKitaev}, that are the counterparts of conditions postulated in~\cite{FSV} for a Turaev-Viro TQFT with topological boundaries and defects. 
	The vertex and face operators at a site define a representation of a Drinfeld double $D(H)$ on the extended space $\HSpace$. 
	This in turn allows us to define excitations and we can generate, move, fuse and braid these excitations by using (twisted) holonomies and the transport operator. 
	The rules governing these operations are the Hopf algebraic counterpart of the categorical rules for Turaev-Viro TQFTs from~\cite{FSV}.
	The last result shows how we regain the Kitaev model without defects from a model that only has transparent defects.
	
	Our construction suggests multiple questions for further research. The relation between the model with transparent defects and the model without defects suggests a definition for the protected space of our model. It is plausible that this would define a topological invariant and coincide with  the one from the Kitaev model without defects for a model with only transparent defects.  It would also be interesting to give an additional extension of our model with codimension-two defects between different defect lines and to investigate the passage of excitations through defect lines. 

	It would also be interesting to compare our models with Kitaev models with other, not necessarily topological, types of defects such as the ones in~\cite{BD,KK} or the defects based on bicomodule algebras in~\cite{K}. While some of these models admit more general defect data, the interaction of defects with excitations is not investigated there, and the works  \cite{KK,K} do not derive transport, fusion or braiding operators that allow one to investigate the behavior of excitations. 
	
	It would also be interesting to compare our construction with models inspired by TQFTs with defects that investigate mapping class group action such as 
~\cite{FSS}.


\end{document}

%% file: paths2.pdf_tex
\begingroup%
  \makeatletter%
  \providecommand\color[2][]{%
    \errmessage{(Inkscape) Color is used for the text in Inkscape, but the package 'color.sty' is not loaded}%
    \renewcommand\color[2][]{}%
  }%
  \providecommand\transparent[1]{%
    \errmessage{(Inkscape) Transparency is used (non-zero) for the text in Inkscape, but the package 'transparent.sty' is not loaded}%
    \renewcommand\transparent[1]{}%
  }%
  \providecommand\rotatebox[2]{#2}%
  \newcommand*\fsize{\dimexpr\f@size pt\relax}%
  \newcommand*\lineheight[1]{\fontsize{\fsize}{#1\fsize}\selectfont}%
  \ifx\svgwidth\undefined%
    \setlength{\unitlength}{841.88976378bp}%
    \ifx\svgscale\undefined%
      \relax%
    \else%
      \setlength{\unitlength}{\unitlength * \real{\svgscale}}%
    \fi%
  \else%
    \setlength{\unitlength}{\svgwidth}%
  \fi%
  \global\let\svgwidth\undefined%
  \global\let\svgscale\undefined%
  \makeatother%
  \begin{picture}(1,0.70707071)%
    \lineheight{1}%
    \setlength\tabcolsep{0pt}%
    \put(0,0){\includegraphics[width=\unitlength,page=1]{paths2.pdf}}%
    \put(0.83256582,0.2786977){\color[rgb]{0,0,0}\makebox(0,0)[lt]{\lineheight{1.25}\smash{\begin{tabular}[t]{l}$\rho$\end{tabular}}}}%
    \put(0.64141496,0.12699834){\color[rgb]{0,0,0}\makebox(0,0)[lt]{\lineheight{1.25}\smash{\begin{tabular}[t]{l}{\color{red}$\sigma$}\end{tabular}}}}%
    \put(0.47291637,0.35989262){\color[rgb]{0,0,0}\makebox(0,0)[lt]{\lineheight{1.25}\smash{\begin{tabular}[t]{l}{\color{blue}$\tau$}\end{tabular}}}}%
    \put(0.17308011,0.47061293){\color[rgb]{0,0,0}\makebox(0,0)[lt]{\lineheight{1.25}\smash{\begin{tabular}[t]{l}{\color{violet}$\omega$}\end{tabular}}}}%
    \put(0,0){\includegraphics[width=\unitlength,page=2]{paths2.pdf}}%
    \put(0.77395219,0.64507669){\color[rgb]{0,0,0}\makebox(0,0)[lt]{\lineheight{1.25}\smash{\begin{tabular}[t]{l}{\color{magenta}$\nu$}\end{tabular}}}}%
    \put(0,0){\includegraphics[width=\unitlength,page=3]{paths2.pdf}}%
  \end{picture}%
\endgroup%

%% file: HolonomyFusion.pdf_tex
\begingroup%
  \makeatletter%
  \providecommand\color[2][]{%
    \errmessage{(Inkscape) Color is used for the text in Inkscape, but the package 'color.sty' is not loaded}%
    \renewcommand\color[2][]{}%
  }%
  \providecommand\transparent[1]{%
    \errmessage{(Inkscape) Transparency is used (non-zero) for the text in Inkscape, but the package 'transparent.sty' is not loaded}%
    \renewcommand\transparent[1]{}%
  }%
  \providecommand\rotatebox[2]{#2}%
  \newcommand*\fsize{\dimexpr\f@size pt\relax}%
  \newcommand*\lineheight[1]{\fontsize{\fsize}{#1\fsize}\selectfont}%
  \ifx\svgwidth\undefined%
    \setlength{\unitlength}{599.18611782bp}%
    \ifx\svgscale\undefined%
      \relax%
    \else%
      \setlength{\unitlength}{\unitlength * \real{\svgscale}}%
    \fi%
  \else%
    \setlength{\unitlength}{\svgwidth}%
  \fi%
  \global\let\svgwidth\undefined%
  \global\let\svgscale\undefined%
  \makeatother%
  \begin{picture}(1,0.58710861)%
    \lineheight{1}%
    \setlength\tabcolsep{0pt}%
    \put(0,0){\includegraphics[width=\unitlength,page=1]{HolonomyFusion.pdf}}%
  \end{picture}%
\endgroup%

%% file: FuseExcitations.pdf_tex
\begingroup%
  \makeatletter%
  \providecommand\color[2][]{%
    \errmessage{(Inkscape) Color is used for the text in Inkscape, but the package 'color.sty' is not loaded}%
    \renewcommand\color[2][]{}%
  }%
  \providecommand\transparent[1]{%
    \errmessage{(Inkscape) Transparency is used (non-zero) for the text in Inkscape, but the package 'transparent.sty' is not loaded}%
    \renewcommand\transparent[1]{}%
  }%
  \providecommand\rotatebox[2]{#2}%
  \newcommand*\fsize{\dimexpr\f@size pt\relax}%
  \newcommand*\lineheight[1]{\fontsize{\fsize}{#1\fsize}\selectfont}%
  \ifx\svgwidth\undefined%
    \setlength{\unitlength}{911.59403643bp}%
    \ifx\svgscale\undefined%
      \relax%
    \else%
      \setlength{\unitlength}{\unitlength * \real{\svgscale}}%
    \fi%
  \else%
    \setlength{\unitlength}{\svgwidth}%
  \fi%
  \global\let\svgwidth\undefined%
  \global\let\svgscale\undefined%
  \makeatother%
  \begin{picture}(1,0.37673727)%
    \lineheight{1}%
    \setlength\tabcolsep{0pt}%
    \put(0,0){\includegraphics[width=\unitlength,page=1]{FuseExcitations.pdf}}%
  \end{picture}%
\endgroup%

%% file: GeneralTwoCiliaTransported.pdf_tex
\begingroup%
  \makeatletter%
  \providecommand\color[2][]{%
    \errmessage{(Inkscape) Color is used for the text in Inkscape, but the package 'color.sty' is not loaded}%
    \renewcommand\color[2][]{}%
  }%
  \providecommand\transparent[1]{%
    \errmessage{(Inkscape) Transparency is used (non-zero) for the text in Inkscape, but the package 'transparent.sty' is not loaded}%
    \renewcommand\transparent[1]{}%
  }%
  \providecommand\rotatebox[2]{#2}%
  \newcommand*\fsize{\dimexpr\f@size pt\relax}%
  \newcommand*\lineheight[1]{\fontsize{\fsize}{#1\fsize}\selectfont}%
  \ifx\svgwidth\undefined%
    \setlength{\unitlength}{335.09501059bp}%
    \ifx\svgscale\undefined%
      \relax%
    \else%
      \setlength{\unitlength}{\unitlength * \real{\svgscale}}%
    \fi%
  \else%
    \setlength{\unitlength}{\svgwidth}%
  \fi%
  \global\let\svgwidth\undefined%
  \global\let\svgscale\undefined%
  \makeatother%
  \begin{picture}(1,0.2634086)%
    \lineheight{1}%
    \setlength\tabcolsep{0pt}%
    \put(0,0){\includegraphics[width=\unitlength,page=1]{GeneralTwoCiliaTransported.pdf}}%
    \put(0.09421537,0.23777622){\color[rgb]{0,0,0}\makebox(0,0)[lt]{\lineheight{1.25}\smash{\begin{tabular}[t]{l}$a$\end{tabular}}}}%
    \put(0.35720043,0.2385756){\color[rgb]{0,0,0}\makebox(0,0)[lt]{\lineheight{1.25}\smash{\begin{tabular}[t]{l}$c$\end{tabular}}}}%
    \put(0.67454113,0.23777622){\color[rgb]{0,0,0}\makebox(0,0)[lt]{\lineheight{1.25}\smash{\begin{tabular}[t]{l}$S(\lambda_{(1)})a$\end{tabular}}}}%
    \put(0.95559231,0.23777622){\color[rgb]{0,0,0}\makebox(0,0)[lt]{\lineheight{1.25}\smash{\begin{tabular}[t]{l}$c\lambda_{(2)}$\end{tabular}}}}%
    \put(0.17894616,0.12266998){\color[rgb]{0,0,1}\makebox(0,0)[lt]{\lineheight{1.25}\smash{\begin{tabular}[t]{l}$\rho$\end{tabular}}}}%
    \put(0.16935402,-0.01){\color[rgb]{0,0,0}\makebox(0,0)[lt]{\lineheight{1.25}\smash{\begin{tabular}[t]{l}$b$\end{tabular}}}}%
    \put(0.71064057,-0.01){\color[rgb]{0,0,0}\makebox(0,0)[lt]{\lineheight{1.25}\smash{\begin{tabular}[t]{l}$\int(b)$\end{tabular}}}}%
    \put(0,0){\includegraphics[width=\unitlength,page=2]{GeneralTwoCiliaTransported.pdf}}%
    \put(0.44512867,0.16103855){\color[rgb]{0,0,0}\makebox(0,0)[lt]{\lineheight{1.25}\smash{\begin{tabular}[t]{l}$Hol_{\rho}^{\lambda \oo \int}$\end{tabular}}}}%
  \end{picture}%
\endgroup%

%% file: GeneralTwoCiliaTransported2.pdf_tex
\begingroup%
  \makeatletter%
  \providecommand\color[2][]{%
    \errmessage{(Inkscape) Color is used for the text in Inkscape, but the package 'color.sty' is not loaded}%
    \renewcommand\color[2][]{}%
  }%
  \providecommand\transparent[1]{%
    \errmessage{(Inkscape) Transparency is used (non-zero) for the text in Inkscape, but the package 'transparent.sty' is not loaded}%
    \renewcommand\transparent[1]{}%
  }%
  \providecommand\rotatebox[2]{#2}%
  \newcommand*\fsize{\dimexpr\f@size pt\relax}%
  \newcommand*\lineheight[1]{\fontsize{\fsize}{#1\fsize}\selectfont}%
  \ifx\svgwidth\undefined%
    \setlength{\unitlength}{335.09501059bp}%
    \ifx\svgscale\undefined%
      \relax%
    \else%
      \setlength{\unitlength}{\unitlength * \real{\svgscale}}%
    \fi%
  \else%
    \setlength{\unitlength}{\svgwidth}%
  \fi%
  \global\let\svgwidth\undefined%
  \global\let\svgscale\undefined%
  \makeatother%
  \begin{picture}(1,0.2634086)%
    \lineheight{1}%
    \setlength\tabcolsep{0pt}%
    \put(0,0){\includegraphics[width=\unitlength,page=1]{GeneralTwoCiliaTransported2.pdf}}%
    \put(0.09421537,0.23777622){\color[rgb]{0,0,0}\makebox(0,0)[lt]{\lineheight{1.25}\smash{\begin{tabular}[t]{l}$S(\lambda_{(1)})a$\end{tabular}}}}%
    \put(0.37958216,0.2385756){\color[rgb]{0,0,0}\makebox(0,0)[lt]{\lineheight{1.25}\smash{\begin{tabular}[t]{l}$\int(\lambda_{(2)}c)$\end{tabular}}}}%
    \put(0.67454113,0.23777622){\color[rgb]{0,0,0}\makebox(0,0)[lt]{\lineheight{1.25}\smash{\begin{tabular}[t]{l}$ca$\end{tabular}}}}%
    \put(0.95559231,0.23777622){\color[rgb]{0,0,0}\makebox(0,0)[lt]{\lineheight{1.25}\smash{\begin{tabular}[t]{l}$1$\end{tabular}}}}%
    \put(0.16935402,-0.00522553){\color[rgb]{0,0,0}\makebox(0,0)[lt]{\lineheight{1.25}\smash{\begin{tabular}[t]{l}$\int(b)\lambda'$\end{tabular}}}}%
    \put(0.75735364,-0.00522553){\color[rgb]{0,0,0}\makebox(0,0)[lt]{\lineheight{1.25}\smash{\begin{tabular}[t]{l}$\int(b)\lambda'$\end{tabular}}}}%
    \put(0.44960502,0.11179945){\color[rgb]{0,0,0}\makebox(0,0)[lt]{\lineheight{1.25}\smash{\begin{tabular}[t]{l}$=$\end{tabular}}}}%
  \end{picture}%
\endgroup%

%% file: BulkAssociativity.pdf_tex
\begingroup%
  \makeatletter%
  \providecommand\color[2][]{%
    \errmessage{(Inkscape) Color is used for the text in Inkscape, but the package 'color.sty' is not loaded}%
    \renewcommand\color[2][]{}%
  }%
  \providecommand\transparent[1]{%
    \errmessage{(Inkscape) Transparency is used (non-zero) for the text in Inkscape, but the package 'transparent.sty' is not loaded}%
    \renewcommand\transparent[1]{}%
  }%
  \providecommand\rotatebox[2]{#2}%
  \newcommand*\fsize{\dimexpr\f@size pt\relax}%
  \newcommand*\lineheight[1]{\fontsize{\fsize}{#1\fsize}\selectfont}%
  \ifx\svgwidth\undefined%
    \setlength{\unitlength}{1475.15901821bp}%
    \ifx\svgscale\undefined%
      \relax%
    \else%
      \setlength{\unitlength}{\unitlength * \real{\svgscale}}%
    \fi%
  \else%
    \setlength{\unitlength}{\svgwidth}%
  \fi%
  \global\let\svgwidth\undefined%
  \global\let\svgscale\undefined%
  \makeatother%
  \begin{picture}(1,0.55730706)%
    \lineheight{1}%
    \setlength\tabcolsep{0pt}%
    \put(0,0){\includegraphics[width=\unitlength,page=1]{BulkAssociativity.pdf}}%
  \end{picture}%
\endgroup%

%% file: BoundaryAssociativity.pdf_tex
\begingroup%
  \makeatletter%
  \providecommand\color[2][]{%
    \errmessage{(Inkscape) Color is used for the text in Inkscape, but the package 'color.sty' is not loaded}%
    \renewcommand\color[2][]{}%
  }%
  \providecommand\transparent[1]{%
    \errmessage{(Inkscape) Transparency is used (non-zero) for the text in Inkscape, but the package 'transparent.sty' is not loaded}%
    \renewcommand\transparent[1]{}%
  }%
  \providecommand\rotatebox[2]{#2}%
  \newcommand*\fsize{\dimexpr\f@size pt\relax}%
  \newcommand*\lineheight[1]{\fontsize{\fsize}{#1\fsize}\selectfont}%
  \ifx\svgwidth\undefined%
    \setlength{\unitlength}{1499.12564976bp}%
    \ifx\svgscale\undefined%
      \relax%
    \else%
      \setlength{\unitlength}{\unitlength * \real{\svgscale}}%
    \fi%
  \else%
    \setlength{\unitlength}{\svgwidth}%
  \fi%
  \global\let\svgwidth\undefined%
  \global\let\svgscale\undefined%
  \makeatother%
  \begin{picture}(1,0.54839735)%
    \lineheight{1}%
    \setlength\tabcolsep{0pt}%
    \put(0,0){\includegraphics[width=\unitlength,page=1]{BoundaryAssociativity.pdf}}%
  \end{picture}%
\endgroup%